\documentclass[10pt]{article}

\usepackage[utf8]{inputenc}
\usepackage[T1]{fontenc}
\usepackage{amsfonts}
\usepackage{amsmath}
\usepackage{amsthm}
\usepackage{amssymb}
\usepackage[english]{babel}
\usepackage{caption}
\usepackage{color}
\usepackage{enumerate}
\usepackage{etoolbox}
\usepackage[a4paper]{geometry}
\usepackage{graphicx}
\usepackage{mathabx}
\usepackage{mathtools}
\usepackage{multicol}
\usepackage{multirow}
\usepackage{pgfplots}
\usepackage{subcaption}
\usepackage{tikz}
\usetikzlibrary{arrows,backgrounds,calc,patterns,plotmarks}
\usepackage{url}
\usepackage{wasysym}
\usepackage{xcolor}

\makeatletter
\patchcmd{\maketitle}{\@fnsymbol}{\@arabic}{}{}  
\makeatother

\title{Curve Diagrams, Laminations, and the Geometric Complexity of Braids}
\author{Vincent Jugé
\thanks{Mines ParisTech, 60 boulevard Saint-Michel, 75272 Paris Cedex 06, France}
\thanks{Université Paris Diderot, Sorbonne Paris Cité, LIAFA, UMR 7089 CNRS, F-75205 Paris, France.}}
\date{\today}

\definecolor{verypalegray}{RGB}{242,242,242}
\definecolor{palegray}{RGB}{220,220,220}
\definecolor{gray}{RGB}{195,195,195}
\definecolor{darkgray}{RGB}{170,170,170}
\definecolor{blackgray}{RGB}{145,145,145}

\newcommand{\NN}{\mathbb{N}}
\newcommand{\ZZ}{\mathbb{Z}}

\newcommand{\RR}{\mathbb{R}}
\newcommand{\CC}{\mathbb{C}}
\newcommand{\conc}[2]{\mathbf{L}_{#1}^{#2}}

\newcommand{\ba}{\mathbf{a}}

\newcommand{\bs}{\mathbf{s}}
\newcommand{\bt}{\mathbf{t}}
\newcommand{\bu}{\mathbf{u}}

\newcommand{\bC}{\mathbf{C}}
\newcommand{\bD}{\mathbf{D}}

\newcommand{\bL}{\mathbf{L}}

\newcommand{\bS}{\mathbf{S}}
\newcommand{\bT}{\mathbf{T}}

\newcommand{\calB}{\mathcal{B}}
\newcommand{\fraB}{\mathfrak{B}}
\newcommand{\calC}{\mathcal{C}}

\newcommand{\calD}{\mathcal{D}}

\newcommand{\calG}{\mathcal{G}}
\newcommand{\calL}{\mathcal{L}}

\newcommand{\calN}{\mathcal{N}}
\newcommand{\calO}{\mathcal{O}}
\newcommand{\calP}{\mathcal{P}}
\newcommand{\calQ}{\mathcal{Q}}

\newcommand{\calZ}{\mathcal{Z}}

\newcommand{\puncture}[1]{
\draw[fill=white,draw=black,thick] (#1,0) circle (1.25);}
\newcommand{\PUNCTURE}[1]{
\draw[fill=black,draw=black,thick] (#1,0) circle (1.25);}
\newcommand{\puncturee}[2]{
\draw[fill=white,draw=black,thick] (#1,0) circle (#2);}
\newcommand{\PUNCTUREE}[2]{
\draw[fill=black,draw=black,thick] (#1,0) circle (#2);}

\newtheorem{theorem}{Theorem}[section]
\newtheorem{lemma}[theorem]{Lemma}
\newtheorem{corollary}[theorem]{Corollary}
\newtheorem{proposition}[theorem]{Proposition}
\newtheorem{definition}[theorem]{Definition}
\newtheorem{example}[theorem]{Example}
\newtheorem{conjecture}[theorem]{Conjecture}

\newenvironment{thm}{\begin{theorem}~\\}{\end{theorem}}
\newenvironment{thm*}[1]{\begin{theorem}[#1]~\\}{\end{theorem}}
\newenvironment{lem}{\begin{lemma}~\\}{\end{lemma}}
\newenvironment{lem*}[1]{\begin{lemma}[#1]~\\}{\end{lemma}}
\newenvironment{cor}{\begin{corollary}~\\}{\end{corollary}}
\newenvironment{cor*}[1]{\begin{corollary}[#1]~\\}{\end{corollary}}
\newenvironment{pro}{\begin{proposition}~\\}{\end{proposition}}
\newenvironment{pro*}[1]{\begin{proposition}[#1]~\\}{\end{proposition}}
\newenvironment{dfn}[1]{\begin{definition}[#1]~\\}{\end{definition}}

\newenvironment{xam*}[1]{\begin{example}[#1]\rm~\\}{\end{example}}
\newenvironment{conj}{\begin{conjecture}\rm~\\}{\end{conjecture}}
\newenvironment{conj*}[1]{\begin{conjecture}[#1]\rm~\\}{\end{conjecture}}

\begin{document}

\maketitle

\begin{abstract}
Braids can be represented geometrically as curve diagrams.
The geometric complexity of a braid is the minimal complexity of a curve diagram representing it.
We introduce and study the corresponding notion of geometric generating function.
We compute explicitly the geometric generating function for the group of braids on three strands and prove that it is
neither rational nor algebraic, nor even holonomic.
This result may appear as counterintuitive.
Indeed, the standard complexity (due to the Artin presentation of braid groups) is
algorithmically harder to compute than the geometric complexity,
yet the associated generating function for the group of braids on three strands is rational.
\end{abstract}

\section{Introduction}
\label{section:introduction}

Braid groups can be approached from various points of view,
including algebraic and geometric ones.

The algebraic point of view is based on finite presentations
of the group of braids, involving several possible generating families,
the most famous being the families of Artin generators~\cite{Artin-1926} and Garside generators~\cite{garside1969braid}.
Artin generators are usually considered as the most ``natural'' generators,
but Garside generators have proved to be more tractable in
answering several algorithmic and combinatorial questions.

Given a finite presentation, a braid is identified with a set of words.
Each word has a complexity, which is its length, and the complexity
of the braid is the minimal complexity of the words that represent it.
Then, the generating function, or growth series, of the group $B_n$ of $n$-strand braids,
is defined by $\fraB_n(z) = \sum_{\beta \in B_n} z^{|\beta|}$, where
$|\beta|$ is the complexity of the braid $\beta$.

Both the complexity of the braid and the corresponding generating function depend on the generators.
A natural question is to compute the generating function.
Each group $B_n$ has a rational generating function for Garside generators~\cite{Charney1995}.
It is also known that the group $B_3$ has a rational generating function for Artin generators~\cite{mairesse2006growth,2003sabalka},
but no such result is known to hold for the groups $B_n$ with $n \geq 4$.

The above-mentioned results are obtained by using clever normal forms.
A normal form consists in selecting a representative word for each braid.
Several questions appear immediately:
does there exist computable normal forms?
regular normal forms?
regular and geodesic normal forms?
Whereas the answers to the first two questions can be shown to be intrinsic
i.e. independent of the generating family,
the answer to the third question is specific to each family of generators~\cite{stoll}.
When a regular and geodesic normal form exists, the generating function is rational and
computing it is straightforward when considering an automaton recognising the normal form~\cite{Epstein:1992:WPG:573874}.

For the braid groups, the symmetric Garside normal form is a well-known regular geodesic normal form for the Garside generators~\cite{garside1969braid}.
On the contrary, the existence of a regular geodesic normal form for the Artin generators is a famous open question.

At first glance, the geometric approach seems quite different from its algebraic counterpart:
braids are no longer considered as sets of words but as sets of drawings.
In the geometric world, each drawing has a complexity (e.g. the number of intersections it has with some fixed set of curves),
and the complexity of a braid is the minimal complexity of the drawings that represent it~\cite{dynnikov:hal-00001267}.
Like in the algebraic case, the group $B_n$ has a ``geometric generating function'', or geometric growth series, defined by
$\calB_n(z) = \sum_{\beta \in B_n} z^{\|\beta\|}$, where $\|\beta\|$ is the geometric complexity of the braid $\beta$.

The general goal of this paper is to study the notion of geometric generating function, which, to the best of our knowledge,
has not been explored yet in the literature.
The first step is to identify a ``geodesic normal form'',
i.e. a set $S$ of drawings such that each braid is represented by one unique drawing in $S$,
and such that this drawing has a minimal geometric complexity.
Here, the relevant geodesic normal form will be related to the notion of
tight curve diagram (see Fig.~\ref{fig:intro:lamination-curve}), which was already studied in~\cite{dynnikov:hal-00001267}.

\begin{figure}[!ht]
\begin{center}
\begin{tikzpicture}[scale=0.16]
\draw[fill=palegray,draw=palegray] (20,0) circle (10);

\draw[draw=black,ultra thick] (12,0) -- (12,-11);
\draw[draw=black,ultra thick] (22,0) -- (22,11);
\draw[draw=black,ultra thick] (23,0) -- (23,11);
\draw[draw=black,ultra thick] (28,0) -- (28,-11);

\draw[draw=black,ultra thick] (12,0) arc (180:0:3);
\draw[draw=black,ultra thick] (13,0) arc (180:0:2);
\draw[draw=black,ultra thick] (24,0) arc (180:0:2);

\draw[draw=black,ultra thick] (13,0) arc (180:360:5.5);
\draw[draw=black,ultra thick] (17,0) arc (180:360:3);
\draw[draw=black,ultra thick] (18,0) arc (180:360:2);

\draw[draw=black,ultra thick] (10,0) -- (30,0);

\draw[draw=white,thick] (10,0) -- (30,0);

\PUNCTURE{10}
\PUNCTURE{30}

\puncture{15}
\puncture{20}
\puncture{26}
\end{tikzpicture}
{\tiny~}
\begin{tikzpicture}[scale=0.19555]
\draw[fill=palegray,draw=palegray] (14,0) circle (8);

\draw[draw=black,ultra thick] (12,-9) -- (12,9);
\draw[draw=black,ultra thick] (16,-9) -- (16,9);

\draw[draw=black,ultra thick] (6,0) arc (180:0:4);
\draw[draw=black,ultra thick] (14,0) arc (0:180:3);
\draw[draw=black,ultra thick] (8,0) arc (180:360:5);
\draw[draw=black,ultra thick] (18,0) arc (360:180:4);
\draw[draw=black,ultra thick] (10,0) arc (180:360:3);
\draw[draw=black,ultra thick] (16,0) arc (180:0:3);

\draw[draw=white,thick] (6,0) arc (180:0:4);
\draw[draw=white,thick] (14,0) arc (0:180:3);
\draw[draw=white,thick] (8,0) arc (180:360:5);
\draw[draw=white,thick] (18,0) arc (360:180:4);
\draw[draw=white,thick] (10,0) arc (180:360:3);
\draw[draw=white,thick] (16,0) arc (180:0:3);

\PUNCTUREE{6}{1.0227}
\PUNCTUREE{22}{1.0227}

\puncturee{10}{1.0227}
\puncturee{14}{1.0227}
\puncturee{18}{1.0227}
\end{tikzpicture}
\end{center}
\caption{Tight lamination (left) and tight curve diagram (right) representing the same braid}
\label{fig:intro:lamination-curve}
\end{figure}
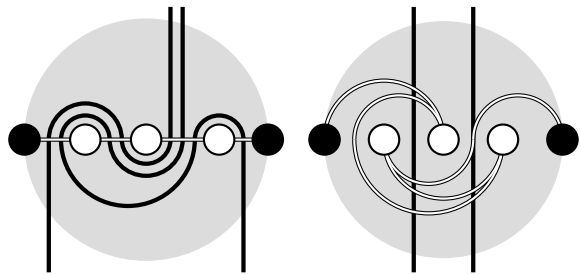

The geometric complexity is arguably easier to compute than the length complexity.
Indeed, consider a fixed braid group $B_n$ and the Artin generators.
On the one hand, computing the length complexity of a braid represented
by a word of length $k$ requires up to $2^{\calO(k)}$ operations:
one has to check braid equality for all the words of length less than $k$.
On the other hand, there exist algorithms
that compute in $\calO(k)$ operations the geometric complexity of a braid represented by a drawing with complexity $k$~\cite{vincent-relax-droite}.

The main result of this paper goes in the opposite direction:
we show that, even in the simple case of the group $B_3$, the geometric generating function $\calB_3(z)$ is not rational
(see Theorem~\ref{thm:G-and-L3}).
This is in sharp contrast with the length-based generating functions of $B_3$ for Artin and Garside generators,
which are rational, as recalled above.
A precise statement of our result on $B_3$ is the following one.
{
 \renewcommand{\thetheorem}{\ref{thm:G-and-L3}}
 \begin{thm}
 Let $\calB_3(z) = \sum_{\beta \in B_3} z^{\|\beta\|}$ be the ``geometric generating function'' associated
 to the geometric norm on braids (see~\cite{dynnikov:hal-00001267}).
 We have 
\[\calB_3(z) = 2 \frac{1+2z^2-z^4}{z^2(1-z^4)} \left(\sum_{n \geq 3} \varphi(n) z^{2n}\right) + \frac{z^2(1-3z^4)}{1-z^4},\]
where $\varphi$ denotes the Euler totient.
The function $\calB_3(z)$ is neither rational nor algebraic nor holonomic.
 \end{thm}
 \addtocounter{theorem}{-1}
}
In addition to this precise computation, we also estimate the series $\calB_n(z)$ with $n \geq 4$.

To obtain these results, we use several tools that also do have an intrinsic interest.
First, we use the connection between the complexities associated to two dual geometric representations of the braids,
the \emph{braid laminations} and the \emph{curve diagrams}, and in particular the fact that both complexities
yield the same generating functions.
Second, we design a system of integer-valued coordinates that capture directly
the geometric representations of braids, and from which computing the geometric complexity of braids is straightforward.
This system of coordinates is analogous to that of Dynnikov~\cite{Dyn02},
and the algorithms for computing both kinds of coordinates have similar flavours.
Thus, they have comparable (low) complexities.
However, our new system has specific advantages, since it is 
better suited for computing braid complexities,
and therefore for computing geometric generating functions.

\section{Braids, Integral Laminations and Curve Diagrams}
\label{section:braids-and-integral-laminations}

In Sections~\ref{section:braids-and-integral-laminations} and~\ref{section:from-laminations-to-graphs},
we mention standard definitions and theorems about
braids, \emph{integral laminations} and to \emph{curve diagrams},
and the notion of pulling tight a curve with respect to a set of punctures and to another curve.
These definitions come from algebraic topology
as well as from discrete group theory.
This material can be found in standard literature, e.g. in~\cite{birman-braids-links-mcg,Dehornoy_whyare,dynnikov:hal-00001267,farb2011primer,Fenn_orderingthe}.
However, we prefer to recall such material here, in order to use it subsequently.

\subsection{Braids}
\label{subsection:braids}

The group of braids on $n$ strands was originally introduced by Artin~\cite{Artin-1926},
who came with the following algebraic description.

\begin{dfn}{Braid group}
The group of braids on $n$ strands is the group
\[B_n = \left\langle \sigma_1, \dots, \sigma_{n-1} \mid
\sigma_i \sigma_{i+1} \sigma_i = \sigma_{i+1} \sigma_i \sigma_{i+1},
\sigma_i \sigma_j = \sigma_j \sigma_i \text{ if } |i-j| \geq 2 \right\rangle.\]
\end{dfn}

This finite presentation of the group of braids
comes along with the representation of braids
as an isotopy class of \emph{braid diagrams}, as illustrated in Figure~\ref{fig:strands}.

\begin{figure}[!ht]
\begin{center}
\begin{tikzpicture}[scale=0.20]
\draw[draw=black,thick] (8,0) -- (22,0);
\draw[draw=black,thick] (8,2) -- (22,2);
\draw[draw=black,thick] (8,5) -- (22,5);
\draw[draw=black,thick] (8,7) -- (13,7) -- (17,9) -- (22,9);
\draw[draw=black,thick] (8,9) -- (13,9) -- (14.75,8.125);
\draw[draw=black,thick] (15.25,7.875) -- (17,7) -- (22,7);
\draw[draw=black,thick] (8,11) -- (22,11);
\draw[draw=black,thick] (8,14) -- (22,14);
\draw[draw=black,thick] (8,16) -- (22,16);
\node[anchor=east] at (8,0) {$1$};
\node[anchor=east] at (8,2) {$2$};
\node[anchor=east] at (8,5) {$i-1$};
\node[anchor=east] at (8,7) {$i$};
\node[anchor=east] at (8,9) {$i+1$};
\node[anchor=east] at (8,11) {$i+2$};
\node[anchor=east] at (8,14) {$n-1$};
\node[anchor=east] at (8,16) {$n$};
\node at (15,3.9) {$\vdots$};
\node at (15,12.9) {$\vdots$};
\end{tikzpicture}
\end{center}
\caption{Braid diagram of the generator $\sigma_i$ ($1 \leq i \leq n-1$)}
\label{fig:strands}
\end{figure}
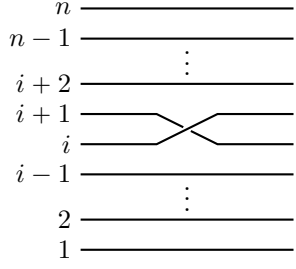

However, in this paper we focus on another, equivalent, approach
of the group of braids.
Indeed, the group of braids on $n$ strands can also be defined
as the mapping class group (also called modular group)
of the unit disk with $n$ punctures.
Let us make this statement more precise.

Let $D^2 \subseteq \CC$ be the closed unit disk,
let $\partial D^2$ be the unit circle (i.e. the boundary of $D^2$),
and let $P_n \subseteq (-1,1)$ be a set of size $n$.
We will refer below to the elements of $P_n$ as being \emph{punctures}
in the disk $D^2$, and number them from left to right: $P_n = \{p_i : 1 \leq i \leq n\}$,
with $p_1 < \ldots < p_n$.
We also call \emph{left point} the point $-1$, which we will also denote by $p_0$;
and call \emph{right point} the point $+1$, which we will also denote by $p_{n+1}$.

Then, let $H_n$ be the group of orientation-preserving homeomorphisms $h : \CC \to \CC$
such that $h(P_n) = P_n$, $h(\partial D^2) = \partial D^2$ and $h(1) = 1$, $h(-1) = -1$,
i.e. the homeomorphisms fixing $\partial D^2$ and $P_n$ setwise, and $\pm 1$ pointwise.

\begin{thm*}{see~\cite{birman-braids-links-mcg}}
The group $B_n$ of braids on $n$ strands is isomorphic to the mapping class group of the punctured disk
$D^2 \setminus P_n$, i.e. isomorphic to the quotient group of $H_n$ by the isotopy relation.
\label{thm:braids}
\end{thm*}

It is remarkable that this definition does \emph{not}
depend on which set $P_n$ of punctures we chose.
In addition, each braid appears as a class of homeomorphisms of the unit disk $D^2$,
which conveys the idea of giving a graphical representation of the braid.

For flexibility reasons, we introduce here
a slightly different characterization of
the group of braids, analogous to that of Theorem~\ref{thm:braids}.
Here, instead of considering the group $H_n$,
we denote by $H^\ast_n$ the group of orientation-preserving
homeomorphisms $h : \CC \to \CC$
such that $h(P_n) = P_n$, $h(-1) = -1$ and
$h(1) = 1$,
i.e. fixing $\{-1\}$, $\{1\}$ and $P_n$ setwise.

\begin{thm}
The group $B_n$ of braids on $n$ strands is isomorphic to the quotient group
of $H_n^\ast$ by the isotopy relation.
\label{thm:braids-2}
\end{thm}

Theorem~\ref{thm:braids-2} identifies braids to isotopy classes of self-homeomorphisms of $\CC$.
More precisely, let $\bS$ be a subset of the complex plane and
let $\beta$ be a braid, and consider the isotopy class
$\beta(\bS) = \{h(\bS) : h$ is an homeomorphism that represents $\beta\}$.
The group of braids $B_n$ acts transitively on the set $\{\beta(\bS) : \beta \in B_n\}$,
which induces an equivalence relation on the group $B_n$ itself.

We focus below on two subsets of $\CC$:
We call respectively \emph{trivial lamination} and \emph{trivial curve diagram} the sets
\[\bL := \left\{\frac{1}{2}(p_j+p_{j+1})+i\RR : j \leq 1 \leq n-1 \right\} \text{ and }\bD := [-1,1].\]
If $\bS = \bL$ or $\bS = \bD$, then the action of $B_n$ on $\{\beta(\bS) : \beta \in B_n\}$
is free (see~\cite{birman-braids-links-mcg} for details).
This means that the sets $\beta(\bS)$ and $\gamma(\bS)$ are disjoint as soon as $\beta \neq \gamma$.
Hence, each set $h(\bS)$ belongs to the set $\beta(\bS)$ for one unique braid $\beta$;
we say that $h(\bS)$ \emph{represents} the braid $\beta$.

\subsection{Laminations}
\label{subsection:integral-laminations}

\begin{dfn}{Lamination}
Let us consider the set $P_n$ of $n$ punctures inside the disk $D^2$.
We call \emph{lamination}, and denote by $\calL$, the union of $n-1$ non-intersecting open curves $\calL_1,\ldots,\calL_{n-1}$ such that
each curve $\calL_j$
\begin{itemize}
\item contains two vertical half-lines with opposite directions (i.e. sets $z_j + i \RR_{\leq 0}$ and $z'_j + i \RR_{\geq 0}$);
\item splits the plane $\CC$ into one \emph{left} region that contains the left point and $j$ punctures, and one
\emph{right} region that contains the right point and $n-j$ punctures.
\end{itemize}
\label{dfn:trivial-lamination}
\end{dfn}

Figure~\ref{fig:laminations} represents two laminations, including $\bL$, the trivial one.
In all subsequent figures,
punctures are indicated by white dots,
and the left and right points are indicated by black dots;
the gray area represents the unit disk $D^2$,
the curves of the lamination are drawn in black,
and the segment $[-1,1]$ is drawn in white.

\begin{figure}[!ht]
\begin{center}
\begin{tikzpicture}[scale=0.16]
\draw[fill=palegray,draw=palegray] (16,0) circle (8);

\draw[draw=black,ultra thick] (14,-9.5) -- (14,9.5);
\draw[draw=black,ultra thick] (18,-9.5) -- (18,9.5);

\draw[draw=black,ultra thick] (8,0) -- (24,0);

\draw[draw=white,thick] (8,0) -- (24,0);

\PUNCTURE{8}
\PUNCTURE{24}

\puncture{12}
\puncture{16}
\puncture{20}

\node at (16,-12) {Trivial lamination};
\end{tikzpicture}
{\tiny~}
\begin{tikzpicture}[scale=0.16]
\draw[fill=palegray,draw=palegray] (16.5,0) circle (8.5);

\draw[draw=black,ultra thick] (14,-9.5) -- (14,9.5);
\draw[draw=black,ultra thick] (15,0) -- (15,9.5);
\draw[draw=black,ultra thick] (23,-9.5) -- (23,0);

\draw[draw=black,ultra thick] (19,0) arc (180:0:2);
\draw[draw=black,ultra thick] (15,0) arc (180:360:2);

\draw[draw=black,ultra thick] (8,0) -- (25,0);

\draw[draw=white,thick] (8,0) -- (25,0);

\PUNCTURE{8}
\PUNCTURE{25}

\puncture{12}
\puncture{17}
\puncture{21}

\node at (16.5,-12) {Non-trivial lamination};
\end{tikzpicture}
\end{center}
\caption{Laminations}
\label{fig:laminations}
\end{figure}
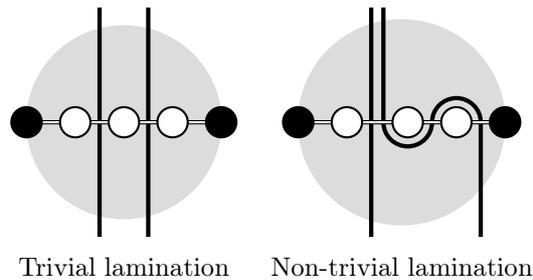

Then, following Dynnikov and Wiest~\cite{dynnikov:hal-00001267},
we define the norm of a lamination,
and the \emph{laminated norm} of a braid.

\begin{dfn}{Laminated norm and tight lamination}
Let $\beta \in B_n$ be a braid on $n$ strands,
and let $\calL$ be a lamination representing $\beta$.

The \emph{laminated norm} of $\calL$, which we denote by $\|\calL\|_\ell$, is
the cardinality of the set $\calL \cap [-1,1]$,
i.e. the number of intersection points between the segment $[-1,1]$
and the $n-1$ curves of the lamination $\calL$.

Moreover, if, among all the laminations that represent $\beta$,
the lamination $\calL$ has a minimal laminated norm, then we say that
$\calL$ is a \emph{tight} lamination.
In this case, we also define the
\emph{laminated norm} of the braid $\beta$, which we denote by $\|\beta\|_\ell$, as
the norm $\|\calL\|_\ell$.
\label{dfn:lamination-norm}
\end{dfn}

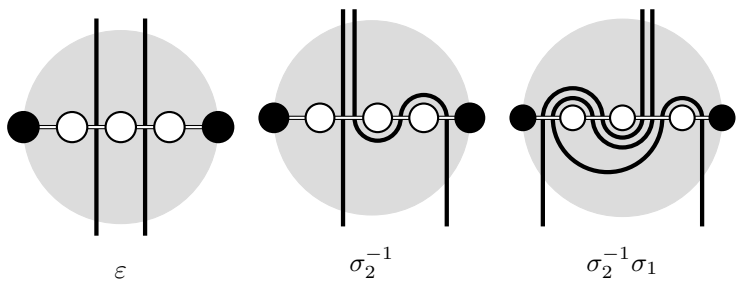
\begin{figure}[!ht]
\begin{center}
\begin{tikzpicture}[scale=0.16]
\draw[fill=palegray,draw=palegray] (16,0) circle (8);

\draw[draw=black,ultra thick] (14,-9) -- (14,9);
\draw[draw=black,ultra thick] (18,-9) -- (18,9);

\draw[draw=black,ultra thick] (8,0) -- (24,0);

\draw[draw=white,thick] (8,0) -- (24,0);

\PUNCTURE{8}
\PUNCTURE{24}

\puncture{12}
\puncture{16}
\puncture{20}

\node at (16,-12) {$\varepsilon$};
\end{tikzpicture}
{\tiny~}
\begin{tikzpicture}[scale=0.151578]
\draw[fill=palegray,draw=palegray] (16.5,0) circle (8.5);

\draw[draw=black,ultra thick] (14,-9.5) -- (14,9.5);
\draw[draw=black,ultra thick] (15,0) -- (15,9.5);
\draw[draw=black,ultra thick] (23,-9.5) -- (23,0);

\draw[draw=black,ultra thick] (19,0) arc (180:0:2);
\draw[draw=black,ultra thick] (15,0) arc (180:360:2);

\draw[draw=black,ultra thick] (8,0) -- (25,0);

\draw[draw=white,thick] (8,0) -- (25,0);

\PUNCTURE{8}
\PUNCTURE{25}

\puncture{12}
\puncture{17}
\puncture{21}

\node at (16.5,-12.667) {$\sigma_2^{-1}$};
\end{tikzpicture}
{\tiny~}
\begin{tikzpicture}[scale=0.130909]
\draw[fill=palegray,draw=palegray] (20,0) circle (10);

\draw[draw=black,ultra thick] (12,0) -- (12,-11);
\draw[draw=black,ultra thick] (22,0) -- (22,11);
\draw[draw=black,ultra thick] (23,0) -- (23,11);
\draw[draw=black,ultra thick] (28,0) -- (28,-11);

\draw[draw=black,ultra thick] (12,0) arc (180:0:3);
\draw[draw=black,ultra thick] (13,0) arc (180:0:2);
\draw[draw=black,ultra thick] (24,0) arc (180:0:2);

\draw[draw=black,ultra thick] (13,0) arc (180:360:5.5);
\draw[draw=black,ultra thick] (17,0) arc (180:360:3);
\draw[draw=black,ultra thick] (18,0) arc (180:360:2);

\draw[draw=black,ultra thick] (10,0) -- (30,0);

\draw[draw=white,thick] (10,0) -- (30,0);

\PUNCTURE{10}
\PUNCTURE{30}

\puncture{15}
\puncture{20}
\puncture{26}

\node at (20,-14.667) {$\sigma_2^{-1} \sigma_1$};
\end{tikzpicture}
\end{center}
\caption{Identifying braids to tight laminations}
\label{fig:3-laminations}
\end{figure}

Note that, although we call the mapping $\beta \mapsto \|\beta\|_\ell$ a \emph{norm},
following the seminal paper of Dynnikov and Wiest~\cite{dynnikov:hal-00001267},
this mapping does not satisfy standard properties of norms on metric spaces, such as
separation axioms (i.e. that $\beta = \varepsilon$ iff $\|\beta\|_\ell = 0$) or
sub-additivity axioms (i.e. that $\|\beta \cdot \gamma\|_\ell \leq \|\beta\|_\ell + \|\gamma\|_\ell$ for all $\beta, \gamma \in B_n$).
Counterexamples to those properties are provided by the fact that $\|\varepsilon\|_\ell = n-1$ and that
$\|(\sigma_1 \sigma_2)^k\|_\ell = 2(F_{2k+3}-1)$, where $F_k$ denotes the $k$-th Fibonacci number.

However, Dynnikov and Wiest prove in~\cite{dynnikov:hal-00001267} that
the mapping $\beta \mapsto \log \|\beta\|_\ell$ is \emph{comparable} to a norm, i.e.
that there exists positive constants $m_n$ and $M_n$ and a norm $\calN$ of $B_n$ such that
$m_n (\calN(\beta)-1) \leq \log \|\beta\|_\ell \leq M_n(\calN(\beta)+1)$ for all $\beta \in B_n$.

Finally, observe that our notions of lamination and of laminated norm follow the ones used in~\cite{dynnikov:hal-00001267} but are
slightly different from the ones defined in previous work (see~\cite{Dehornoy_whyare,Fenn_orderingthe}), which we call \emph{closed laminations}.

A \emph{closed lamination} is the union of $n$ non-intersecting closed curves $\calL_1,\ldots,\calL_n$ such that
each curve $\calL_j$ splits the plane $\CC$ into one \emph{inner} region that contains the left point and $j$ punctures, and one
\emph{outer} region that contains the right point and $n-j$ punctures.
Informally, closed laminations may be easily obtained from (non-closed) laminations as follows:
for each integer $i \leq n-1$, bend both half-line of the curve $\calL_i$ to the left,
in order to transform $\calL_i$ into a closed line;
then, add a circle $\calL_n$ that will enclose all punctures and all lines $\calL_i$ but not the right point.

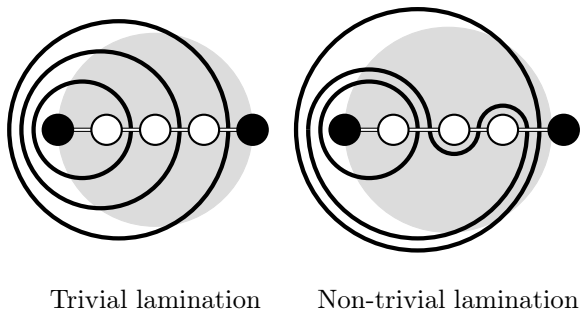
\begin{figure}[!ht]
\begin{center}
\begin{tikzpicture}[scale=0.16]
\draw[fill=palegray,draw=palegray] (16,0) circle (8);

\draw[draw=black,ultra thick] (14,0) arc (0:360:4);
\draw[draw=black,ultra thick] (18,0) arc (0:360:6.5);
\draw[draw=black,ultra thick] (22,0) arc (0:360:9);

\draw[draw=black,ultra thick] (8,0) -- (24,0);

\draw[draw=white,thick] (8,0) -- (24,0);

\PUNCTURE{8}
\PUNCTURE{24}

\puncture{12}
\puncture{16}
\puncture{20}

\node at (16,-14) {Trivial lamination};
\node at (16,10) {};
\end{tikzpicture}
{\tiny~}
\begin{tikzpicture}[scale=0.16]
\draw[fill=palegray,draw=palegray] (16.5,0) circle (8.5);

\draw[draw=black,ultra thick] (14,0) arc (0:360:4);
\draw[draw=black,ultra thick] (15,0) arc (0:180:5);
\draw[draw=black,ultra thick] (23,0) arc (360:180:9);
\draw[draw=black,ultra thick] (24,0) arc (360:0:10);

\draw[draw=black,ultra thick] (19,0) arc (180:0:2);
\draw[draw=black,ultra thick] (15,0) arc (180:360:2);

\draw[draw=black,ultra thick] (8,0) -- (26,0);

\draw[draw=white,thick] (8,0) -- (26,0);

\PUNCTURE{8}
\PUNCTURE{26}

\puncture{12}
\puncture{17}
\puncture{21}

\node at (16.5,-14) {Non-trivial lamination};
\end{tikzpicture}
\end{center}
\caption{Closed laminations}
\label{fig:old-laminations}
\end{figure}

Figure~\ref{fig:old-laminations} presents the closed laminations
corresponding to the laminations displayed in Figure~\ref{fig:laminations}.
From now on, we exclusively use laminations.

\subsection{Curve Diagrams}
\label{subsection:curve-diagrams}

As mentioned in Section~\ref{subsection:braids}, curve diagrams are an alternative to laminations.

\begin{dfn}{Curve diagram}
Let us consider the set $P_n$ of $n$ punctures inside the disk $D^2$.
We call \emph{curve diagram}, and denote by $\calD$, each non-intersecting open curve,
with endpoints $-1$ and $+1$,
that contains each puncture of the disk.
\label{dfn:trivial-diagram}
\end{dfn}

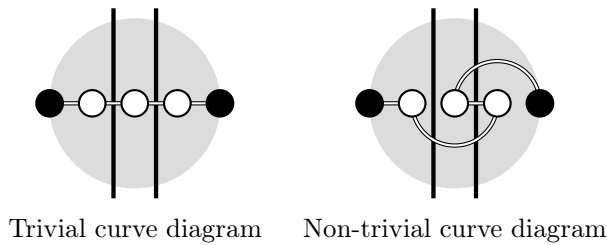
\begin{figure}[!ht]
\begin{center}
\begin{tikzpicture}[scale=0.14]
\draw[fill=palegray,draw=palegray] (16,0) circle (8);

\draw[draw=black,ultra thick] (14,-9) -- (14,9);
\draw[draw=black,ultra thick] (18,-9) -- (18,9);

\draw[draw=black,ultra thick] (8,0) -- (24,0);

\draw[draw=white,thick] (8,0) -- (24,0);

\PUNCTURE{8}
\PUNCTURE{24}

\puncture{12}
\puncture{16}
\puncture{20}

\node at (16,-12) {Trivial curve diagram};
\end{tikzpicture}
{\tiny~}
\begin{tikzpicture}[scale=0.14]
\draw[fill=palegray,draw=palegray] (16,0) circle (8);

\draw[draw=black,ultra thick] (14,-9) -- (14,9);
\draw[draw=black,ultra thick] (18,-9) -- (18,9);

\draw[draw=black,ultra thick] (8,0) -- (12,0);
\draw[draw=black,ultra thick] (12,0) arc (180:360:4);
\draw[draw=black,ultra thick] (20,0) -- (16,0);
\draw[draw=black,ultra thick] (16,0) arc (180:0:4);

\draw[draw=white,thick] (8,0) -- (12,0);
\draw[draw=white,thick] (12,0) arc (180:360:4);
\draw[draw=white,thick] (20,0) -- (16,0);
\draw[draw=white,thick] (16,0) arc (180:0:4);

\PUNCTURE{8}
\PUNCTURE{24}

\puncture{12}
\puncture{16}
\puncture{20}

\node at (16,-12) {Non-trivial curve diagram};
\end{tikzpicture}
\end{center}
\caption{Curve diagrams}
\label{fig:diagrams}
\end{figure}

Since both laminations and curve diagrams consist in
drawings on the complex plane, we represent them in analogous ways
(see Figure~\ref{fig:diagrams}).

We adapt the notion of laminated norm to curve diagrams,
and thereby define the norm of a curve diagram,
and the diagrammatic norm of a braid.

\begin{dfn}{Diagrammatic norm and tight curve diagram}
Let $\beta \in B_n$ be a braid on $n$ strands,
and let $\calD$ be a curve diagram representing $\beta$.
Recall that $\bL$ is the trivial lamination.

The \emph{diagrammatic norm} of $\calD$, which we denote by $\|\calD\|_d$, is
the cardinality of the set $\calD \cap \bL$,
i.e. the number of intersection points between the curve diagram $\calD$
and the $n-1$ vertical lines of the lamination $\bL$.

Moreover, if, among all the curve diagrams that represent $\beta$,
the curve diagram $\calD$ has a minimal diagrammatic norm, then we say that
$\calD$ is a \emph{tight} curve diagram.
In this case, we also define the
\emph{diagrammatic norm} of the braid $\beta$, which we denote by $\|\beta\|_d$, as
the norm $\|\calD\|_d$.
\label{dfn:diagram-norm}
\end{dfn}

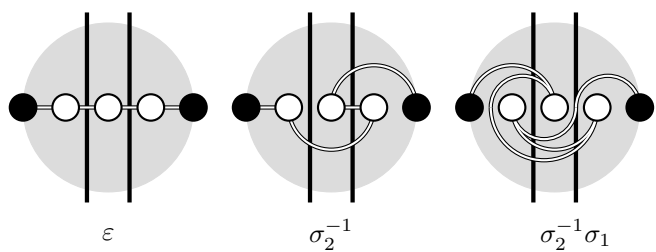
\begin{figure}[!ht]
\begin{center}
\begin{tikzpicture}[scale=0.14]
\draw[fill=palegray,draw=palegray] (16,0) circle (8);

\draw[draw=black,ultra thick] (14,-9) -- (14,9);
\draw[draw=black,ultra thick] (18,-9) -- (18,9);

\draw[draw=black,ultra thick] (8,0) -- (24,0);

\draw[draw=white,thick] (8,0) -- (24,0);

\PUNCTURE{8}
\PUNCTURE{24}

\puncture{12}
\puncture{16}
\puncture{20}

\node at (16,-12) {$\varepsilon$};
\node at (16,-14) {};
\end{tikzpicture}
{\tiny~}
\begin{tikzpicture}[scale=0.14]
\draw[fill=palegray,draw=palegray] (16,0) circle (8);

\draw[draw=black,ultra thick] (14,-9) -- (14,9);
\draw[draw=black,ultra thick] (18,-9) -- (18,9);

\draw[draw=black,ultra thick] (8,0) -- (12,0);
\draw[draw=black,ultra thick] (12,0) arc (180:360:4);
\draw[draw=black,ultra thick] (20,0) -- (16,0);
\draw[draw=black,ultra thick] (16,0) arc (180:0:4);

\draw[draw=white,thick] (8,0) -- (12,0);
\draw[draw=white,thick] (12,0) arc (180:360:4);
\draw[draw=white,thick] (20,0) -- (16,0);
\draw[draw=white,thick] (16,0) arc (180:0:4);

\PUNCTURE{8}
\PUNCTURE{24}

\puncture{12}
\puncture{16}
\puncture{20}

\node at (16,-12) {$\sigma_2^{-1}$};
\node at (16,-14) {};
\end{tikzpicture}
{\tiny~}
\begin{tikzpicture}[scale=0.14]
\draw[fill=palegray,draw=palegray] (14,0) circle (8);

\draw[draw=black,ultra thick] (12,-9) -- (12,9);
\draw[draw=black,ultra thick] (16,-9) -- (16,9);

\draw[draw=black,ultra thick] (6,0) arc (180:0:4);
\draw[draw=black,ultra thick] (14,0) arc (0:180:3);
\draw[draw=black,ultra thick] (8,0) arc (180:360:5);
\draw[draw=black,ultra thick] (18,0) arc (360:180:4);
\draw[draw=black,ultra thick] (10,0) arc (180:360:3);
\draw[draw=black,ultra thick] (16,0) arc (180:0:3);

\draw[draw=white,thick] (6,0) arc (180:0:4);
\draw[draw=white,thick] (14,0) arc (0:180:3);
\draw[draw=white,thick] (8,0) arc (180:360:5);
\draw[draw=white,thick] (18,0) arc (360:180:4);
\draw[draw=white,thick] (10,0) arc (180:360:3);
\draw[draw=white,thick] (16,0) arc (180:0:3);

\PUNCTURE{6}{0}
\PUNCTURE{22}{4}

\puncture{10}
\puncture{14}
\puncture{18}

\node at (16,-12) {$\sigma_2^{-1} \sigma_1$};
\node at (16,-14) {};
\end{tikzpicture}
\end{center}
\caption{Identifying braids to tight curve diagrams}
\label{fig:curves}
\end{figure}

\section{Norm-Preserving Transformations}
\label{section:from-laminations-to-graphs}

Counting intersections between the curves of a lamination $\beta(\bL)$ and
the trivial curve diagram $\bD$ was the basic idea that led to the
\emph{norm} introduced by Dynnikov and Wiest in~\cite{dynnikov:hal-00001267}.

A natural question the comparison
between the norm defined on laminations
and the norm defined on curve diagrams (Definitions~\ref{dfn:lamination-norm} and~\ref{dfn:diagram-norm}).
We demonstrate here a simple connection between laminated and diagrammatic norms.

\begin{pro}
Let $\beta$ be a braid on $n$ strands.
We have: $\|\beta\|_\ell = \|\beta^{-1}\|_d$, i.e. the laminated norm of the braid $\beta$ is equal to the
diagrammatic norm of the braid $\beta^{-1}$.
\label{pro:laminated-norm-diagrammatic}
\end{pro}

\begin{proof}
Let $\bL$ be the trivial lamination and let $h \in H_n^\ast$ be a representative of the braid $\beta$
such that $h(\bL)$ is a tight lamination.
Since the curve $h^{-1}(\bD)$ is a curve diagram of the braid $\beta^{-1}$,
it follows that \[\|\beta\|_\ell = |h(\bL) \cap \bD| = |h^{-1} (h(\bL) \cap \bD) | = |\bL \cap h^{-1}(\bD) | \geq \|\beta^{-1}\|_d.\]
One proves similarly that $\|\beta^{-1}\|_d \geq \|\beta\|_\ell$, which completes the proof.
\end{proof}

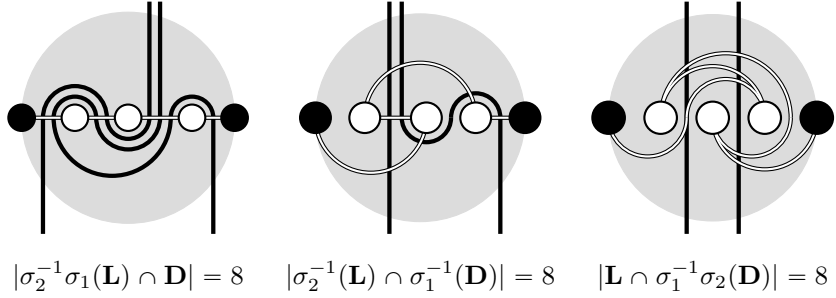
\begin{figure}[!ht]
\begin{center}
\begin{tikzpicture}[scale=0.14]
\draw[fill=palegray,draw=palegray] (20,0) circle (10);

\draw[draw=black,ultra thick] (12,0) -- (12,-11);
\draw[draw=black,ultra thick] (22,0) -- (22,11);
\draw[draw=black,ultra thick] (23,0) -- (23,11);
\draw[draw=black,ultra thick] (28,0) -- (28,-11);

\draw[draw=black,ultra thick] (12,0) arc (180:0:3);
\draw[draw=black,ultra thick] (13,0) arc (180:0:2);
\draw[draw=black,ultra thick] (24,0) arc (180:0:2);
\draw[draw=black,ultra thick] (13,0) arc (180:360:5.5);
\draw[draw=black,ultra thick] (17,0) arc (180:360:3);
\draw[draw=black,ultra thick] (18,0) arc (180:360:2);

\draw[draw=black,ultra thick] (10,0) -- (30,0);

\draw[draw=white,thick] (10,0) -- (30,0);

\PUNCTURE{10}
\PUNCTURE{30}

\puncture{15}
\puncture{20}
\puncture{26}

\node at (20,-15) {$|\sigma_2^{-1} \sigma_1(\bL) \cap \bD| = 8$};
\end{tikzpicture}
{\tiny~}
\begin{tikzpicture}[scale=0.1621]
\draw[fill=palegray,draw=palegray] (6.5,0) circle (8.5);

\draw[draw=black,ultra thick] (4,-9.5) -- (4,9.5);
\draw[draw=black,ultra thick] (5,0) -- (5,9.5);
\draw[draw=black,ultra thick] (5,0) arc (180:360:2);
\draw[draw=black,ultra thick] (9,0) arc (180:0:2);
\draw[draw=black,ultra thick] (13,0) -- (13,-9.5);

\draw[draw=black,ultra thick] (-2,0) arc (180:360:4.5);
\draw[draw=black,ultra thick] (2,0) -- (7,0);
\draw[draw=black,ultra thick] (2,0) arc (180:0:4.5);
\draw[draw=black,ultra thick] (11,0) -- (15,0);

\draw[draw=white,thick] (-2,0) arc (180:360:4.5);
\draw[draw=white,thick] (2,0) -- (7,0);
\draw[draw=white,thick] (2,0) arc (180:0:4.5);
\draw[draw=white,thick] (11,0) -- (15,0);

\PUNCTURE{-2}
\PUNCTURE{15}

\puncture{2}
\puncture{7}
\puncture{11}

\node at (6.5,-12.955) {$|\sigma_2^{-1}(\bL) \cap \sigma_1^{-1}(\bD)| = 8$};
\end{tikzpicture}
{\tiny~}
\begin{tikzpicture}[scale=0.1711]
\draw[fill=palegray,draw=palegray] (6,0) circle (8);

\draw[draw=black,ultra thick] (4,-9) -- (4,9);
\draw[draw=black,ultra thick] (8,-9) -- (8,9);

\draw[draw=black,ultra thick] (-2,0) arc (180:360:3);
\draw[draw=black,ultra thick] (4,0) arc (180:0:3);
\draw[draw=black,ultra thick] (10,0) arc (0:180:4);
\draw[draw=black,ultra thick] (2,0) arc (180:0:5);
\draw[draw=black,ultra thick] (12,0) arc (360:180:3);
\draw[draw=black,ultra thick] (6,0) arc (180:360:4);

\draw[draw=white,thick] (-2,0) arc (180:360:3);
\draw[draw=white,thick] (4,0) arc (180:0:3);
\draw[draw=white,thick] (10,0) arc (0:180:4);
\draw[draw=white,thick] (2,0) arc (180:0:5);
\draw[draw=white,thick] (12,0) arc (360:180:3);
\draw[draw=white,thick] (6,0) arc (180:360:4);

\PUNCTURE{-2}
\PUNCTURE{14}

\puncture{2}
\puncture{6}
\puncture{10}

\node at (6,-12.2735) {$|\bL \cap \sigma_1^{-1}\sigma_2(\bD)| = 8$};
\end{tikzpicture}
\end{center}
\caption{From $\|\sigma_2^{-1} \sigma_1\|_\ell$ to $\|\sigma_1^{-1} \sigma_2\|_d$}
\label{fig:one-norm-to-the-other}
\end{figure}

From Proposition~\ref{pro:laminated-norm-diagrammatic} follow directly Corollaries~\ref{cor:count-twice-size} and~\ref{cor:equal-series}.

\begin{cor}
Let $n$ and $k$ be positive integers.
We have: $|\{\beta \in B_n : \|\beta\|_\ell = k\}| = |\{\beta \in B_n : \|\beta\|_d = k\}|$, i.e.
the braids on $n$ strands with \emph{laminated norm} $k$ are as numerous as those with \emph{diagrammatic norm} $k$.
\label{cor:count-twice-size}
\end{cor}

\begin{cor}
For all positive integers $n$, the geometric generating functions
$\sum_{\beta \in B_n} z^{\|\beta\|_d}$ and
$\sum_{\beta \in B_n} z^{\|\beta\|_\ell}$ are equal.
\label{cor:equal-series}
\end{cor}

Geometrical symmetries induce some additional invariance properties of the laminated and diagrammatic norms.
Indeed, consider the group morphisms $\bS_v$, $\bS_h$ and $\bS_c$ such that
$\bS_v : \sigma_i \mapsto \sigma_{n-i}^{-1}$, $\bS_h : \sigma_i \mapsto \sigma_i^{-1}$ and
$\bS_c : \sigma_i \mapsto \sigma_{n-i}$.
Observe that $\bS_h \circ \bS_v = \bS_v \circ \bS_h = \bS_c$
is the morphism $\beta \mapsto \Delta^{-1} \cdot \beta \cdot \Delta$.

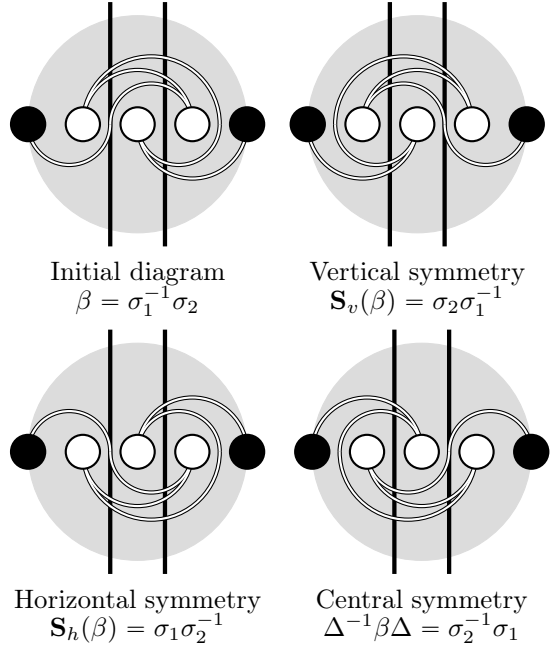
\begin{figure}[!ht]
\begin{center}
\begin{tikzpicture}[scale=0.18]
\draw[fill=palegray,draw=palegray] (6,0) circle (8);

\draw[draw=black,ultra thick] (4,-9) -- (4,9);
\draw[draw=black,ultra thick] (8,-9) -- (8,9);

\draw[draw=black,ultra thick] (-2,0) arc (180:360:3);
\draw[draw=black,ultra thick] (4,0) arc (180:0:3);
\draw[draw=black,ultra thick] (10,0) arc (0:180:4);
\draw[draw=black,ultra thick] (2,0) arc (180:0:5);
\draw[draw=black,ultra thick] (12,0) arc (360:180:3);
\draw[draw=black,ultra thick] (6,0) arc (180:360:4);

\draw[draw=white,thick] (-2,0) arc (180:360:3);
\draw[draw=white,thick] (4,0) arc (180:0:3);
\draw[draw=white,thick] (10,0) arc (0:180:4);
\draw[draw=white,thick] (2,0) arc (180:0:5);
\draw[draw=white,thick] (12,0) arc (360:180:3);
\draw[draw=white,thick] (6,0) arc (180:360:4);

\PUNCTURE{-2}
\PUNCTURE{14}

\puncture{2}
\puncture{6}
\puncture{10}

\node at (6,-11) {Initial diagram};
\node at (6,-13) {$\beta = \sigma_1^{-1} \sigma_2$};
\end{tikzpicture}
{\tiny~}
\begin{tikzpicture}[scale=0.18]
\draw[fill=palegray,draw=palegray] (-6,0) circle (8);

\draw[draw=black,ultra thick] (-4,-9) -- (-4,9);
\draw[draw=black,ultra thick] (-8,-9) -- (-8,9);

\draw[draw=black,ultra thick] (2,0) arc (360:180:3);
\draw[draw=black,ultra thick] (-4,0) arc (0:180:3);
\draw[draw=black,ultra thick] (-10,0) arc (180:0:4);
\draw[draw=black,ultra thick] (-2,0) arc (0:180:5);
\draw[draw=black,ultra thick] (-12,0) arc (180:360:3);
\draw[draw=black,ultra thick] (-6,0) arc (360:180:4);

\draw[draw=white,thick] (2,0) arc (360:180:3);
\draw[draw=white,thick] (-4,0) arc (0:180:3);
\draw[draw=white,thick] (-10,0) arc (180:0:4);
\draw[draw=white,thick] (-2,0) arc (0:180:5);
\draw[draw=white,thick] (-12,0) arc (180:360:3);
\draw[draw=white,thick] (-6,0) arc (360:180:4);

\PUNCTURE{2}
\PUNCTURE{-14}

\puncture{-2}
\puncture{-6}
\puncture{-10}

\node at (-6,-11) {Vertical symmetry};
\node at (-6,-13) {$\bS_v(\beta) = \sigma_2 \sigma_1^{-1}$};
\end{tikzpicture}

\begin{tikzpicture}[scale=0.18]
\draw[fill=palegray,draw=palegray] (6,0) circle (8);

\draw[draw=black,ultra thick] (4,-9) -- (4,9);
\draw[draw=black,ultra thick] (8,-9) -- (8,9);

\draw[draw=black,ultra thick] (-2,0) arc (180:0:3);
\draw[draw=black,ultra thick] (4,0) arc (180:360:3);
\draw[draw=black,ultra thick] (10,0) arc (360:180:4);
\draw[draw=black,ultra thick] (2,0) arc (180:360:5);
\draw[draw=black,ultra thick] (12,0) arc (0:180:3);
\draw[draw=black,ultra thick] (6,0) arc (180:0:4);

\draw[draw=white,thick] (-2,0) arc (180:0:3);
\draw[draw=white,thick] (4,0) arc (180:360:3);
\draw[draw=white,thick] (10,0) arc (360:180:4);
\draw[draw=white,thick] (2,0) arc (180:360:5);
\draw[draw=white,thick] (12,0) arc (0:180:3);
\draw[draw=white,thick] (6,0) arc (180:0:4);

\PUNCTURE{-2}
\PUNCTURE{14}

\puncture{2}
\puncture{6}
\puncture{10}

\node at (6,-11) {Horizontal symmetry};
\node at (6,-13) {$\bS_h(\beta) = \sigma_1 \sigma_2^{-1}$};
\end{tikzpicture}
{\tiny~}
\begin{tikzpicture}[scale=0.18]
\draw[fill=palegray,draw=palegray] (-6,0) circle (8);

\draw[draw=black,ultra thick] (-4,-9) -- (-4,9);
\draw[draw=black,ultra thick] (-8,-9) -- (-8,9);

\draw[draw=black,ultra thick] (2,0) arc (0:180:3);
\draw[draw=black,ultra thick] (-4,0) arc (360:180:3);
\draw[draw=black,ultra thick] (-10,0) arc (180:360:4);
\draw[draw=black,ultra thick] (-2,0) arc (360:180:5);
\draw[draw=black,ultra thick] (-12,0) arc (180:0:3);
\draw[draw=black,ultra thick] (-6,0) arc (0:180:4);

\draw[draw=white,thick] (2,0) arc (0:180:3);
\draw[draw=white,thick] (-4,0) arc (360:180:3);
\draw[draw=white,thick] (-10,0) arc (180:360:4);
\draw[draw=white,thick] (-2,0) arc (360:180:5);
\draw[draw=white,thick] (-12,0) arc (180:0:3);
\draw[draw=white,thick] (-6,0) arc (0:180:4);

\PUNCTURE{2}
\PUNCTURE{-14}

\puncture{-2}
\puncture{-6}
\puncture{-10}

\node at (-6,-11) {Central symmetry};
\node at (-6,-13) {$\Delta^{-1} \beta \Delta = \sigma_2^{-1} \sigma_1$};
\end{tikzpicture}
\end{center}
\caption{Vertical, horizontal and central symmetries}
\label{fig:symmetries}
\end{figure}

\begin{lem}
For all braids $\beta \in B_n$, we have:
$\|\beta\|_\ell = \|\bS_v(\beta)\|_\ell = \|\bS_h(\beta)\|_\ell = \|\bS_c(\beta)\|_\ell$ and 
$\|\beta\|_d = \|\bS_v(\beta)\|_d = \|\bS_h(\beta)\|_d = \|\bS_c(\beta)\|_d$, i.e.
the laminated and diagrammatic norms are invariant under $\bS_v$, $\bS_h$ and $\bS_c$.
\label{lem:invariant-norm}
\end{lem}

\begin{proof}
From a geometric point of view, the braid morphisms $\bS_v$, $\bS_h$ and $\bS_c$
respectively induce vertical, horizontal and central symmetries
on the laminations and the curve diagrams.
More precisely, if $\calL$ and $\calD$ are a lamination and a curve diagram representing some braid $\beta \in B_n$,
then:
\begin{itemize}
\setlength{\itemsep}{0mm}
\item their vertically symmetric lamination $\calL_v$ and curve diagram $\calD_v$ represent the braid $\bS_v(\beta)$;
\item their horizontally symmetric lamination $\calL_h$ and curve diagram $\calD_h$ represent the braid $\bS_h(\beta)$;
\item their centrally symmetric lamination $\calL_c$ and curve diagram $\calD_c$ represent the braid $\bS_c(\beta)$.
\end{itemize}
\end{proof}

\section{Counting Braids With a Given Norm}
\label{section:counting-braids-with-a-given-norm}

We aim now at counting directly tight curve diagrams with a given (diagrammatic) norm.
Henceforth, we will denote by $N_{n,k}$ the number of braids on $n$ strands and (laminated or diagrammatic) norm $k$, i.e.
\[N_{n,k} := |\{\beta \in B_n : \|\beta\|_\ell = k\}| = |\{\beta \in B_n : \|\beta\|_d = k\}|,\]
and we will focus on computing the generating function $\calB_n(x) := \sum_{k \geq 0} N_{n,k} x^k$.
In order to achieve this goal,
we provide here intrinsic characterizations of tight curve diagrams,
then introduce discrete combinatorial structures
that will be in bijection with tight curve diagrams of a given norm.

\subsection{A Characterization of Tight Curve Diagrams}
\label{subsection:a-characterization-of-tight-curve-diagrams}

In the literature, most characterizations of tightness hold for laminations.
Here, we express them in terms of curve diagrams.
However, we first introduce a generalisation of curve diagrams.

\begin{dfn}{Generalised curve diagram}
Let us consider the set $P_n$ of $n$ punctures inside the disk $D^2$, and let $k \geq 1$ be some positive integer.

We call \emph{$k$-generalised curve diagram}, and denote by $\calD$,
a family of non-intersecting curves that consists of
\begin{itemize}
\item one open curve with endpoints $-1$ and $+1$;
\item $k-1$ closed curves, non of which encircling any of the points $\pm 1$
\end{itemize}
and such that each puncture of the disk belongs to one of these $k$ curves.
\label{dfn:generalised-trivial-diagram}
\end{dfn}

Note that ``normal'' curve diagrams are exactly $1$-generalised curve diagrams.

\begin{dfn}{Arcs and neighbour endpoints}
Let $\calL$ be a lamination and let $\calD$ be a generalised curve diagram such that $\calL \cap \calD$ is a finite set,
and such that $\calL$ and $\calD$ actually cross each other at each point of $\calL \cap \calD$:
we say that the sets $\calD$ and $\calL$ are \emph{compatible} with each other.

We call \emph{arc} of $\calL$ with respect to $\calD$ (or $(\calL,\calD)$-arc)
a connected component of $\calL \setminus \calD$.
The endpoints of a $(\calL,\calD)$-arc necessarily lie on $\calD$;
moreover, if the arc is bounded, i.e. if it has two endpoints, then
we call these endpoints \emph{neighbour endpoints} in $\calL$ with respect to $\calD$
(or $(\calL,\calD)$-neighbour endpoints).

Similarly, we call \emph{arc} of $\calD$ with respect to $\calL$ (or $(\calD,\calL)$-arc)
a connected component of $\calD \setminus \calL$.
Such an arc has two endpoints, which we call \emph{neighbour endpoints} in $\calD$ with respect to $\calL$
(or $(\calD,\calL)$-neighbour endpoints).
\label{dfn:dual-arcs-and-crosspoints}
\end{dfn}

Observe that, if $\calD$ is a ($1$-generalised) curve diagram,
the endpoints of the curve $\calD$ itself are $-1$ and $+1$,
and therefore two distinct $(\calD,\calL)$-arcs can share at most one endpoint.
Hence, if $P$ and $Q$ are $(\calD,\calL)$-neighbour endpoints, there exists one unique $(\calD,\calL)$-arc with endpoints $P$ and $Q$,
and we denote this arc by $[P,Q]_\calD$.

In addition, from this notion of arcs and neighbour endpoints follows a
standard intrinsic characterization of tight laminations(see~\cite{Dehornoy_whyare,Fenn_orderingthe}).

\begin{pro}
Let $\calL$ be a lamination and let $\bD = [-1,1]$ be the trivial curve diagram.
The lamination $\calL$ is tight if and only if,
for every pair $(P,Q)$ of elements of $\calL \cap \bD$, with $P \leq Q$, that are
both $(\calL,\bD)$- and $(\bD,\calL)$-neighbour endpoints,
the real interval $\{x \in \RR : P < x < Q\}$
contains at least one puncture among $p_1,\ldots,p_n$.
\label{pro:tight-if-full-arcs}
\end{pro}

From Proposition~\ref{pro:tight-if-full-arcs}
follows a characterization of (1-generalised) tight curve diagrams.

\begin{pro}
Let $\calD$ be a (1-generalised) curve diagram.
The curve diagram $\calD$ is tight if and only $\calD$ is compatible with the trivial lamination $\bL$ and if,
for every pair $(P,Q)$ of elements of $\bL \cap \calD$ that are
both $(\bL,\calD)$- and $(\calD,\bL)$-neighbour endpoints,
the arc $[P,Q]_\calD$ contains exactly one point among $p_1,\ldots,p_n$.
\label{pro:tight-iff-full-arcs-2}
\end{pro}

\begin{proof}
First, note that, if $\bL$ and $\calD$ have a common point $p$ at which $\calD$ does not \emph{cross} a curve of $\bL$,
then $\calD$ is certainly not tight:
indeed, $p$ cannot be one of the points $p_1,\ldots,p_n$ and therefore one can modify slightly $\calD$
in order to obtain a curve diagram $\calD'$ isotopic to $\calD$ (i.e. representing the same braid) and
such that $\bL \cap \calD' \subseteq (\bL \cap \calD) \setminus \{p\}$.
Henceforth, we assume that $\calD$ is compatible with $\bL$.

Let $h$ be an homeomorphism of $\CC$
such that $h(\bD) = \calD$, $h(P_n) = P_n$,
$h(1) = 1$ and $h(-1) = -1$.
Then, let $\calL$ be the lamination such that $h(\calL) = \bL$.
According to Proposition~\ref{pro:laminated-norm-diagrammatic},
the diagram $\calD$ is tight if and only $\calL$ is tight.
According to Proposition~\ref{pro:tight-if-full-arcs},
this also means that, for every pair $(P,Q)$ of elements of $\calL \cap \bD$ that are
both $(\calL,\bD)$- and $(\bD,\calL)$-neighbour endpoints,
the arc $[P,Q]_\bD$ contains at least one point among $p_1,\ldots,p_n$.
Since $(P,Q)$ are both $(\calL,\bD)$- and $(\bD,\calL)$-neighbour endpoints if and only if
$(h(P),h(Q))$ are both $(\bL,\calD)$- and $(\calD,\bL)$-neighbour endpoints,
Proposition~\ref{pro:tight-iff-full-arcs-2} follows.
\end{proof}

Following Proposition~\ref{pro:tight-iff-full-arcs-2},
we introduce the notion of \emph{tight} generalised curve diagram.

\begin{dfn}{Tight generalised curve diagram}
Let $\calD$ be a generalised curve diagram.
We say that $\calD$ is \emph{tight} if $\calD$ is compatible with $\bL$ and if
for every $(\calD,\bL)$-arc $A$ whose endpoints are $(\bL,\calD)$-neighbour endpoints,
the arc $A$ contains one point among $p_1,\ldots,p_n$.
\end{dfn}

Note that, according to Proposition~\ref{pro:tight-iff-full-arcs-2}, this notion of tightness coincides
with the usual notion of tightness on ($1$-generalised) curve diagrams.

\subsection{From Diagrams to Coordinates}
\label{subsection:from-drawings-to-combinatorics}

Proposition~\ref{pro:tight-iff-full-arcs-2} provides an intrinsic characterisation of ($1$-generalised) tight curve diagrams.
However, such a characterization is not directly suitable for counting tight curve diagrams:
to achieve this counting process, we need new tools, which we introduce now.

\begin{dfn}{Endpoints ordering}
Let $\calD$ be a generalised curve diagram compatible with the trivial lamination $\bL$,
and let $\bL_i$ be a curve of $\bL$, with $1 \leq i \leq n-1$.

The punctures $p_1$ and $p_n$ belong to distinct connected components of $\CC \setminus \bL_i$.
Hence, the curve $\bL_i$ intersects the curve diagram $\calD$ at least once.
We orient each curve $\bL_i$ \emph{from bottom to top} and thereby
induce a linear ordering on $\bL_i \cap \calD$:
we denote by $\conc{i}{j}$ the $j$-th smallest element of $\bL_i \cap \calD$.

For the sake of coherence, we also define the sets $\bL_0 := \{-1\}$ and $\bL_n := \{+1\}$.
Then, we denote by $\conc{0}{1}$ the left point $-1$, and by $\conc{n}{1}$ the right point $+1$.
\label{dfn:endpoints-ordering}
\end{dfn}

An immediate induction on $i$ and on $k$ shows that, if $\calD$ is a $k$-generalised curve diagram
compatible with $\bL$, then the cardinality of the set $\bL_i \cap \calD$ is odd.

In addition, if $\calD$ is compatible with $\bL$, any two points $P$ and $Q$ lying on distinct lines
$\bL_{i-1}$ and $\bL_i$ can be linked by at most one $(\calD,\bL)$-arc.
Therefore, we can unambiguously denote this arc by $[P,Q]_\calD$,
which gives rise to the following definition of \emph{coordinates} of a tight generalised curve diagram.

\begin{dfn}{Curve diagram coordinates and braid coordinates}
Let $\calD$ be a tight generalised curve diagram.
The \emph{coordinates} of $\calD$ are defined
as the tuple $\bs\ba := (s_0,a_1,s_1,a_2,\ldots,a_n,s_n)$ such that
\begin{itemize}
\item $s_i = \frac{1}{2}(|\bL_i \cap \calD|-1)$, for all $i \in \{0,\ldots,n\}$;
\item $a_i = \max\{j \geq 0 : \forall k \in \{1,\ldots,j\}, \conc{i-1}{k}$ and $\conc{i}{k}$ are $(\calD,\bL)$-neighbour endpoints$\}$,
for all $i \in \{1,\ldots,n\}$ such that $s_{i-1} \neq s_i$;
\item $a_i$ is the integer such that the puncture $p_i$ lies on the arc $[\conc{i-1}{a_i+1}, \conc{i}{a_i+1}]_\calD$, for all $i \in \{1,\ldots,n\}$
such that $s_{i-1} = s_i$.
\end{itemize}

If, in addition, $\calD$ is a ($1$-generalised) curve diagram, representing some braid $\beta$, then we also say that
$\bs\ba$ are the coordinates of $\beta$.
\label{dfn:curve-diagram-coordinates}
\end{dfn}

\begin{figure}[!ht]
\begin{center}
\begin{tikzpicture}[scale=0.24]
\draw[fill=palegray,draw=palegray] (14,0) circle (8);

\draw[draw=black,ultra thick] (12,-9) -- (12,9);
\draw[draw=black,ultra thick] (16,-9) -- (16,9);

\draw[draw=black,ultra thick] (6,0) arc (180:0:4);
\draw[draw=black,ultra thick] (14,0) arc (0:180:3);
\draw[draw=black,ultra thick] (8,0) arc (180:360:5);
\draw[draw=black,ultra thick] (18,0) arc (360:180:4);
\draw[draw=black,ultra thick] (10,0) arc (180:360:3);
\draw[draw=black,ultra thick] (16,0) arc (180:0:3);

\draw[draw=white,thick] (6,0) arc (180:0:4);
\draw[draw=white,thick] (14,0) arc (0:180:3);
\draw[draw=white,thick] (8,0) arc (180:360:5);
\draw[draw=white,thick] (18,0) arc (360:180:4);
\draw[draw=white,thick] (10,0) arc (180:360:3);
\draw[draw=white,thick] (16,0) arc (180:0:3);

\PUNCTUREE{6}{1}
\PUNCTUREE{22}{1}

\puncturee{10}{1}
\puncturee{14}{1}
\puncturee{18}{1}

\draw[draw=black,densely dotted,thick] (16,0) -- (19,4) -- (24,4);
\draw[draw=black,densely dotted,thick] (16,-3.4641) -- (24,-3.4641);
\draw[draw=black,densely dotted,thick] (16,-4) -- (17.4641,-5.4641) -- (24,-5.4641);
\draw[draw=black,densely dotted,thick] (22,0) -- (24,0);

\draw[draw=black,densely dotted,thick] (12,3.4641) -- (10.63567,4.82843) -- (4,4.82843);
\draw[draw=black,densely dotted,thick] (12,2.82843) -- (4,2.82843);
\draw[draw=black,densely dotted,thick] (12,-2.82843) -- (4,-2.82843);
\draw[draw=black,densely dotted,thick] (12,-3.4641) -- (10.63567,-4.82843) -- (4,-4.82843);
\draw[draw=black,densely dotted,thick] (12,-4.89898) -- (10.07055,-6.82843) -- (4,-6.82843);
\draw[draw=black,densely dotted,thick] (6,0) -- (4,0);

\node[anchor=east] at (4.2,0) {$\conc{0}{1}$};
\node[anchor=east] at (4.2,4.82843) {$\conc{1}{5}$};
\node[anchor=east] at (4.2,2.82843) {$\conc{1}{4}$};
\node[anchor=east] at (4.2,-2.82843) {$\conc{1}{3}$};
\node[anchor=east] at (4.2,-4.82843) {$\conc{1}{2}$};
\node[anchor=east] at (4.2,-6.82843) {$\conc{1}{1}$};
\node[anchor=west] at (23.8,4) {$\conc{2}{3}$};
\node[anchor=west] at (23.8,-3.4641) {$\conc{2}{2}$};
\node[anchor=west] at (23.8,-5.4641) {$\conc{2}{1}$};
\node[anchor=west] at (23.8,0) {$\conc{3}{1}$};

\node at (14,11) {$1$-generalised diagram};
\node at (14,-11) {Coordinates: $(0,0,2,3,1,0,0)$};
\end{tikzpicture}
{\tiny~}
\begin{tikzpicture}[scale=0.24]
\draw[fill=palegray,draw=palegray] (14,0) circle (8);

\draw[draw=black,ultra thick] (12,-9) -- (12,9);
\draw[draw=black,ultra thick] (16,-9) -- (16,9);

\draw[draw=black,ultra thick] (6,0) arc (180:0:6) -- (22,0);
\draw[draw=black,ultra thick] (14,0) arc (0:360:2);

\draw[draw=white,thick] (6,0) arc (180:0:6) -- (22,0);
\draw[draw=white,thick] (14,0) arc (0:360:2);

\PUNCTUREE{6}{1}
\PUNCTUREE{22}{1}

\puncturee{10}{1}
\puncturee{14}{1}
\puncturee{18}{1}

\draw[draw=black,densely dotted,thick] (16,4.47214) -- (24,4.47214);
\draw[draw=black,densely dotted,thick] (22,0) -- (24,0);

\draw[draw=black,densely dotted,thick] (12,6) -- (4,6);
\draw[draw=black,densely dotted,thick] (12,2) -- (4,2);
\draw[draw=black,densely dotted,thick] (12,-2) -- (4,-2);
\draw[draw=black,densely dotted,thick] (6,0) -- (4,0);

\node[anchor=east] at (4.2,0) {$\conc{0}{1}$};
\node[anchor=east] at (4.2,6) {$\conc{1}{3}$};
\node[anchor=east] at (4.2,2) {$\conc{1}{2}$};
\node[anchor=east] at (4.2,-2) {$\conc{1}{1}$};
\node[anchor=west] at (23.8,4.47214) {$\conc{2}{1}$};
\node[anchor=west] at (23.8,0) {$\conc{3}{1}$};

\node at (14,11) {$2$-generalised diagram};
\node at (14,-11) {Coordinates: $(0,0,1,0,0,0,0)$};
\end{tikzpicture}
\end{center}
\caption{Tight generalised curve diagram and associated coordinates}
\label{fig:coordinates-ordering}
\end{figure}
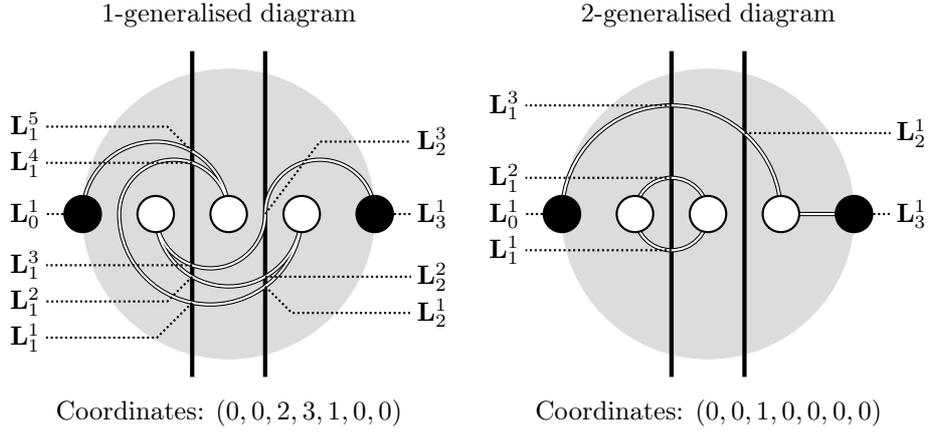

We show below that coordinates indeed characterise tight generalised curve diagrams.

\begin{dfn}{Zones and neighbours}
Let $\calD$ be some tight generalised curve diagram, and let $i \in \{1,\ldots,n\}$ be some integer.
We denote by $\calZ_i$ be the area lying to the left of $\bL_i$ (if $i \leq n-1$) and to the right of $\bL_{i-1}$ (if $i \geq 2$):
we call $\calZ_i$ the \emph{$i$-th zone} of the diagram.
In addition, let $A$ be some $(\calD,\bL)$-arc lying inside the area $\calZ_i$, and let $P$ and $Q$ be the endpoints of the arc $A$.
We say that $P$ and $Q$ are \emph{$i$-th zone neighbours}, which we denote by $P \stackrel{i}{\sim} Q$,
and denote the arc $A$ by $[P,Q]_\calD^i$.
\end{dfn}

Observe that, although there may exist two $(\calD,\bL)$-arcs with endpoints $P$ and $Q$
(when $\calD$ is a tight $k$-generalised curve diagram with $k \geq 2$, e.g. the points
$\conc{1}{2}$ and $\conc{1}{3}$ in the $2$-generalised curve diagram of Fig.~\ref{fig:coordinates-ordering}),
the arc $[P,Q]_\calD^i$ itself is uniquely defined.
Indeed, remember that $\calD$ and $\bL$ are compatible with each other,
which contradicts the fact that $P$ and $Q$ may be linked by two or more arcs lying in the same area $\calZ_i$.

The notion of zones and of $i$-th zone neighbours leads to the following results.

\begin{lem}
Let $\calD$ be a tight generalised curve diagram with coordinates $\bs\ba := (s_0,a_1,s_1,\ldots,s_n)$, and let $i \in \{1,\ldots,n\}$
be some integer.
The $(\calD,\bL)$-arcs contained in the area $\calZ_i$ link repectively:
\begin{itemize}
\item the points $\conc{i-1}{j}$ and $\conc{i}{j}$ such that $j \leq a_i$;
\item the points $\conc{i-1}{j}$ and $\conc{i-1}{k}$ such that $j+k = 2(a_i+s_{i-1}-s_i)+1$ and $\min\{j,k\} > a_i$, if $s_{i-1} > s_i$;
\item the points $\conc{i}{j}$ and $\conc{i}{k}$ such that $j+k = 2(a_i+s_i-s_{i-1})+1$ and $a_i < j < k$, if $s_i > s_{i-1}$;
\item the points $\conc{i-1}{j}$ and $\conc{i}{k}$ such that $j-k = 2(s_{i-1}-s_i)$ and $\min\{j,k\} > a_i$.
\end{itemize}
\label{lem:left-right-box}
\end{lem}

\begin{proof}
We call \emph{left-box} each interval $\{u,\ldots,v\}$ such that the points $\conc{i-1}{u} \stackrel{i}{\sim} \conc{i-1}{v}$, and
\emph{right-box} each interval $\{u,\ldots,v\}$ such that $\conc{i}{u} \stackrel{i}{\sim} \conc{i}{v}$.

If $\{u,\ldots,v\}$ is a left-box (with $u < v$),
then for all $w \in \{u+1,\ldots,v-1\}$ there exists some $x \in \{u+1,\ldots,v-1\}$ such that
$\{w,\ldots,x\}$ (or $\{x,\ldots,w\}$ if $x < w$) is also a left-box. An immediate induction then shows that
the interval $\{u,\ldots,v\}$ necessarily contains some minimal left-box, which must be of the form $\{x,x+1\}$.
Hence Proposition~\ref{pro:tight-iff-full-arcs-2} shows that
the arc $[\conc{i-1}{x},\conc{i-1}{x+1}]_\calD^i$ must be \emph{the} $(\calD,\bL)$-arc containing the puncture $p_i$.

This proves that, among any two left-boxes, one must contain the other, and again an immediate induction proves that
the family of left-boxes must be a set of intervals of the form $\{\{y-k,\ldots,y+k-1\} : 1 \leq k \leq \ell\}$ for some integer $\ell \geq 0$.
Similarly, the family of right-boxes must be a set of intervals of the form
$\{\{z-k,\ldots,z+k-1\} : 1 \leq k \leq m\}$ for some integer $m \geq 0$.
In addition, since $p_i$ belongs to only one $(\calD,\bL)$-arc, we must either have $\ell = 0$ or $m = 0$.

Finally, consider the two sets $\bS_1 = \{1,\ldots,y-\ell-1,y+\ell,\ldots,2s_{i-1}+1\}$ and $\bS_2 = \{1,\ldots,z-m-1,z+m,\ldots,2s_i+1\}$.
Each point $\conc{i-1}{u}$ with $u \in \bS_1$ must be $i$-th zone neighbour with some point $\conc{i}{v}$ with $v \in \bS_2$, and vice-versa.
Since the arcs of $\calD$ cannot cross each other, it comes immediately that the set $\bS_1$ and $\bS_2$ have the same cardinality
and that, if $u$ is the $j$-th smallest element of $\bS_1$ and $v$ is the $j$-th smallest element of $\bS_2$, for some $j \in \{1,\ldots,|\bS_1|\}$,
then $\conc{i-1}{u} \stackrel{i}{\sim} \conc{i}{v}$.
Considering the definition of the coordinates $\bs\ba$, Lemma~\ref{lem:left-right-box} follows.
\end{proof}

\begin{pro}
Let $\calD$ and $\calD'$ be two tight generalised curve diagrams with respective coordinates $\bs\ba$ and $\bs\ba'$.
If $\bs\ba = \bs\ba'$, then there exists some isopoty of $\CC$, preserving $\bL$ setwise and $\{p_i : 1 \leq n\}$ pointwise,
and that maps $\calD$ to $\calD'$.
\label{pro:same-coordinates-implies-isotopic}
\end{pro}

\begin{proof}
Let $(s_0,a_1,\ldots,s_n)$ be the common coordinates of $\calD$ and $\calD'$.
Since $|\calD \cap \bL_i| = 2s_i +1 = |\calD' \cap \bL_i|$ for all $i \in \{0,\ldots,n\}$,
we assume without loss of generality that $\calD \cap \bL = \calD' \cap \bL$.
Lemma~\ref{lem:left-right-box} proves then that some isotopy of $\CC$ preserving pointwise the lines $\bL_i$ (for $1 \leq i \leq n-1$)
maps the diagram $\calD$ to the diagram $\calD'$.

In addition, denote by $b_i$ the integer $a_i+|s_{i-1}-s_i|$.
Observe that the puncture $p_i$ lies on the arcs
\begin{itemize}
\item $[\conc{i-1}{b_i},\conc{i-1}{b_i+1}]_\calD^i$ and
$[\conc{i-1}{b_i},\conc{i-1}{b_i+1}]_{\calD'}^i$ if $s_{i-1} > s_i$;
\item $[\conc{i}{b_i},\conc{i}{b_i+1}]_\calD^i$ and
$[\conc{i}{b_i},\conc{i}{b_i+1}]_{\calD'}^i$ if $s_i > s_{i-1}$;
\item $[\conc{i-1}{a_i+1},\conc{i}{a_i+1}]_\calD^i$ and
$[\conc{i-1}{a_i+1},\conc{i}{a_i+1}]_{\calD'}^i$ if $s_{i-1} = s_i$.
\end{itemize}
Therefore, we can even assume that the above-mentioned isotopy preserves each puncture $p_i$,
which is the statement of Proposition~\ref{pro:same-coordinates-implies-isotopic}.
\end{proof}

\begin{cor}
Let $\beta$ and $\beta'$ be two braids with coordinates $\bs\ba$ and $\bs\ba'$.
If $\bs\ba = \bs\ba'$, then $\beta = \beta'$.
\label{cor:same-coordinates-same-braid}
\end{cor}

These coordinates are therefore analogous to the Dynnikov coordinates
(see~\cite{efficientword,Dehornoy_whyare} for details) in several respects.
First, both arise from counting intersection points between different collections of lines.
Second, both provide an injective mapping from the braid group $B_n$ into the set $\ZZ^{2n}$ or $\ZZ^{2n+1}$.
Finally, both systems of coordinates come with very efficient algorithms, whose complexities are of the same order of magnitude.
However, the coordinates used here are very closely linked with the notion of (diagrammatic or laminated) norm,
whereas the process of computing the norm of a braid from its Dynnikov coordinates is less immediate.

\subsection{From Coordinates to Diagrams}
\label{subsection:from-combinatorics-to-drawings}

Proposition~\ref{pro:same-coordinates-implies-isotopic} and Corollary~\ref{cor:same-coordinates-same-braid}
allow us to identify each tight generalised curve diagram and each braid with a tuple of coordinates.
Aiming to count ($1$-generalised) tight curve diagrams,
we aim now at describing which coordinates correspond to generalised tight curve diagrams.

\begin{lem}
Let $\calD$ be a tight generalised curve diagram, with coordinates $(s_0,a_1,\ldots,s_n)$.
We have $s_0 = s_n = 0$, and $0 \leq a_i \leq 2 \min \{s_{i-1},s_i\} + \mathbf{1}_{s_{i-1} \neq s_i}$
for all integers $i \in \{1,\ldots,n\}$.
\label{lem:inequalities-reduced-coordinates}
\end{lem}

\begin{proof}
By definition, we have $2 s_i + 1 = |\calD \cap \bL_i|$ whenever $0 \leq i \leq n$.
Hence, whenever $s_{i-1} = s_i$, Definition~\ref{dfn:curve-diagram-coordinates} directly implies that 
$1 \leq a_i+1 \leq |\calD \cap \bL_i| = 2 s_i+1$.
However, if $s_{i-1} \neq s_i$, then the points $\conc{i-1}{j}$ and $\conc{i}{j}$ exist (and are $(\calD,\bL)$-neighbours) whenever $j \leq a_i$,
which proves that
\[0 \leq a_i \leq \min \{|\calD \cap \bL_{i-1}|,|\calD \cap \bL_i|\} = 2 \min \{s_{i-1},s_i\} + 1.\]
\end{proof}

In the following, we call \emph{virtual coordinates} the tuples $(s_0,a_1,\ldots,s_n)$
that satisfy the equalities and inequalities mentioned in Lemma~\ref{lem:inequalities-reduced-coordinates}.
A natural question is, provided some virtual coordinates $\bs\ba$, whether
there exists some tight generalised curve diagram whose coordinates ade $\bs\ba$.
We prove now that this is the case.

\begin{pro}
Let $\bs\ba = (s_0,a_1,\ldots,s_n)$ be virtual coordinates.
There exists some tight generalised curve diagram $\calD$ whose coordinates are $\bs\ba$.
\label{pro:virtual-is-enough}
\end{pro}

\begin{center}
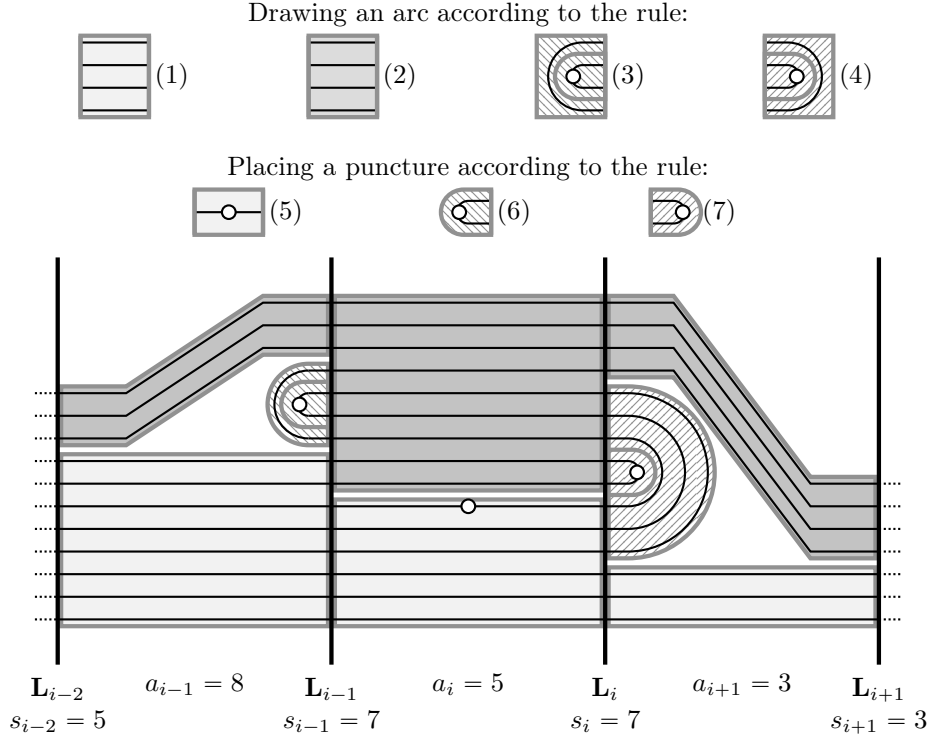
\begin{figure}[!ht]
\begin{center}
\begin{tikzpicture}[scale=0.3]
\draw[draw=blackgray,fill=verypalegray,ultra thick] (-3,21.2) -- (0,21.2) -- (0,24.8) -- (-3,24.8) -- cycle;
\draw[draw=blackgray,fill=palegray,ultra thick] (7,21.2) -- (10,21.2) -- (10,24.8) -- (7,24.8) -- cycle;
\draw[draw=blackgray,pattern=north west lines,pattern color=blackgray,ultra thick] (17,21.2) -- (20,21.2) -- (20,24.8) -- (17,24.8) -- cycle;
\draw[draw=blackgray,pattern=north east lines,pattern color=blackgray,ultra thick] (27,21.2) -- (30,21.2) -- (30,24.8) -- (27,24.8) -- cycle;
\draw[draw=blackgray,ultra thick] (20,22) -- (18.8,22) arc (270:90:1) -- (20,24) -- cycle;
\draw[draw=blackgray,ultra thick] (27,22) -- (28.2,22) arc (-90:90:1) -- (27,24) -- cycle;
\draw[draw=black,thick] (-3,21.5) -- (0,21.5);
\draw[draw=black,thick] (-3,22.5) -- (0,22.5);
\draw[draw=black,thick] (-3,23.5) -- (0,23.5);
\draw[draw=black,thick] (-3,24.5) -- (0,24.5);
\draw[draw=black,thick] (7,21.5) -- (10,21.5);
\draw[draw=black,thick] (7,22.5) -- (10,22.5);
\draw[draw=black,thick] (7,23.5) -- (10,23.5);
\draw[draw=black,thick] (7,24.5) -- (10,24.5);
\draw[draw=black,thick] (20,22.5) -- (19,22.5) arc (270:90:0.5) -- (20,23.5);
\draw[draw=black,thick] (20,21.5) -- (19,21.5) arc (270:90:1.5) -- (20,24.5);
\draw[draw=black,thick] (27,22.5) -- (28,22.5) arc (-90:90:0.5) -- (27,23.5);
\draw[draw=black,thick] (27,21.5) -- (28,21.5) arc (-90:90:1.5) -- (27,24.5);
\draw[draw=blackgray,ultra thick] (-3,21.2) -- (-3,24.8);
\draw[draw=blackgray,ultra thick] (0,21.2) -- (0,24.8);
\draw[draw=blackgray,ultra thick] (7,21.2) -- (7,24.8);
\draw[draw=blackgray,ultra thick] (10,21.2) -- (10,24.8);
\draw[draw=blackgray,ultra thick] (20,21.2) -- (20,24.8);
\draw[draw=blackgray,ultra thick] (27,21.2) -- (27,24.8);
\draw[draw=black,fill=white,thick] (18.6,23) circle (0.3);
\draw[draw=black,fill=white,thick] (28.4,23) circle (0.3);

\draw[draw=blackgray,fill=verypalegray,ultra thick] (2,16) -- (5,16) -- (5,18) -- (2,18) -- cycle;
\draw[draw=blackgray,pattern=north west lines,pattern color=blackgray,ultra thick] (15,16) -- (13.8,16) arc (270:90:1) -- (15,18) -- cycle;
\draw[draw=blackgray,pattern=north east lines,pattern color=blackgray,ultra thick] (22,16) -- (23.2,16) arc (-90:90:1) -- (22,18) -- cycle;
\draw[draw=black,thick] (2,17) -- (5,17);
\draw[draw=black,thick] (15,16.5) -- (14,16.5) arc (270:90:0.5) -- (15,17.5);
\draw[draw=black,thick] (22,16.5) -- (23,16.5) arc (-90:90:0.5) -- (22,17.5);
\draw[draw=blackgray,ultra thick] (2,16) -- (2,18);
\draw[draw=blackgray,ultra thick] (5,16) -- (5,18);
\draw[draw=blackgray,ultra thick] (15,16) -- (15,18);
\draw[draw=blackgray,ultra thick] (22,16) -- (22,18);
\draw[draw=black,fill=white,thick] (3.5,17) circle (0.3);
\draw[draw=black,fill=white,thick] (13.6,17) circle (0.3);
\draw[draw=black,fill=white,thick] (23.4,17) circle (0.3);

\node[anchor=south] at (14,24.8) {Drawing an arc according to the rule:};
\node at (1,23) {$(1)$};
\node at (11,23) {$(2)$};
\node at (21,23) {$(3)$};
\node at (31,23) {$(4)$};

\node[anchor=south] at (14,18) {Placing a puncture according to the rule:};
\node at (6,17) {$(5)$};
\node at (16,17) {$(6)$};
\node at (25,17) {$(7)$};

\draw[draw=blackgray,fill=verypalegray,ultra thick] (-3.85,-1.3) -- (7.85,-1.3) -- (7.85,6.3) -- (-3.85,6.3) -- cycle;
\draw[draw=blackgray,fill=gray,ultra thick] (-3.85,6.7) -- (-1,6.7) -- (5,10.7) -- (7.85,10.7) -- (7.85,13.3) -- (5,13.3) -- (-1,9.3) -- (-3.85,9.3) -- cycle;
\draw[draw=blackgray,pattern=north west lines,pattern color=blackgray,ultra thick] (7.85,10.3) -- (7,10.3) arc (90:270:1.8) -- (7.85,6.7) -- cycle;
\draw[draw=blackgray,ultra thick] (7.85,7.5) -- (6.8,7.5) arc (270:90:1) -- (7.85,9.5);

\draw[draw=blackgray,fill=verypalegray,ultra thick] (8.15,-1.3) -- (19.85,-1.3) -- (19.85,4.3) -- (8.15,4.3) -- cycle;
\draw[draw=blackgray,fill=gray,ultra thick] (8.15,4.7) -- (19.85,4.7) -- (19.85,13.3) -- (8.15,13.3) -- cycle;

\draw[draw=blackgray,fill=verypalegray,ultra thick] (20.15,-1.3) -- (31.85,-1.3) -- (31.85,1.3) -- (20.15,1.3) -- cycle;
\draw[draw=blackgray,fill=gray,ultra thick] (20.15,9.7) -- (23,9.7) -- (29,1.7) -- (31.85,1.7) -- (31.85,5.3) -- (29,5.3) -- (23,13.3) -- (20.15,13.3) -- cycle;
\draw[draw=blackgray,pattern=north east lines,pattern color=blackgray,ultra thick] (20.15,1.7) -- (21,1.7) arc (-90:90:3.8) -- (20.15,9.3) -- cycle;
\draw[draw=blackgray,ultra thick] (20.15,4.5) -- (21.2,4.5) arc (-90:90:1) -- (20.15,6.5);

\draw[draw=black,ultra thick] (-4,-3) -- (-4,15);
\draw[draw=black,ultra thick] (8,-3) -- (8,15);
\draw[draw=black,ultra thick] (20,-3) -- (20,15);
\draw[draw=black,ultra thick] (32,-3) -- (32,15);

\draw[draw=black,thick] (-4,-1) -- (32,-1);
\draw[draw=black,thick] (-4,0) -- (32,0);
\draw[draw=black,thick] (-4,1) -- (32,1);
\draw[draw=black,thick] (-4,2) -- (21,2);
\draw[draw=black,thick] (-4,3) -- (21,3);
\draw[draw=black,thick] (-4,4) -- (21,4);
\draw[draw=black,thick] (-4,5) -- (21,5);
\draw[draw=black,thick] (-4,6) -- (21,6);
\draw[draw=black,thick] (-4,7) -- (-1,7) -- (5,11) -- (23,11) -- (29,3) -- (32,3);
\draw[draw=black,thick] (-4,8) -- (-1,8) -- (5,12) -- (23,12) -- (29,4) -- (32,4);
\draw[draw=black,thick] (-4,9) -- (-1,9) -- (5,13) -- (23,13) -- (29,5) -- (32,5);
\draw[draw=black,thick] (7,7) -- (21,7);
\draw[draw=black,thick] (7,8) -- (21,8);
\draw[draw=black,thick] (7,9) -- (21,9);
\draw[draw=black,thick] (7,10) -- (23,10) -- (29,2) -- (32,2);

\draw[draw=black,thick] (21,2) arc (-90:90:3.5);
\draw[draw=black,thick] (21,3) arc (-90:90:2.5);
\draw[draw=black,thick] (21,4) arc (-90:90:1.5);
\draw[draw=black,thick] (21,5) arc (-90:90:0.5);

\draw[draw=black,thick] (7,7) arc (270:90:1.5);
\draw[draw=black,thick] (7,8) arc (270:90:0.5);

\draw[draw=black,thick,densely dotted] (-5,-1) -- (-4,-1);
\draw[draw=black,thick,densely dotted] (-5,0) -- (-4,0);
\draw[draw=black,thick,densely dotted] (-5,1) -- (-4,1);
\draw[draw=black,thick,densely dotted] (-5,2) -- (-4,2);
\draw[draw=black,thick,densely dotted] (-5,3) -- (-4,3);
\draw[draw=black,thick,densely dotted] (-5,4) -- (-4,4);
\draw[draw=black,thick,densely dotted] (-5,5) -- (-4,5);
\draw[draw=black,thick,densely dotted] (-5,6) -- (-4,6);
\draw[draw=black,thick,densely dotted] (-5,7) -- (-4,7);
\draw[draw=black,thick,densely dotted] (-5,8) -- (-4,8);
\draw[draw=black,thick,densely dotted] (-5,9) -- (-4,9);

\draw[draw=black,thick,densely dotted] (32,-1) -- (33,-1);
\draw[draw=black,thick,densely dotted] (32,0) -- (33,0);
\draw[draw=black,thick,densely dotted] (32,1) -- (33,1);
\draw[draw=black,thick,densely dotted] (32,2) -- (33,2);
\draw[draw=black,thick,densely dotted] (32,3) -- (33,3);
\draw[draw=black,thick,densely dotted] (32,4) -- (33,4);
\draw[draw=black,thick,densely dotted] (32,5) -- (33,5);

\node[anchor=north] at (-4,-3) {\begin{tabular}{c} $\bL_{i-2}$ \\ $s_{i-2}=5$ \end{tabular}};
\node[anchor=north] at (8,-3) {\begin{tabular}{c} $\bL_{i-1}$ \\ $s_{i-1}=7$ \end{tabular}};
\node[anchor=north] at (20,-3) {\begin{tabular}{c} $\bL_i$ \\ $s_i=7$ \end{tabular}};
\node[anchor=north] at (32,-3) {\begin{tabular}{c} $\bL_{i+1}$ \\ $s_{i+1}=3$ \end{tabular}};

\node[anchor=north] at (2,-3) {$a_{i-1}=8$};
\node[anchor=north] at (14,-3) {$a_i=5$};
\node[anchor=north] at (26,-3) {$a_{i+1}=3$};

\draw[draw=black,fill=white,thick] (6.6,8.5) circle (0.3);
\draw[draw=black,fill=white,thick] (14,4) circle (0.3);
\draw[draw=black,fill=white,thick] (21.4,5.5) circle (0.3);
\end{tikzpicture}
\end{center}
\caption{Drawing lines and placing punctures of a diagram based on its coordinates}
\label{fig:virtual-witnesses}
\end{figure}
\end{center}

\begin{proof}
We just need follow the recipe provided by Lemma~\ref{lem:left-right-box}
and draw the diagram $\calD$. First, we call $\conc{0}{1}$ and $\conc{n}{1}$ the points $-1$ and $+1$.
Then, on each line $\bL_i$, for $i \in \{1,\ldots,n-1\}$, let us place $2s_i+1$ points $\conc{i}{1},\ldots,\conc{i}{2s_i+1}$, from bottom to top.
Now, for each integer $i \in \{1,\ldots,n\}$, define the integer $b_i := a_i+|s_{i-1}-s_i|$.
We draw lines that lie inside the zone $\calZ_i$ and that link
\begin{enumerate}
\item points $\conc{i-1}{j}$ and $\conc{i}{j}$ such that $j \leq a_i$;
\item points $\conc{i-1}{j}$ and $\conc{i-1}{k}$ such that $j+k = 2b_i+1$ and $a_i < j < k$, if $s_{i-1} > s_i$;
\item points $\conc{i}{j}$ and $\conc{i}{k}$ such that $j+k = 2b_i+1$ and $a_i < j < k$, if $s_i > s_{i-1}$;
\item points $\conc{i-1}{j}$ and $\conc{i}{k}$ such that $j-k = 2(s_{i-1}-s_i)$ and $\min\{j,k\} > a_i$,
\end{enumerate}
so that no two such lines intersect each other.
Note that it is indeed possible to do so (e.g., if the points $(\conc{i}{j})_{1 \leq j \leq 2s_i+1}$ are close enough to each other,
one can draw straight segments $[\conc{i-1}{j}, \conc{i}{k}]$ and half-circles with diameter $[\conc{i-1}{j},\conc{i-1}{k}]$ or
$[\conc{i}{j},\conc{i}{k}]$).

Then, place a point $\overline{p}_i$ on the arc
\begin{enumerate}
\setcounter{enumi}{4}
 \item $[\conc{i-1}{b_i},\conc{i-1}{b_i+1}]_\calD^i$ if $s_{i-1} > s_i$;
 \item $[\conc{i}{b_i},\conc{i}{b_i+1}]_\calD^i$ if $s_i > s_{i-1}$;
 \item $[\conc{i-1}{a_i+1},\conc{i}{a_i+1}]_\calD^i$ if $s_{i-1} = s_i$.
\end{enumerate}

Up to an isotopy of $\CC$ preserving the lamination $\bL$ pointwise and
mapping each point $\overline{p}_i$ to the actual position of the puncture $p_i$,
we have just drawn the diagram $\calD$.
\end{proof}

\section{Actually Counting Braids}
\label{section:actually-counting-braids}

We proceed now to completing the task we set at the beginning of
Section~\ref{section:counting-braids-with-a-given-norm}:
we study here the integers
$N_{n,k} := |\{\beta \in B_n : \|\beta\|_\ell = k\}| = |\{\beta \in B_n : \|\beta\|_d = k\}|$ and the
geometric generating functions $\calB_n(z) := \sum_{\beta \in B_n} z^{\|\beta\|_d} = \sum_{\beta \in B_n} z^{\|\beta\|_\ell}$
provided by Corollaries~\ref{cor:count-twice-size} and~\ref{cor:equal-series}.

By definition of the coordinates of braids,
if a braid $\beta \in B_n$ has coordinates $\bs\ba := (s_0,a_1,\ldots,s_n)$, then
$\|\beta\|_d = n-1 + 2 \sum_{i=1}^{n-1}s_i$.
Therefore, instead of considering directly the integers $N_{n,k}$ and the function $\calB_n(z)$ themselves,
we will rather focus on the integers $g_{n,k} := N_{n,2k+n-1}$ and on the generating function $\calG_n(z) := \sum_{k \geq 0} g_{n,k} z^k$.
In particular, observe that $\calB_n(z) = z^{n-1} \calG_n(z^2)$,
so that properties about the integers $g_{n,k}$ and on the function $\calG_n(z)$ reflect on the integers $N_{n,k}$ and the function $\calB_n(z)$.

Then, following Proposition~\ref{pro:same-coordinates-implies-isotopic} and~\ref{pro:virtual-is-enough},
we want to count the tuples $\bs\ba$ of virtual coordinates, with a given sum $\sum_{i=1}^{n-1}s_i$,
and whose associated generalised curve diagram is a $1$-generalised curve diagram.

Let $\calD$ be the generalised diagram associated to $\bs\ba$.
We denote by $\sim$ the relation of $(\calD,\bL)$-neighbourhood
(i.e., $P \sim Q$ if and only if $P$ and $Q$ are $(\calD,\bL)$-neighbour endpoints),
and we denote by $\equiv$ the reflexive transitive closure of $\sim$.
Observe that, if $\calD$ is a $k$-generalised curve diagram, then the relation $\equiv$ has exactly $k$ equivalence classes.
Therefore, we aim below at counting coordinates $\bs\ba$ where the relation $\equiv$
has exactly one equivalence class; we will say that $\bs\ba$ are \emph{actual coordinates}.

Aiming to reduce the number of cases to look at, we will
use the symmetries mentioned in Section~\ref{section:from-laminations-to-graphs}.
If a braid $\beta$ has coordinates $(s_0,a_1,s_1,\ldots,a_n,s_n)$,
then its horizontally symmetric braid $\bS_h(\beta)$ has coordinates $(s_n,a_n,s_{n-1},\ldots,a_1,s_0)$ and
its vertically symmetric braid $\bS_v(\beta)$ has coordinates $(s_0,a'_1,s_1,\ldots,a'_n,s_n)$,
where $a'_i = 2 \min\{s_{i-1},s_i\} + \mathbf{1}_{s_{i-1} \neq s_i} - a_i$.

\subsection{An Introductory Example: The Braid Group $B_2$}
\label{subsection:toy-example-n=2}

In the braid group $B_2$, everything is obvious.
Indeed, the group $B_2$ is isomorphic to $\ZZ$, and generated by the Artin braid $\sigma_1$.
Since $\|\sigma_1^k\|_d = 1 + 2|k|$ for all integers $k \in \ZZ$, it follows that
\[g_{2,k} = \mathbf{1}_{k = 0} + 2 \cdot \mathbf{1}_{k \geq 1}, ~ \calG_2(z) = \frac{1+z}{1-z}, ~ \calB_2(z) = \frac{z (1+z^2)}{1-z^2}.\]

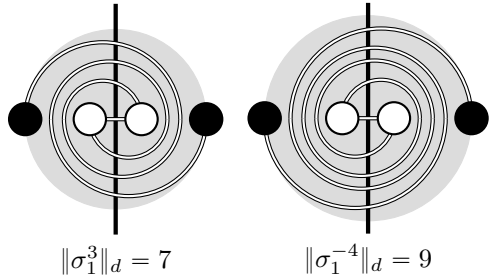
\begin{figure}[!ht]
\begin{center}
\begin{tikzpicture}[scale=0.17]
\draw[fill=palegray,draw=palegray] (7,0) circle (7);

\draw[draw=black,ultra thick] (7,-9) -- (7,9);

\draw[draw=black,ultra thick] (5,0) -- (9,0);

\draw[draw=black,ultra thick] (0,0) arc (180:0:6);
\draw[draw=black,ultra thick] (2,0) arc (180:0:4.5);
\draw[draw=black,ultra thick] (3,0) arc (180:0:3);

\draw[draw=black,ultra thick] (2,0) arc (180:360:6);
\draw[draw=black,ultra thick] (3,0) arc (180:360:4.5);
\draw[draw=black,ultra thick] (5,0) arc (180:360:3);

\draw[draw=white,thick] (5,0) -- (9,0);

\draw[draw=white,thick] (0,0) arc (180:0:6);
\draw[draw=white,thick] (2,0) arc (180:0:4.5);
\draw[draw=white,thick] (3,0) arc (180:0:3);

\draw[draw=white,thick] (2,0) arc (180:360:6);
\draw[draw=white,thick] (3,0) arc (180:360:4.5);
\draw[draw=white,thick] (5,0) arc (180:360:3);

\PUNCTURE{0}
\PUNCTURE{14}

\puncture{5}
\puncture{9}

\node at (7,-11) {$\|\sigma_1^3\|_d =7$};
\end{tikzpicture}
{\tiny~}
\begin{tikzpicture}[scale=0.17]
\draw[fill=palegray,draw=palegray] (8,0) circle (8);

\draw[draw=black,ultra thick] (8,-9) -- (8,9);

\draw[draw=black,ultra thick] (6,0) -- (10,0);

\draw[draw=black,ultra thick] (2,0) arc (180:0:7);
\draw[draw=black,ultra thick] (3,0) arc (180:0:5.5);
\draw[draw=black,ultra thick] (4,0) arc (180:0:4.5);
\draw[draw=black,ultra thick] (6,0) arc (180:0:3);

\draw[draw=black,ultra thick] (0,0) arc (180:360:7);
\draw[draw=black,ultra thick] (2,0) arc (180:360:5.5);
\draw[draw=black,ultra thick] (3,0) arc (180:360:4.5);
\draw[draw=black,ultra thick] (4,0) arc (180:360:3);

\draw[draw=white,thick] (6,0) -- (10,0);

\draw[draw=white,thick] (2,0) arc (180:0:7);
\draw[draw=white,thick] (3,0) arc (180:0:5.5);
\draw[draw=white,thick] (4,0) arc (180:0:4.5);
\draw[draw=white,thick] (6,0) arc (180:0:3);

\draw[draw=white,thick] (0,0) arc (180:360:7);
\draw[draw=white,thick] (2,0) arc (180:360:5.5);
\draw[draw=white,thick] (3,0) arc (180:360:4.5);
\draw[draw=white,thick] (4,0) arc (180:360:3);

\PUNCTURE{0}
\PUNCTURE{16}

\puncture{6}
\puncture{10}

\node at (8,-11) {$\|\sigma_1^{-4}\|_d =9$};
\end{tikzpicture}
\end{center}
\caption{$\|\sigma_1^k\|_d = 2k+1$}
\label{fig:norm-in-B2}
\end{figure}

Let us recover this result with the tools introduced above,
in particular the reflexive transitive closure (denoted by $\equiv$) of
the $(\calD,\bL)$-neighbourhood relation (denoted by $\sim$):
we detail computations as a warm-up.
A braid $\beta$ with norm $2k+1$ has coordinates
of the form $(0,a_1,k,a_2,0)$, with $k \geq 0$ and $a_1, a_2 \in \{0,1\}$.

If $k = 0$, then $\bs\ba = (0,0,0,0,0)$,
hence $\{\conc{0}{1},\conc{1}{1},\conc{2}{1}\}$ is the only equivalence class of $\equiv$.
It follows that $g_{2,0} = 0$.

We consider now the case $k \geq 1$.
Using the vertical symmetry, we focus on the case $a_1 \leq a_2$.
If $a_1 = 1$, then $a_2 = 1$, hence
$\{\conc{0}{1},\conc{1}{1},\conc{2}{1}\}$ is an equivalence class of $\equiv$ that does not contain $\conc{1}{2k+1}$.
Similarly, if $a_2 = 0$, then $a_1 = 0$, hence
$\{\conc{0}{1},\conc{1}{2k+1},\conc{2}{1}\}$ is an equivalence class of $\equiv$ that does not contain $\conc{1}{1}$.

Therefore, we must have $\bs\ba = (0,0,k,1,0)$ (or $\bs\ba = (0,1,k,0,0)$, but we decided to let this case aside for now).
Hence, Lemma~\ref{lem:left-right-box} proves that
\[\conc{0}{1} \sim \conc{1}{2k+1} \sim \conc{1}{2} \sim \conc{1}{2k-1} \sim \ldots \sim \conc{1}{3} \sim \conc{1}{2k-2} \sim \conc{1}{1} \sim \conc{2}{1}.\]
This proves that $(0,0,k,1,0)$ and $(0,1,k,0,0)$ are actual coordinates, and therefore that
$g_{2,k} = 2$.
We deduce from these values of $g_{2,k}$ the above expression of the functions $\calG_2(z) = \sum_{k \geq 0}z^k$ and $\calB_2(z) = z \calG_2(z^2)$.

This second proof is longer and more convoluted than the direct proof obtained by
enumerating the braids in the group $B_2$.
However, enumerating the braids in $B_3$ seems out of reach,
whereas considering virtual coordinates and identifying which are actual coordinates
will be possible, as shown in Section~\ref{subsection:challenging-example-n=3}.

\subsection{A Challenging Example: The Braid Group $B_3$}
\label{subsection:challenging-example-n=3}

Holonomic functions are univariate power series $f$ that satisfy some linear differential equation
\[\sum_{i=0}^k c_i(z) \frac{\partial^i}{\partial z^i}f(z) = 0,\]
where $c_0(z), \ldots, c_k(z)$ are complex polynomials.
This class of function generalises rational and algebraic functions, and is closed under various operations,
such as addition, multiplication, term-wise multiplication and
algebraic substitution (i.e. replacing the function $z \mapsto f(z)$ by some function $z \mapsto f(g(z))$
where $g$ is a solution of some equation $P(z,g(z)) = 0$, $P$ being some non-degenerate polynomial).

A short introduction on holonomic series and their use in analytic combinatorics,
including the associated tools for manipulating holonomic series,
can be found in~\cite[Annex B.4]{Flajolet:2009:AC:1506267}.

Our central result is the following one.

\begin{thm}
The integers $g_{3,k}$ and the generating functions $\calG_3(z)$ and $\calB_3(z)$ are given by:
\begin{eqnarray*}
\calG_3(z) & = & 2 \frac{1+2z-z^2}{z^2(1-z^2)} \left(\sum_{n \geq 3} \varphi(n) z^n\right) + \frac{1-3z^2}{1-z^2}, \\
\calB_3(z) & = & 2 \frac{1+2z^2-z^4}{z^2(1-z^4)} \left(\sum_{n \geq 3} \varphi(n) z^{2n}\right) + \frac{z^2(1-3z^4)}{1-z^4}, \\
g_{3,k} & = & \mathbf{1}_{k=0} + 2 \left(
\varphi(k+2) - \mathbf{1}_{k \in 2\ZZ} + 2 \sum_{i=1}^{\lfloor k/2 \rfloor} \varphi(k+3-2i) \right) \mathbf{1}_{k \geq 1},
\end{eqnarray*}
where $\varphi$ denotes the Euler totient. The functions $\calG_3(z)$ and $\calB_3(z)$ are not holonomic.
\label{thm:G-and-L3}
\end{thm}

In particular, the functions $\calG_3(z)$ and $\calB_3(z)$ are neither rational nor algebraic.
Furthermore, and since the ordinary generating function $\calG_3(z) = \sum_{k \geq 0}g_{3,k}z^k$ is not holonomic,
neither are the exponential generating function $\sum_{k \geq 0}\frac{g_{3,k}}{k!}z^k$
nor the Poisson generating function $\exp(-z) \sum_{k \geq 0}\frac{g_{3,k}}{k!}z^k$.

The following sections are devoted to the proof of Theorem~\ref{thm:G-and-L3},
which is done in five steps.

\subsubsection{Proof of Theorem~\ref{thm:G-and-L3} -- Step 1: Simple Cases}
\label{subsubsection:proof-1}

Following the preliminary remarks we made at the beginning of this section,
we focus on virtual coordinates the form
$\bs\ba = (0,a_1,k,a_2,\ell,a_3,0)$ associated to some generalised curve diagram $\calD$
whose induced relation $\equiv$ has an unique equivalence class.
Henceforth, we consider $\ell$ and $k$ as parameters, and compute the integer
\[C_{k,\ell} = |\{(a_1,a_2,a_3) : (0,a_1,k,a_2,\ell,a_3,0) \text{ are actual coordinates}\}|.\]
Using the vertical symmetry, we know that $C_{k,\ell} = C_{\ell,k}$.
We focus below on the case where $\ell \leq k$,
and proceed to a disjunction of cases.

First, if $k = \ell = 0$, then $\bs\ba = (0,0,0,0,0,0,0)$ are the coordinates of the trivial braid $\varepsilon \in B_3$.
Therefore, $C_{0,0} = 1$.

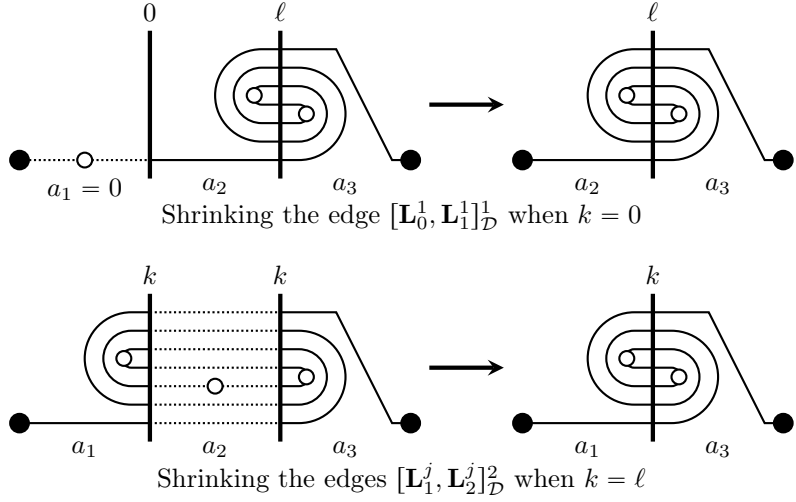
\begin{figure}[!ht]
\begin{center}
\begin{tikzpicture}[scale=0.245]
\draw[draw=black,ultra thick] (7,-1) -- (7,7);
\draw[draw=black,ultra thick] (14,-1) -- (14,7);

\draw[draw=black,fill=black,thick] (0,0) circle (0.5);
\draw[draw=black,fill=black,thick] (21,0) circle (0.5);

\draw[draw=black,thick,densely dotted] (0,0) -- (7,0);

\draw[draw=black,thick] (7,0) -- (14,0);
\draw[draw=black,thick] (14,1) -- (13,1) arc (270:90:2.5) -- (14,6);
\draw[draw=black,thick] (14,2) -- (13,2) arc (270:90:1.5) -- (14,5);
\draw[draw=black,thick] (14,3) -- (13,3) arc (270:90:0.5) -- (14,4);

\draw[draw=black,thick] (14,0) -- (15,0) arc (-90:90:2.5) -- (14,5);
\draw[draw=black,thick] (14,1) -- (15,1) arc (-90:90:1.5) -- (14,4);
\draw[draw=black,thick] (14,2) -- (15,2) arc (-90:90:0.5) -- (14,3);
\draw[draw=black,thick] (14,6) -- (17,6) -- (20,0) -- (21,0);

\draw[draw=black,fill=white,thick] (3.5,0) circle (0.4);
\draw[draw=black,fill=white,thick] (12.6,3.5) circle (0.4);
\draw[draw=black,fill=white,thick] (15.4,2.5) circle (0.4);

\node[anchor=north] at (3.5,-0.5) {$a_1 = 0$};
\node[anchor=south] at (7,7) {$0$};
\node[anchor=north] at (10.5,-0.5) {$a_2$};
\node[anchor=south] at (14,7) {$\ell$};
\node[anchor=north] at (17.5,-0.5) {$a_3$};

\draw[draw=black,ultra thick,>=stealth, ->] (22,3) -- (26,3);

\draw[draw=black,ultra thick] (34,-1) -- (34,7);

\draw[draw=black,fill=black,thick] (27,0) circle (0.5);
\draw[draw=black,fill=black,thick] (41,0) circle (0.5);

\draw[draw=black,thick] (27,0) -- (34,0);
\draw[draw=black,thick] (34,1) -- (33,1) arc (270:90:2.5) -- (34,6);
\draw[draw=black,thick] (34,2) -- (33,2) arc (270:90:1.5) -- (34,5);
\draw[draw=black,thick] (34,3) -- (33,3) arc (270:90:0.5) -- (34,4);

\draw[draw=black,thick] (34,0) -- (35,0) arc (-90:90:2.5) -- (34,5);
\draw[draw=black,thick] (34,1) -- (35,1) arc (-90:90:1.5) -- (34,4);
\draw[draw=black,thick] (34,2) -- (35,2) arc (-90:90:0.5) -- (34,3);
\draw[draw=black,thick] (34,6) -- (37,6) -- (40,0) -- (41,0);

\draw[draw=black,fill=white,thick] (32.6,3.5) circle (0.4);
\draw[draw=black,fill=white,thick] (35.4,2.5) circle (0.4);

\node[anchor=north] at (30.5,-0.5) {$a_2$};
\node[anchor=south] at (34,7) {$\ell$};
\node[anchor=north] at (37.5,-0.5) {$a_3$};

\node at (20.5,-3) {Shrinking the edge $[\conc{0}{1}, \conc{1}{1}]_\calD^1$ when $k = 0$};
\end{tikzpicture}

\medskip

\begin{tikzpicture}[scale=0.245]
\draw[draw=black,ultra thick] (7,-1) -- (7,7);
\draw[draw=black,ultra thick] (14,-1) -- (14,7);

\draw[draw=black,fill=black,thick] (0,0) circle (0.5);
\draw[draw=black,fill=black,thick] (21,0) circle (0.5);

\draw[draw=black,thick] (0,0) -- (7,0);
\draw[draw=black,thick] (7,1) -- (6,1) arc (270:90:2.5) -- (7,6);
\draw[draw=black,thick] (7,2) -- (6,2) arc (270:90:1.5) -- (7,5);
\draw[draw=black,thick] (7,3) -- (6,3) arc (270:90:0.5) -- (7,4);

\draw[draw=black,thick,densely dotted] (7,0) -- (14,0);
\draw[draw=black,thick,densely dotted] (7,1) -- (14,1);
\draw[draw=black,thick,densely dotted] (7,2) -- (14,2);
\draw[draw=black,thick,densely dotted] (7,3) -- (14,3);
\draw[draw=black,thick,densely dotted] (7,4) -- (14,4);
\draw[draw=black,thick,densely dotted] (7,5) -- (14,5);
\draw[draw=black,thick,densely dotted] (7,6) -- (14,6);

\draw[draw=black,thick] (14,0) -- (15,0) arc (-90:90:2.5) -- (14,5);
\draw[draw=black,thick] (14,1) -- (15,1) arc (-90:90:1.5) -- (14,4);
\draw[draw=black,thick] (14,2) -- (15,2) arc (-90:90:0.5) -- (14,3);
\draw[draw=black,thick] (14,6) -- (17,6) -- (20,0) -- (21,0);

\draw[draw=black,fill=white,thick] (5.6,3.5) circle (0.4);
\draw[draw=black,fill=white,thick] (10.5,2) circle (0.4);
\draw[draw=black,fill=white,thick] (15.4,2.5) circle (0.4);

\node[anchor=north] at (3.5,-0.5) {$a_1$};
\node[anchor=south] at (7,7) {$k$};
\node[anchor=north] at (10.5,-0.5) {$a_2$};
\node[anchor=south] at (14,7) {$k$};
\node[anchor=north] at (17.5,-0.5) {$a_3$};

\draw[draw=black,ultra thick,>=stealth,->] (22,3) -- (26,3);

\draw[draw=black,ultra thick] (34,-1) -- (34,7);

\draw[draw=black,fill=black,thick] (27,0) circle (0.5);
\draw[draw=black,fill=black,thick] (41,0) circle (0.5);

\draw[draw=black,thick] (27,0) -- (34,0);
\draw[draw=black,thick] (34,1) -- (33,1) arc (270:90:2.5) -- (34,6);
\draw[draw=black,thick] (34,2) -- (33,2) arc (270:90:1.5) -- (34,5);
\draw[draw=black,thick] (34,3) -- (33,3) arc (270:90:0.5) -- (34,4);

\draw[draw=black,thick] (34,0) -- (35,0) arc (-90:90:2.5) -- (34,5);
\draw[draw=black,thick] (34,1) -- (35,1) arc (-90:90:1.5) -- (34,4);
\draw[draw=black,thick] (34,2) -- (35,2) arc (-90:90:0.5) -- (34,3);
\draw[draw=black,thick] (34,6) -- (37,6) -- (40,0) -- (41,0);

\draw[draw=black,fill=white,thick] (32.6,3.5) circle (0.4);
\draw[draw=black,fill=white,thick] (35.4,2.5) circle (0.4);

\node[anchor=north] at (30.5,-0.5) {$a_1$};
\node[anchor=south] at (34,7) {$k$};
\node[anchor=north] at (37.5,-0.5) {$a_3$};

\node at (20.5,-3) {Shrinking the edges $[\conc{1}{j},\conc{2}{j}]_\calD^2$ when $k = \ell$};
\end{tikzpicture}
\end{center}
\caption{Shrinking edges of tight generalised curve diagrams when $k = 0$ and $k = \ell$}
\label{fig:collapse}
\end{figure}

If $k = 0 < \ell$, then $a_1 = 0$, and $\bs\ba$ are actual coordinates if and only if
$(0,a_2,\ell,a_3,0)$ are also actual coordinates.
Indeed, as illustrated by Fig.~\ref{fig:collapse},
the virtual coordinates $(0,a_2,\ell,a_3,0)$ can be obtained from $(0,a_1,0,a_2,\ell,a_3,0)$
by ``shrinking'' the edge $[\conc{0}{1}, \conc{1}{1}]_\calD^1$.
Hence, $\bs\ba$ are actual coordinates if and only if $\{a_2,a_3\} = \{0,1\}$.
It follows that $C_{0,\ell} = C_{\ell,0} = 2$.

Similarly, if $1 \leq k = \ell$, then $\bs\ba$ are actual coordinates if and only if $0 \leq a_2 \leq 2k$ and if
$(0,a_1,k,a_3,0)$ are also actual coordinates:
as illustrated by Fig.~\ref{fig:collapse},
the virtual coordinates $(0,a_1,k,a_3,0)$ can be obtained from $(0,a_1,k,a_2,k,a_3,0)$ by
``shrinking'' each edge $[\conc{1}{j}, \conc{2}{j}]_\calD^2$, when $1 \leq j \leq 2k+1$.
We therefore have $2k+1$ ways of choosing $a_2$ and $2$ ways of choosing $(a_1,a_3)$, which proves that $C_{k,k} = 2(2k+1)$.

\subsubsection{Proof of Theorem~\ref{thm:G-and-L3} -- Step 2: Towards Cyclic Permutations}
\label{subsubsection:proof-2}

We consider now the case where $1 \leq k < \ell$.
Using the \emph{horizontal} symmetry,
we may focus on the case where $a_1 = 1$: doing so, we will find exactly half of the
actual coordinates $(0,a_1,k,a_2,\ell,a_3,0)$.

In order to ease subsequent computations, we decide here to modify slightly the generalised curve diagram $\calD$ we drew from the coordinates $\bs\ba$,
as illustrated in Fig.~\ref{fig:one-more-curve}.
We proceed as follows:
\begin{itemize}
 \item we add points $\conc{1}{2k+2}$ and $\conc{2}{2\ell+2}$ on the lines $\bL_1$ and $\bL_2$, \emph{above} the points $\conc{1}{2k+1}$ and $\conc{2}{2\ell+1}$;
 \item we draw a curve (drawn in gray in Fig.~\ref{fig:one-more-curve}) from $\conc{0}{1}$ to $\conc{3}{1}$,
 that does not cross the other curves of $\calD$,
 and that crosses the lines $\bL_1$ and $\bL_2$ at $\conc{1}{2k+2}$ and $\conc{2}{2\ell+2}$.
\end{itemize}
Informally, we decided to ``close by above'' the unique open curve contained in $\calD$.
Recall that $\sim$ denotes the $(\calD,\bL)$-neighbourhood relation, and that $\equiv$ denotes the
reflexive transitive closure of $\sim$.
What we just did was to add the relations $\conc{0}{1} \sim \conc{1}{2k+2} \sim \conc{2}{2\ell+2} \sim \conc{3}{1}$.
Since $\calD$ already contained an open curve with endpoints $\conc{0}{1}$ and $\conc{3}{1}$,
adding these points, curves and relations did not change the number of equivalence classes of the relation $\equiv$.

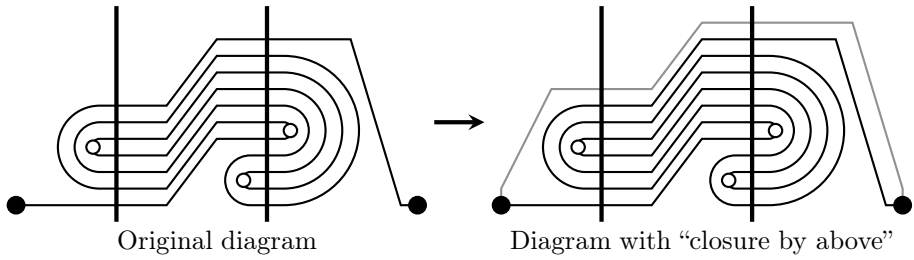
\begin{figure}[!ht]
\begin{center}
\begin{tikzpicture}[scale=0.22]
\draw[draw=black,ultra thick] (35,-1) -- (35,12);
\draw[draw=black,ultra thick] (44,-1) -- (44,12);

\draw[draw=black,fill=black,thick] (29,0) circle (0.5);
\draw[draw=black,fill=black,thick] (53,0) circle (0.5);

\draw[draw=black,thick] (29,0) -- (35,0);
\draw[draw=black,thick] (35,1) -- (34,1) arc (270:90:2.5) -- (35,6);
\draw[draw=black,thick] (35,2) -- (34,2) arc (270:90:1.5) -- (35,5);
\draw[draw=black,thick] (35,3) -- (34,3) arc (270:90:0.5) -- (35,4);

\draw[draw=black,thick] (44,0) -- (43,0) arc (270:90:1.5) -- (44,3);
\draw[draw=black,thick] (44,1) -- (43,1) arc (270:90:0.5) -- (44,2);
\draw[draw=black,thick] (35,0) -- (38,0) -- (41,4) -- (44,4);
\draw[draw=black,thick] (35,1) -- (38,1) -- (41,5) -- (44,5);
\draw[draw=black,thick] (35,2) -- (38,2) -- (41,6) -- (44,6);
\draw[draw=black,thick] (35,3) -- (38,3) -- (41,7) -- (44,7);
\draw[draw=black,thick] (35,4) -- (38,4) -- (41,8) -- (44,8);
\draw[draw=black,thick] (35,5) -- (38,5) -- (41,9) -- (44,9);
\draw[draw=black,thick] (35,6) -- (38,6) -- (41,10) -- (44,10);

\draw[draw=black,thick] (44,0) -- (45,0) arc (-90:90:4.5) -- (44,9);
\draw[draw=black,thick] (44,1) -- (45,1) arc (-90:90:3.5) -- (44,8);
\draw[draw=black,thick] (44,2) -- (45,2) arc (-90:90:2.5) -- (44,7);
\draw[draw=black,thick] (44,3) -- (45,3) arc (-90:90:1.5) -- (44,6);
\draw[draw=black,thick] (44,4) -- (45,4) arc (-90:90:0.5) -- (44,5);
\draw[draw=black,thick] (44,10) -- (49,10) -- (52,0) -- (53,0);

\draw[draw=black,fill=white,thick] (33.6,3.5) circle (0.4);
\draw[draw=black,fill=white,thick] (42.6,1.5) circle (0.4);
\draw[draw=black,fill=white,thick] (45.4,4.5) circle (0.4);

\node[anchor=north] at (41,-1) {Original diagram};

\draw[draw=black,ultra thick,>=stealth,->] (54,5) -- (57,5);

\draw[draw=blackgray,thick] (58,0) -- (58,1) -- (61,7) -- (67,7) -- (70,11) -- (79,11) -- (82,1) -- (82,0);

\draw[draw=black,ultra thick] (64,-1) -- (64,12);
\draw[draw=black,ultra thick] (73,-1) -- (73,12);

\draw[draw=black,fill=black,thick] (58,0) circle (0.5);
\draw[draw=black,fill=black,thick] (82,0) circle (0.5);

\draw[draw=black,thick] (58,0) -- (64,0);
\draw[draw=black,thick] (64,1) -- (63,1) arc (270:90:2.5) -- (64,6);
\draw[draw=black,thick] (64,2) -- (63,2) arc (270:90:1.5) -- (64,5);
\draw[draw=black,thick] (64,3) -- (63,3) arc (270:90:0.5) -- (64,4);

\draw[draw=black,thick] (73,0) -- (72,0) arc (270:90:1.5) -- (73,3);
\draw[draw=black,thick] (73,1) -- (72,1) arc (270:90:0.5) -- (73,2);
\draw[draw=black,thick] (64,0) -- (67,0) -- (70,4) -- (73,4);
\draw[draw=black,thick] (64,1) -- (67,1) -- (70,5) -- (73,5);
\draw[draw=black,thick] (64,2) -- (67,2) -- (70,6) -- (73,6);
\draw[draw=black,thick] (64,3) -- (67,3) -- (70,7) -- (73,7);
\draw[draw=black,thick] (64,4) -- (67,4) -- (70,8) -- (73,8);
\draw[draw=black,thick] (64,5) -- (67,5) -- (70,9) -- (73,9);
\draw[draw=black,thick] (64,6) -- (67,6) -- (70,10) -- (73,10);

\draw[draw=black,thick] (73,0) -- (74,0) arc (-90:90:4.5) -- (73,9);
\draw[draw=black,thick] (73,1) -- (74,1) arc (-90:90:3.5) -- (73,8);
\draw[draw=black,thick] (73,2) -- (74,2) arc (-90:90:2.5) -- (73,7);
\draw[draw=black,thick] (73,3) -- (74,3) arc (-90:90:1.5) -- (73,6);
\draw[draw=black,thick] (73,4) -- (74,4) arc (-90:90:0.5) -- (73,5);
\draw[draw=black,thick] (73,10) -- (78,10) -- (81,0) -- (82,0);

\draw[draw=black,fill=white,thick] (62.6,3.5) circle (0.4);
\draw[draw=black,fill=white,thick] (71.6,1.5) circle (0.4);
\draw[draw=black,fill=white,thick] (74.4,4.5) circle (0.4);

\node[anchor=north] at (70,-1) {Diagram with ``closure by above''};
\end{tikzpicture}
\end{center}
\caption{Closing the open curve of $\calD$ by above}
\label{fig:one-more-curve}
\end{figure}

From now on, and in the rest of the proof of Theorem~\ref{thm:G-and-L3},
we will only use such ``closed by above'' generalised diagrams,
and we will identify $\calD$ with this ``closed by above'' version.

Then, let us define the integer $m := \ell-k$. Since $0 \leq k \leq \ell$, observe that:
\begin{itemize}
\item $\conc{0}{1} \sim \conc{1}{2k+2} \sim \conc{2}{2\ell+2} \sim \conc{3}{1}$;
\item $\conc{1}{j} \sim \conc{2}{j}$ for all $j \in \{1,\ldots,a_2\}$;
\item $\conc{1}{j} \sim \conc{2}{j+2m}$ for all $j \in \{a_2+1,\ldots,2k+1\}$.
\end{itemize}

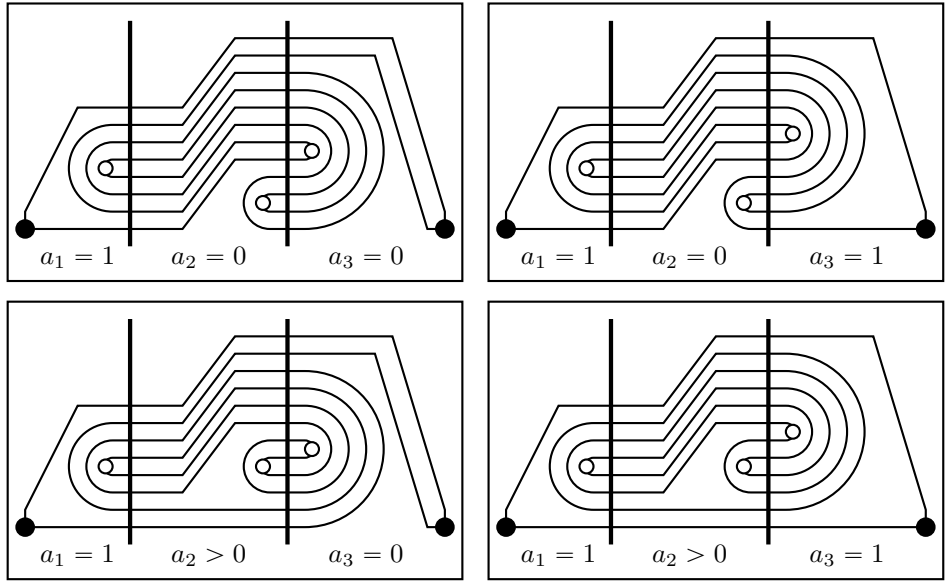
\begin{figure}[!ht]
\begin{center}
\begin{tikzpicture}[scale=0.23]
\draw[draw=black,ultra thick] (35,-1) -- (35,12);
\draw[draw=black,ultra thick] (44,-1) -- (44,12);

\draw[draw=black,fill=black,thick] (29,0) circle (0.5);
\draw[draw=black,fill=black,thick] (53,0) circle (0.5);

\draw[draw=black,thick] (29,0) -- (35,0);
\draw[draw=black,thick] (35,1) -- (34,1) arc (270:90:2.5) -- (35,6);
\draw[draw=black,thick] (35,2) -- (34,2) arc (270:90:1.5) -- (35,5);
\draw[draw=black,thick] (35,3) -- (34,3) arc (270:90:0.5) -- (35,4);

\draw[draw=black,thick] (44,0) -- (43,0) arc (270:90:1.5) -- (44,3);
\draw[draw=black,thick] (44,1) -- (43,1) arc (270:90:0.5) -- (44,2);
\draw[draw=black,thick] (35,0) -- (38,0) -- (41,4) -- (44,4);
\draw[draw=black,thick] (35,1) -- (38,1) -- (41,5) -- (44,5);
\draw[draw=black,thick] (35,2) -- (38,2) -- (41,6) -- (44,6);
\draw[draw=black,thick] (35,3) -- (38,3) -- (41,7) -- (44,7);
\draw[draw=black,thick] (35,4) -- (38,4) -- (41,8) -- (44,8);
\draw[draw=black,thick] (35,5) -- (38,5) -- (41,9) -- (44,9);
\draw[draw=black,thick] (35,6) -- (38,6) -- (41,10) -- (44,10);

\draw[draw=black,thick] (44,0) -- (45,0) arc (-90:90:4.5) -- (44,9);
\draw[draw=black,thick] (44,1) -- (45,1) arc (-90:90:3.5) -- (44,8);
\draw[draw=black,thick] (44,2) -- (45,2) arc (-90:90:2.5) -- (44,7);
\draw[draw=black,thick] (44,3) -- (45,3) arc (-90:90:1.5) -- (44,6);
\draw[draw=black,thick] (44,4) -- (45,4) arc (-90:90:0.5) -- (44,5);
\draw[draw=black,thick] (44,10) -- (49,10) -- (52,0) -- (53,0);

\draw[draw=black,thick] (29,0) -- (29,1) -- (32,7) -- (38,7) -- (41,11) -- (50,11) -- (53,1) -- (53,0);

\draw[draw=black,fill=white,thick] (33.6,3.5) circle (0.4);
\draw[draw=black,fill=white,thick] (42.6,1.5) circle (0.4);
\draw[draw=black,fill=white,thick] (45.4,4.5) circle (0.4);

\node[anchor=north] at (32,-0.5) {$a_1 = 1$};
\node[anchor=north] at (39.5,-0.5) {$a_2 = 0$};
\node[anchor=north] at (48.5,-0.5) {$a_3 = 0$};

\draw[draw=black,thick] (28,-3) -- (54,-3) -- (54,13) -- (28,13) -- cycle;
\end{tikzpicture}
{\tiny~}
\begin{tikzpicture}[scale=0.23]
\draw[draw=black,ultra thick] (35,-1) -- (35,12);
\draw[draw=black,ultra thick] (44,-1) -- (44,12);

\draw[draw=black,fill=black,thick] (29,0) circle (0.5);
\draw[draw=black,fill=black,thick] (53,0) circle (0.5);

\draw[draw=black,thick] (29,0) -- (35,0);
\draw[draw=black,thick] (35,1) -- (34,1) arc (270:90:2.5) -- (35,6);
\draw[draw=black,thick] (35,2) -- (34,2) arc (270:90:1.5) -- (35,5);
\draw[draw=black,thick] (35,3) -- (34,3) arc (270:90:0.5) -- (35,4);

\draw[draw=black,thick] (44,0) -- (43,0) arc (270:90:1.5) -- (44,3);
\draw[draw=black,thick] (44,1) -- (43,1) arc (270:90:0.5) -- (44,2);
\draw[draw=black,thick] (35,0) -- (38,0) -- (41,4) -- (44,4);
\draw[draw=black,thick] (35,1) -- (38,1) -- (41,5) -- (44,5);
\draw[draw=black,thick] (35,2) -- (38,2) -- (41,6) -- (44,6);
\draw[draw=black,thick] (35,3) -- (38,3) -- (41,7) -- (44,7);
\draw[draw=black,thick] (35,4) -- (38,4) -- (41,8) -- (44,8);
\draw[draw=black,thick] (35,5) -- (38,5) -- (41,9) -- (44,9);
\draw[draw=black,thick] (35,6) -- (38,6) -- (41,10) -- (44,10);

\draw[draw=black,thick] (44,1) -- (45,1) arc (-90:90:4.5) -- (44,10);
\draw[draw=black,thick] (44,2) -- (45,2) arc (-90:90:3.5) -- (44,9);
\draw[draw=black,thick] (44,3) -- (45,3) arc (-90:90:2.5) -- (44,8);
\draw[draw=black,thick] (44,4) -- (45,4) arc (-90:90:1.5) -- (44,7);
\draw[draw=black,thick] (44,5) -- (45,5) arc (-90:90:0.5) -- (44,6);
\draw[draw=black,thick] (44,0) -- (53,0);

\draw[draw=black,thick] (29,0) -- (29,1) -- (32,7) -- (38,7) -- (41,11) -- (50,11) -- (53,1) -- (53,0);

\draw[draw=black,fill=white,thick] (33.6,3.5) circle (0.4);
\draw[draw=black,fill=white,thick] (42.6,1.5) circle (0.4);
\draw[draw=black,fill=white,thick] (45.4,5.5) circle (0.4);

\node[anchor=north] at (32,-0.5) {$a_1 = 1$};
\node[anchor=north] at (39.5,-0.5) {$a_2 = 0$};
\node[anchor=north] at (48.5,-0.5) {$a_3 = 1$};

\draw[draw=black,thick] (28,-3) -- (54,-3) -- (54,13) -- (28,13) -- cycle;
\end{tikzpicture}

\medskip

\begin{tikzpicture}[scale=0.23]
\draw[draw=black,ultra thick] (5,-1) -- (5,12);
\draw[draw=black,ultra thick] (14,-1) -- (14,12);

\draw[draw=black,fill=black,thick] (-1,0) circle (0.5);
\draw[draw=black,fill=black,thick] (23,0) circle (0.5);

\draw[draw=black,thick] (-1,0) -- (5,0);
\draw[draw=black,thick] (5,1) -- (4,1) arc (270:90:2.5) -- (5,6);
\draw[draw=black,thick] (5,2) -- (4,2) arc (270:90:1.5) -- (5,5);
\draw[draw=black,thick] (5,3) -- (4,3) arc (270:90:0.5) -- (5,4);

\draw[draw=black,thick] (5,0) -- (14,0);
\draw[draw=black,thick] (5,1) -- (14,1);
\draw[draw=black,thick] (14,2) -- (13,2) arc (270:90:1.5) -- (14,5);
\draw[draw=black,thick] (14,3) -- (13,3) arc (270:90:0.5) -- (14,4);
\draw[draw=black,thick] (5,2) -- (8,2) -- (11,6) -- (14,6);
\draw[draw=black,thick] (5,3) -- (8,3) -- (11,7) -- (14,7);
\draw[draw=black,thick] (5,4) -- (8,4) -- (11,8) -- (14,8);
\draw[draw=black,thick] (5,5) -- (8,5) -- (11,9) -- (14,9);
\draw[draw=black,thick] (5,6) -- (8,6) -- (11,10) -- (14,10);

\draw[draw=black,thick] (14,0) -- (15,0) arc (-90:90:4.5) -- (14,9);
\draw[draw=black,thick] (14,1) -- (15,1) arc (-90:90:3.5) -- (14,8);
\draw[draw=black,thick] (14,2) -- (15,2) arc (-90:90:2.5) -- (14,7);
\draw[draw=black,thick] (14,3) -- (15,3) arc (-90:90:1.5) -- (14,6);
\draw[draw=black,thick] (14,4) -- (15,4) arc (-90:90:0.5) -- (14,5);
\draw[draw=black,thick] (14,10) -- (19,10) -- (22,0) -- (23,0);

\draw[draw=black,thick] (-1,0) -- (-1,1) -- (2,7) -- (8,7) -- (11,11) -- (20,11) -- (23,1) -- (23,0);

\draw[draw=black,fill=white,thick] (3.6,3.5) circle (0.4);
\draw[draw=black,fill=white,thick] (12.6,3.5) circle (0.4);
\draw[draw=black,fill=white,thick] (15.4,4.5) circle (0.4);

\node[anchor=north] at (2,-0.5) {$a_1 = 1$};
\node[anchor=north] at (9.5,-0.5) {$a_2 > 0$};
\node[anchor=north] at (18.5,-0.5) {$a_3 = 0$};

\draw[draw=black,thick] (-2,-3) -- (24,-3) -- (24,13) -- (-2,13) -- cycle;
\end{tikzpicture}
{\tiny~}
\begin{tikzpicture}[scale=0.23]
\draw[draw=black,ultra thick] (5,-1) -- (5,12);
\draw[draw=black,ultra thick] (14,-1) -- (14,12);

\draw[draw=black,fill=black,thick] (-1,0) circle (0.5);
\draw[draw=black,fill=black,thick] (23,0) circle (0.5);

\draw[draw=black,thick] (-1,0) -- (5,0);
\draw[draw=black,thick] (5,1) -- (4,1) arc (270:90:2.5) -- (5,6);
\draw[draw=black,thick] (5,2) -- (4,2) arc (270:90:1.5) -- (5,5);
\draw[draw=black,thick] (5,3) -- (4,3) arc (270:90:0.5) -- (5,4);

\draw[draw=black,thick] (5,0) -- (14,0);
\draw[draw=black,thick] (5,1) -- (14,1);
\draw[draw=black,thick] (14,2) -- (13,2) arc (270:90:1.5) -- (14,5);
\draw[draw=black,thick] (14,3) -- (13,3) arc (270:90:0.5) -- (14,4);
\draw[draw=black,thick] (5,2) -- (8,2) -- (11,6) -- (14,6);
\draw[draw=black,thick] (5,3) -- (8,3) -- (11,7) -- (14,7);
\draw[draw=black,thick] (5,4) -- (8,4) -- (11,8) -- (14,8);
\draw[draw=black,thick] (5,5) -- (8,5) -- (11,9) -- (14,9);
\draw[draw=black,thick] (5,6) -- (8,6) -- (11,10) -- (14,10);

\draw[draw=black,thick] (14,1) -- (15,1) arc (-90:90:4.5) -- (14,10);
\draw[draw=black,thick] (14,2) -- (15,2) arc (-90:90:3.5) -- (14,9);
\draw[draw=black,thick] (14,3) -- (15,3) arc (-90:90:2.5) -- (14,8);
\draw[draw=black,thick] (14,4) -- (15,4) arc (-90:90:1.5) -- (14,7);
\draw[draw=black,thick] (14,5) -- (15,5) arc (-90:90:0.5) -- (14,6);
\draw[draw=black,thick] (14,0) -- (23,0);

\draw[draw=black,thick] (-1,0) -- (-1,1) -- (2,7) -- (8,7) -- (11,11) -- (20,11) -- (23,1) -- (23,0);

\draw[draw=black,fill=white,thick] (3.6,3.5) circle (0.4);
\draw[draw=black,fill=white,thick] (12.6,3.5) circle (0.4);
\draw[draw=black,fill=white,thick] (15.4,5.5) circle (0.4);

\node[anchor=north] at (2,-0.5) {$a_1 = 1$};
\node[anchor=north] at (9.5,-0.5) {$a_2 > 0$};
\node[anchor=north] at (18.5,-0.5) {$a_3 = 1$};

\draw[draw=black,thick] (-2,-3) -- (24,-3) -- (24,13) -- (-2,13) -- cycle;
\end{tikzpicture}
\end{center}
\caption{Four different cases: $a_1 = 1$, $a_2 \stackrel{?}{=} 0$ and $a_3 \stackrel{?}{=} 0$}
\label{fig:a2-a3-four-cases}
\end{figure}

Hence, each equivalence class of the relation $\equiv$
contains points of the type $\conc{2}{m}$, as illustrated by Figure~\ref{fig:a2-a3-four-cases}.

We then define additional relations on the set $\{1,\ldots,2\ell+2\}$.
Let $u$ and $v$ be elements of $\{0,\ldots,2\ell+1\}$.
We write $u \stackrel{\sqsubset}{\sim} v$ if some
connected component of $\calD \setminus \bL_2$ has endpoints $\conc{2}{u}$ and $\conc{2}{v}$
and lies to the left of $\bL_2$ (i.e. in the area $\calZ_1 \cup \calZ_2$).
Similarly, we write $u \stackrel{\sqsupset}{\sim} v$ if some
connected component of $\calD \setminus \bL_2$ has endpoints $\conc{2}{u}$ and $\conc{2}{v}$
and lies to the right of $\bL_2$ (i.e. in the area $\calZ_3$).

Alternatively, one might define the relations $\stackrel{\sqsubset}{\sim}$ and $\stackrel{\sqsupset}{\sim}$ by saying that
$u \stackrel{\sqsubset}{\sim} v$ whenever 
$\conc{2}{u} \stackrel{2}{\sim} \conc{2}{v}$,
$\conc{2}{u} \stackrel{2}{\sim} \conc{1}{w} \stackrel{1}{\sim} \conc{1}{x} \stackrel{2}{\sim} \conc{2}{v}$ or
$\conc{2}{u} \stackrel{2}{\sim} \conc{1}{w} \stackrel{1}{\sim} \conc{0}{1} \stackrel{1}{\sim} \conc{1}{x} \stackrel{2}{\sim} \conc{2}{v}$
for some $w$, $x$,
and that $u \stackrel{\sqsupset}{\sim} v$ whenever
$\conc{2}{u} \stackrel{3}{\sim} \conc{2}{v}$ or $\conc{2}{u} \stackrel{3}{\sim} \conc{3}{1} \stackrel{3}{\sim} \conc{2}{v}$.

One checks easily that, whenever $u \stackrel{\sqsubset}{\sim} v$ or $u \stackrel{\sqsupset}{\sim} v$,
the integers $u$ and $v$ have different parities.
Hence, consider the permutation $\theta$ of $\{0,\ldots,\ell\}$ such that
$\theta(u) = v$ if and only if there exists some (even) integer $w \in \{1,\ldots,2\ell+2\}$ such that
$2u+1 \stackrel{\sqsubset}{\sim} w \stackrel{\sqsupset}{\sim} 2v+1$.
By construction, there is a bijection between the equivalence classes of the relation $\equiv$ and
the orbits of $\theta$, as follows:
we identify the equivalence class $\calC$ of $\equiv$ with the orbit $\{u : \conc{2}{2u+1} \in \calC\}$ of $\theta$.

\begin{figure}[!ht]
\begin{center}
\begin{tikzpicture}[scale=0.22]
\draw[draw=black,thick] (58,0) -- (58,1) -- (61,7) -- (67,7) -- (70,11) -- (79,11) -- (82,1) -- (82,0);

\draw[draw=black,ultra thick] (64,-1) -- (64,12);
\draw[draw=black,ultra thick] (73,-1) -- (73,12);

\draw[draw=black,fill=black,thick] (58,0) circle (0.5);
\draw[draw=black,fill=black,thick] (82,0) circle (0.5);

\draw[draw=black,thick] (58,0) -- (64,0);
\draw[draw=black,thick] (64,1) -- (63,1) arc (270:90:2.5) -- (64,6);
\draw[draw=black,thick] (64,2) -- (63,2) arc (270:90:1.5) -- (64,5);
\draw[draw=black,thick] (64,3) -- (63,3) arc (270:90:0.5) -- (64,4);

\draw[draw=black,thick] (73,0) -- (72,0) arc (270:90:1.5) -- (73,3);
\draw[draw=black,thick] (73,1) -- (72,1) arc (270:90:0.5) -- (73,2);
\draw[draw=black,thick] (64,0) -- (67,0) -- (70,4) -- (73,4);
\draw[draw=black,thick] (64,1) -- (67,1) -- (70,5) -- (73,5);
\draw[draw=black,thick] (64,2) -- (67,2) -- (70,6) -- (73,6);
\draw[draw=black,thick] (64,3) -- (67,3) -- (70,7) -- (73,7);
\draw[draw=black,thick] (64,4) -- (67,4) -- (70,8) -- (73,8);
\draw[draw=black,thick] (64,5) -- (67,5) -- (70,9) -- (73,9);
\draw[draw=black,thick] (64,6) -- (67,6) -- (70,10) -- (73,10);

\draw[draw=black,thick] (73,0) -- (74,0) arc (-90:90:4.5) -- (73,9);
\draw[draw=black,thick] (73,1) -- (74,1) arc (-90:90:3.5) -- (73,8);
\draw[draw=black,thick] (73,2) -- (74,2) arc (-90:90:2.5) -- (73,7);
\draw[draw=black,thick] (73,3) -- (74,3) arc (-90:90:1.5) -- (73,6);
\draw[draw=black,thick] (73,4) -- (74,4) arc (-90:90:0.5) -- (73,5);
\draw[draw=black,thick] (73,10) -- (78,10) -- (81,0) -- (82,0);

\draw[draw=black,fill=white,thick] (62.6,3.5) circle (0.4);
\draw[draw=black,fill=white,thick] (71.6,1.5) circle (0.4);
\draw[draw=black,fill=white,thick] (74.4,4.5) circle (0.4);

\node[anchor=west] at (82.5,10) {$1 \stackrel{\sqsubset}{\sim} 4$, $2 \stackrel{\sqsubset}{\sim} 3$, $5 \stackrel{\sqsubset}{\sim} 12$,
$6 \stackrel{\sqsubset}{\sim} 11$, $7 \stackrel{\sqsubset}{\sim} 10$, $8 \stackrel{\sqsubset}{\sim} 9$};
\node[anchor=west] at (82.5,7) {$1 \stackrel{\sqsupset}{\sim} 10$, $2 \stackrel{\sqsupset}{\sim} 9$, $3 \stackrel{\sqsupset}{\sim} 8$,
$4 \stackrel{\sqsupset}{\sim} 7$, $5 \stackrel{\sqsupset}{\sim} 6$, $11 \stackrel{\sqsupset}{\sim} 12$};
\node[anchor=west] at (82.5,4) {$0 \stackrel{\theta}\rightarrow 3 \stackrel{\theta}\rightarrow 0$,
$1 \stackrel{\theta}\rightarrow 4 \stackrel{\theta}\rightarrow 1$,
$2 \stackrel{\theta}\rightarrow 5 \stackrel{\theta}\rightarrow 2$};
\end{tikzpicture}
\end{center}
\caption{Relations $\stackrel{\sqsubset}{\sim}$ and $\stackrel{\sqsupset}{\sim}$, and permutation $\theta$ on a $3$-generalised diagram}
\label{fig:new-relations}
\end{figure}

Hence, $\bs\ba$ are actual coordinates if and only if $\theta$ is a cyclic permutation of $\{0,\ldots,\ell\}$.
For the ease of the computation, we identify below
the set $\{0,\ldots,\ell\}$ with the set $\ZZ_{\ell+1} := \ZZ/(\ell+1)\ZZ$.

\subsubsection{Proof of Theorem~\ref{thm:G-and-L3} -- Step 3: Which Permutations are Cyclic?}
\label{subsubsection:proof-2b}

Let us define the real number $\alpha := \frac{a_2}{2}$.
Note that $\alpha$ is not necessarily an integer, and that $a_2 = \lfloor\alpha\rfloor + \lceil\alpha\rceil$.
In addition, recall that we defined above the integer $m := \ell-k$, such that $m > 0$.
We consider separately various cases.

\noindent$\rhd$ If $a_2 > 0$ and $a_3 = 1$, then $0 \stackrel{\theta}{\rightarrow} 0$, as shown in Fig.~\ref{fig:a2-a3-four-cases}
(bottom-right case).
It follows that $\theta$ is not a cyclic permutation of the set $\mathbb{Z}_{\ell+1}$.

\medskip

\noindent$\rhd$ If $a_2 = 0$, then one checks easily, as shown in Fig.~\ref{fig:a2-a3-four-cases} (top cases), that
\begin{enumerate}[a.]
\item if $0 \leq u < m$, then $2u+1 \stackrel{\sqsubset}{\sim} 2(m-u) \stackrel{\sqsupset}{\sim} 2u+1+2(k+a_3)$;
\item if $u = m$, then $2u+1 \stackrel{\sqsubset}{\sim} 2\ell+2 \stackrel{\sqsupset}{\sim} \mathbf{1}_{a_3 = 0} \cdot 2 \ell +1$;
\item if $m < u \leq \ell$, then $2u+1 \stackrel{\sqsubset}{\sim} 2(\ell+1+m-u) \stackrel{\sqsupset}{\sim} 2u+1-2(m+1-a_3)$.
\end{enumerate}
It follows that $u \stackrel{\theta}{\rightarrow} u+(k+a_3)$ for all $u \in \mathbb{Z}_{\ell+1}$.

\begin{figure}[!ht]
\begin{center}
\begin{tikzpicture}[scale=0.35]
\draw[draw=black,ultra thick] (5,-1) -- (5,12);
\draw[draw=black,ultra thick] (14,-1) -- (14,12);

\draw[draw=black,fill=black,thick] (-1,0) circle (0.5);
\draw[draw=black,fill=black,thick] (23,0) circle (0.5);

\draw[draw=black,thick,>=stealth,->] (14,0) -- (-1,0) -- (-1,1) -- (2,7) -- (8,7) -- (11,11) -- (20,11) -- (23,1) -- (23,0) -- (22,0) -- (19,10) -- (14.5,10);
\draw[draw=black,thick,>=stealth,->] (14,2) -- (4,2) arc (270:90:1.5) -- (8,5) -- (11,9) -- (15,9) arc (90:-90:4.5) -- (14.5,0);
\draw[draw=black,thick,>=stealth,->] (14,4) -- (13,4) arc (270:90:0.5) -- (15,5) arc (90:-90:0.5) -- (14.5,4);
\draw[draw=black,thick,>=stealth,->] (14,6) -- (13,6) arc (90:270:1.5) -- (15,3) arc (-90:90:1.5) -- (14.5,6);
\draw[draw=black,thick,>=stealth,->] (14,8) -- (11,8) -- (8,4) -- (4,4) arc (90:270:0.5) -- (8,3) -- (11,7) -- (15,7) arc (90:-90:2.5) -- (14.5,2);
\draw[draw=black,thick,>=stealth,->] (14,10) -- (11,10) -- (8,6) -- (4,6) arc (90:270:2.5) -- (15,1) arc (-90:90:3.5) -- (14.5,8);

\draw[draw=black,fill=white,thick] (3.6,3.5) circle (0.4);
\draw[draw=black,fill=white,thick] (12.6,4.5) circle (0.4);
\draw[draw=black,fill=white,thick] (15.4,4.5) circle (0.4);

\draw[draw=black,fill=white,thick] (14.5,-0.5) -- (14.5,0.5) -- (13.5,0.5) -- (13.5,-0.5) -- cycle;
\draw[draw=black,fill=white,thick] (14.5,1.5) -- (14.5,2.5) -- (13.5,2.5) -- (13.5,1.5) -- cycle;
\draw[draw=black,fill=white,thick] (14.5,3.5) -- (14.5,4.5) -- (13.5,4.5) -- (13.5,3.5) -- cycle;
\draw[draw=black,fill=white,thick] (14.5,5.5) -- (14.5,6.5) -- (13.5,6.5) -- (13.5,5.5) -- cycle;
\draw[draw=black,fill=white,thick] (14.5,7.5) -- (14.5,8.5) -- (13.5,8.5) -- (13.5,7.5) -- cycle;
\draw[draw=black,fill=white,thick] (14.5,9.5) -- (14.5,10.5) -- (13.5,10.5) -- (13.5,9.5) -- cycle;
\node at (14,0) {a.};
\node at (14,2) {b.};
\node at (14,4) {c.};
\node at (14,6) {c.};
\node at (14,8) {d.};
\node at (14,10) {e.};

\node[anchor=west] at (23.5,9.5) {Case $k+1 \geq a_2$:};
\node[anchor=west] at (25,8) {$\bullet$};
\node[anchor=west] at (25,6.5) {$\bullet$};
\node[anchor=west] at (25,5) {$\bullet$};
\node[anchor=west] at (25,3.5) {$\bullet$};
\node[anchor=west] at (25,2) {$\bullet$};
\node[anchor=west] at (26,8) {$k = 3$};
\node[anchor=west] at (26,6.5) {$\ell = 5$};
\node[anchor=west] at (26,5) {$m = 2$};
\node[anchor=west] at (26,3.5) {$a_2 = 3$};
\node[anchor=west] at (26,2) {$\alpha = \frac{3}{2}$};

\draw[draw=black,thick] (-2,-2) -- (32,-2) -- (32,13) -- (-2,13) -- cycle;
\end{tikzpicture}

\medskip

\begin{tikzpicture}[scale=0.35]
\draw[draw=black,ultra thick] (5,-1) -- (5,12);
\draw[draw=black,ultra thick] (14,-1) -- (14,12);

\draw[draw=black,fill=black,thick] (-1,0) circle (0.5);
\draw[draw=black,fill=black,thick] (23,0) circle (0.5);

\draw[draw=black,thick,>=stealth,->] (14,0) -- (-1,0) -- (-1,1) -- (2,7) -- (8,7) -- (11,11) -- (20,11) -- (23,1) -- (23,0) -- (22,0) -- (19,10) -- (14.5,10);
\draw[draw=black,thick,>=stealth,->] (14,2) -- (4,2) arc (270:90:1.5) -- (8,5) -- (11,9) -- (15,9) arc (90:-90:4.5) -- (14.5,0);
\draw[draw=black,thick,>=stealth,->] (14,4) -- (4,4) arc (90:270:0.5) -- (15,3) arc (-90:90:1.5) -- (14.5,6);
\draw[draw=black,thick,>=stealth,->] (14,6) -- (13,6) arc (270:90:0.5) -- (15,7) arc (90:-90:2.5) -- (14.5,2);
\draw[draw=black,thick,>=stealth,->] (14,8) -- (13,8) arc (90:270:1.5) -- (15,5) arc (90:-90:0.5) -- (14.5,4);
\draw[draw=black,thick,>=stealth,->] (14,10) -- (11,10) -- (8,6) -- (4,6) arc (90:270:2.5) -- (15,1) arc (-90:90:3.5) -- (14.5,8);

\draw[draw=black,fill=white,thick] (3.6,3.5) circle (0.4);
\draw[draw=black,fill=white,thick] (12.6,6.5) circle (0.4);
\draw[draw=black,fill=white,thick] (15.4,4.5) circle (0.4);

\draw[draw=black,fill=white,thick] (14.5,-0.5) -- (14.5,0.5) -- (13.5,0.5) -- (13.5,-0.5) -- cycle;
\draw[draw=black,fill=white,thick] (14.5,1.5) -- (14.5,2.5) -- (13.5,2.5) -- (13.5,1.5) -- cycle;
\draw[draw=black,fill=white,thick] (14.5,3.5) -- (14.5,4.5) -- (13.5,4.5) -- (13.5,3.5) -- cycle;
\draw[draw=black,fill=white,thick] (14.5,5.5) -- (14.5,6.5) -- (13.5,6.5) -- (13.5,5.5) -- cycle;
\draw[draw=black,fill=white,thick] (14.5,7.5) -- (14.5,8.5) -- (13.5,8.5) -- (13.5,7.5) -- cycle;
\draw[draw=black,fill=white,thick] (14.5,9.5) -- (14.5,10.5) -- (13.5,10.5) -- (13.5,9.5) -- cycle;
\node at (14,0) {a.};
\node at (14,2) {b.};
\node at (14,4) {c.};
\node at (14,6) {d.};
\node at (14,8) {d.};
\node at (14,10) {e.};

\node[anchor=west] at (23.5,9.5) {Case $a_2 > k+1$:};
\node[anchor=west] at (25,8) {$\bullet$};
\node[anchor=west] at (25,6.5) {$\bullet$};
\node[anchor=west] at (25,5) {$\bullet$};
\node[anchor=west] at (25,3.5) {$\bullet$};
\node[anchor=west] at (25,2) {$\bullet$};
\node[anchor=west] at (26,8) {$k = 3$};
\node[anchor=west] at (26,6.5) {$\ell = 5$};
\node[anchor=west] at (26,5) {$m = 1$};
\node[anchor=west] at (26,3.5) {$a_2 = 5$};
\node[anchor=west] at (26,2) {$\alpha = \frac{5}{2}$};

\draw[draw=black,thick] (-2,-2) -- (32,-2) -- (32,13) -- (-2,13) -- cycle;
\end{tikzpicture}
\end{center}
\caption{Case $a_2 > 0$ and $a_3 = 0$: $k+1 \geq a_2$ and $a_2 > k+1$}
\label{fig:a2>a3=0}
\end{figure}
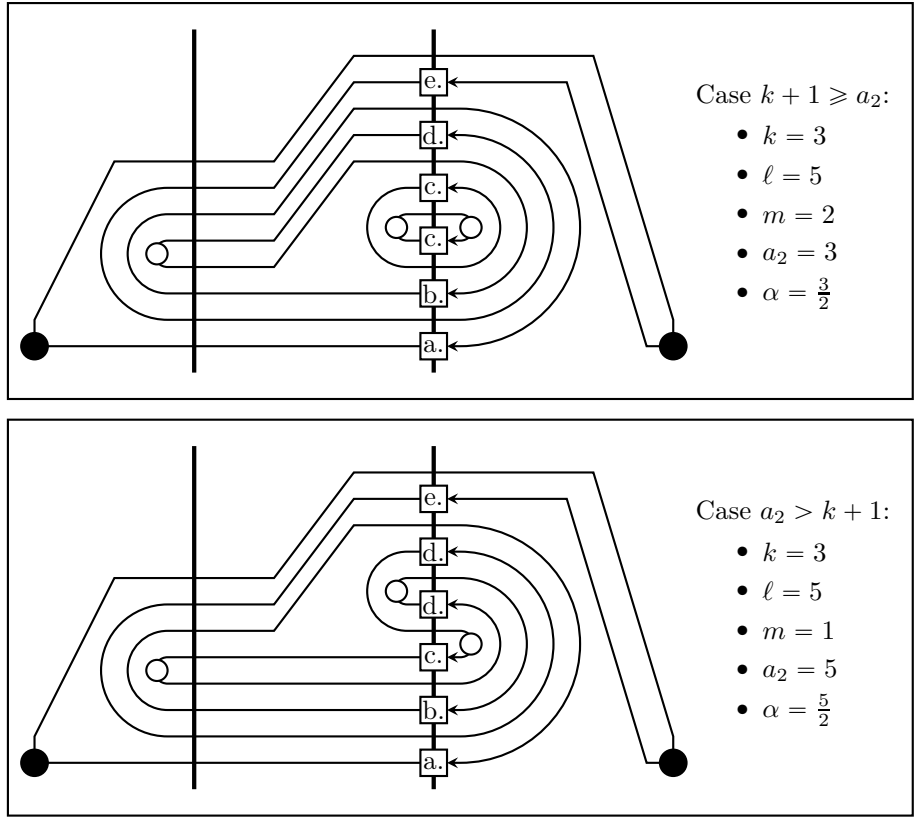

\noindent$\rhd$ If $k+1 \geq a_2 > a_3 = 0$, then one checks, as shown in Fig.~\ref{fig:a2>a3=0} (top case), that
\begin{enumerate}[a.]
\item if $u = 0$, then $2u+1 \stackrel{\sqsubset}{\sim} 2\ell+2 \stackrel{\sqsupset}{\sim} 2\ell+1$;
\item if $1 \leq u < \alpha$, then $2u+1 \stackrel{\sqsubset}{\sim} 2(\ell+1-u) \stackrel{\sqsupset}{\sim} 2u-1$;
\item if $\alpha \leq u < m+\alpha$, then $2u+1 \stackrel{\sqsubset}{\sim} 2(m+a_2-u) \stackrel{\sqsupset}{\sim} 2u+1+2(k-a_2)$;
\item if $m+\alpha \leq u < \ell+1-\alpha$, then $2u+1 \stackrel{\sqsubset}{\sim} 2(\ell+1+m-u) \stackrel{\sqsupset}{\sim} 2u+1-2(m+1)$;
\item if $\ell+1-\alpha \leq u \leq \ell$, then $2u+1 \stackrel{\sqsubset}{\sim} 2(\ell+1-u) \stackrel{\sqsupset}{\sim} 2u-1$.
\end{enumerate}
It follows that
\begin{itemize}
\item $u \rightarrow u-1$ if $0 \leq u < \alpha$ or if $\ell+1-\alpha \leq u \leq \ell$;
\item $u \rightarrow u+(k-a_2)$ if $\alpha \leq u < m+\alpha$;
\item $u \rightarrow u-(m+1)$ if $m+\alpha \leq u < \ell+1-\alpha$.
\end{itemize}

\medskip

\noindent$\rhd$ If $a_2 > k+1 > a_3 = 0$, then one checks, as shown in Fig.~\ref{fig:a2>a3=0} (bottom case), that
\begin{enumerate}[a.]
\item if $u = 0$, then $2u+1 \stackrel{\sqsubset}{\sim} 0 \stackrel{\sqsupset}{\sim} 2\ell+1$;
\item if $1 \leq u < k+1-\alpha$, then $2u+1 \stackrel{\sqsubset}{\sim} 2(\ell+1-u) \stackrel{\sqsupset}{\sim} 2u-1$;
\item if $k+1-\alpha \leq u < \alpha$, then $2u+1 \stackrel{\sqsubset}{\sim} 2(k+1-u) \stackrel{\sqsupset}{\sim} 2u+1+2(m-1)$;
\item if $\alpha \leq u < m+\alpha$, then $2u+1 \stackrel{\sqsubset}{\sim} 2(a_2+m-u) \stackrel{\sqsupset}{\sim} 2u+1-2(a_2-k)$;
\item if $m+\alpha \leq u \leq \ell$, then $2u+1 \stackrel{\sqsubset}{\sim} 2(\ell+1-u) \stackrel{\sqsupset}{\sim} 2u-1$.
\end{enumerate}
It follows that
\begin{itemize}
\item $u \rightarrow u-1$ if $0 \leq u < k+1-\alpha$ or if $m+\alpha \leq u \leq \ell$;
\item $u \rightarrow u+(m-1)$ if $k+1-\alpha \leq u < \alpha$;
\item $u \rightarrow u-(a_2-k)$ if $\alpha \leq u < m+\alpha$.
\end{itemize}

In both pictures of Fig.~\ref{fig:a2>a3=0}, on each point $\conc{2}{2u+1}$,
is written the subcase (from a. to e.) to which we should refer.
For instance, in the top picture, the label ``b.'' is written on the point $\conc{2}{3}$,
which is associated to the case $i = 1$, i.e. $1 \leq u < \alpha$.
Indeed, one checks that $\conc{2}{3} \stackrel{2}{\sim} \conc{1}{3} \stackrel{1}{\sim} \conc{1}{6} \stackrel{2}{\sim} \conc{2}{10} \stackrel{2}{\sim} \conc{2}{0}$,
which shows that $3 \stackrel{\sqsubset}{\sim} 10 \stackrel{\sqsupset}{\sim} 0$, as mentioned in the above enumeration of cases.

Overall, in each case, we observe that each permutation $\ell$ has a specific structure,
which we call \emph{translation} and \emph{translated cut}.

\begin{figure}[!ht]
\begin{center}
\begin{tikzpicture}[scale=0.245]

\draw[draw=black,pattern=crosshatch,pattern color=blackgray,ultra thick] (-12,0) -- (-6,0) -- (-6,2) -- (-12,2) -- cycle;
\draw[draw=black,pattern=north east lines,pattern color=blackgray,ultra thick] (-12,2) -- (-6,2) -- (-6,6) -- (-12,6) -- cycle;
\draw[draw=black,fill=gray,ultra thick] (-12,6) -- (-6,6) -- (-6,8) -- (-12,8) -- cycle;
\draw[draw=black,fill=verypalegray,ultra thick] (-12,8) -- (-6,8) -- (-6,11) -- (-12,11) -- cycle;
\draw[draw=black,pattern=north west lines,pattern color=blackgray,ultra thick] (-12,11) -- (-6,11) -- (-6,20) -- (-12,20) -- cycle;

\draw[draw=black,pattern=crosshatch,pattern color=blackgray,ultra thick] (0,0) -- (6,0) -- (6,2) -- (0,2) -- cycle;
\draw[draw=black,pattern=north east lines,pattern color=blackgray,ultra thick] (0,2) -- (6,2) -- (6,6) -- (0,6) -- cycle;
\draw[draw=black,fill=verypalegray,ultra thick] (0,6) -- (6,6) -- (6,9) -- (0,9) -- cycle;
\draw[draw=black,fill=gray,ultra thick] (0,9) -- (6,9) -- (6,11) -- (0,11) -- cycle;
\draw[draw=black,pattern=north west lines,pattern color=blackgray,ultra thick] (0,11) -- (6,11) -- (6,20) -- (0,20) -- cycle;

\draw[draw=black,pattern=north east lines,pattern color=blackgray,ultra thick] (12,0) -- (18,0) -- (18,4) -- (12,4) -- cycle;
\draw[draw=black,fill=verypalegray,ultra thick] (12,4) -- (18,4) -- (18,7) -- (12,7) -- cycle;
\draw[draw=black,fill=gray,ultra thick] (12,7) -- (18,7) -- (18,9) -- (12,9) -- cycle;
\draw[draw=black,pattern=north west lines,pattern color=blackgray,ultra thick] (12,9) -- (18,9) -- (18,18) -- (12,18) -- cycle;
\draw[draw=black,pattern=crosshatch,pattern color=blackgray,ultra thick] (12,18) -- (18,18) -- (18,20) -- (12,20) -- cycle;

\draw[draw=black,->,>=stealth',very thick] (-5.5,1) -- (-0.5,1);
\draw[draw=black,->,>=stealth',very thick] (-5.5,4) -- (-0.5,4);
\draw[draw=black,->,>=stealth',very thick] (-5.5,7) -- (-4.5,7) -- (-2,10) -- (-0.5,10);
\draw[draw=black,->,>=stealth',very thick] (-5.5,9.5) -- (-4.5,9.5) -- (-2,7.5) -- (-0.5,7.5);
\draw[draw=black,->,>=stealth',very thick] (-5.5,15.5) -- (-0.5,15.5);

\draw[draw=black,->,>=stealth',very thick] (6.5,1) -- (7.5,1) -- (10,19) -- (11.5,19);
\draw[draw=black,->,>=stealth',very thick] (6.5,4) -- (7.5,4) -- (10,2) -- (11.5,2);
\draw[draw=black,->,>=stealth',very thick] (6.5,7.5) -- (7.5,7.5) -- (10,5.5) -- (11.5,5.5);
\draw[draw=black,->,>=stealth',very thick] (6.5,10) -- (7.5,10) -- (10,8) -- (11.5,8);
\draw[draw=black,->,>=stealth',very thick] (6.5,15.5) -- (7.5,15.5) -- (10,13.5) -- (11.5,13.5);

\node[anchor=south] at (-3,20) {$\bC^{\bu\bt}_{n,a,b,c}$};
\node[anchor=south] at (9,20) {$\bT_{n,1}$};

\draw[draw=black,<->,>=stealth',thick] (-13,0) -- (-13,2);
\draw[draw=black,<->,>=stealth',thick] (-13,2) -- (-13,6);
\draw[draw=black,<->,>=stealth',thick] (-13,6) -- (-13,8);
\draw[draw=black,<->,>=stealth',thick] (-13,8) -- (-13,11);
\draw[draw=black,<->,>=stealth',thick] (-13,11) -- (-13,20);
\draw[draw=black,<->,>=stealth',thick] (19,0) -- (19,4);
\draw[draw=black,<->,>=stealth',thick] (19,4) -- (19,7);
\draw[draw=black,<->,>=stealth',thick] (19,7) -- (19,9);
\draw[draw=black,<->,>=stealth',thick] (19,9) -- (19,18);
\draw[draw=black,<->,>=stealth',thick] (19,18) -- (19,20);

\node[anchor=south,rotate=90] at (-13,1) {\small $1$};
\node[anchor=south,rotate=90] at (-13,4) {\small $a-1$};
\node[anchor=south,rotate=90] at (-13,7) {\small $b$};
\node[anchor=south,rotate=90] at (-13,9.5) {\small $c$};
\node[anchor=south,rotate=90] at (-13,15.5) {\small $n-a-b-c$};
\node[anchor=north,rotate=90] at (19,2) {\small $a-1$};
\node[anchor=north,rotate=90] at (19,5.5) {\small $c$};
\node[anchor=north,rotate=90] at (19,8) {\small $b$};
\node[anchor=north,rotate=90] at (19,13.5) {\small $n-a-b-c$};
\node[anchor=north,rotate=90] at (19,19) {\small $1$};
\end{tikzpicture}
\end{center}
\caption{Translated cut $\bT\bC^{\bu\bt}_{n,a,b,c}$}
\label{fig:translated-cut}
\end{figure}

\begin{dfn}{Translation and translated cut}
Let $a$, $b$, $c$ and $n$ be non-negative integers such that $a+b+c \leq n$, and set $\ZZ_n := \ZZ/n\ZZ$.
We call \emph{translation}, and denote by $\bT_{n,a}$, the permutation of $\ZZ_n$ such that
$\bT_{n,a} : k \mapsto k-a$.
We call \emph{translated cut}, and denote by $\bT\bC^{\bu\bt}_{n,a,b,c}$, the permutation $\bT_{n,1} \circ \bC^{\bu\bt}_{n,a,b,c}$ of $\ZZ_n$, where
$\bC^{\bu\bt}_{n,a,b,c}$ is the permutation such that
\begin{eqnarray*}
\bC^{\bu\bt}_{n,a,b,c} : k & \mapsto & k \text{ if } k \in \{0,\ldots,a-1,a+b+c,\ldots,n-1\} \\
& & k+c \text{ if } k \in \{a,\ldots,a+b-1\} \\
& & k-b \text{ if } k \in \{a+b,\ldots,a+b+c-1\}.
\end{eqnarray*}
\end{dfn}

We proved above that
\begin{itemize}
\item if $a_2 = 0$, then $\theta$ is the translation $\bT_{\ell+1,m+1-a_3}$;
\item if $k+1 \geq a_2 > a_3 = 0$, then $\theta$ is the translated cut
$\bT\bC^{\bu\bt}_{\ell+1,\lceil\alpha\rceil,m,k+1-a_2}$;
\item if $a_2 > k+1 > a_3 = 0$, then $\theta$ is the translated cut
$\bT\bC^{\bu\bt}_{\ell+1,k+1-\lfloor\alpha\rfloor,a_2-k-1,m}$.
\end{itemize}
Hence, it remains to check which translations and translated cuts are cyclic permutations.
The first case is immediate, whereas the second one is not.
Both cases are expressed in terms of coprimality:
for all relative integers $a$ and $b$,
we denote by $a \wedge b$ the greatest common divisor of $a$ and $b$,
i.e. the (unique) non-negative integer $d$ such that $\{a x + b y : x,y \in \ZZ\} = d \ZZ$.
In particular, note that $a \wedge b = |a| \wedge |b|$ and $0 \wedge b = |b|$ for all integers $a, b \in \ZZ$.

\begin{lem}
Let $a$ and $n$ be integers such that $0 \leq a \leq n$.
The translation $\bT_{n,a} : \ZZ_n \mapsto \ZZ_n$ is cyclic if and only if $a \wedge n = 1$.
\label{lem:cut-cyclic}
\end{lem}

\begin{lem}
Let $a$, $b$, $c$ and $n$ be non-negative integers such that $a+b+c \leq n$.
The translated cut $\bT\bC^{\bu\bt}_{n,a,b,c} : \ZZ_n \mapsto \ZZ_n$ is cyclic if and only if $(c-1) \wedge (b+1) = 1$.
\label{lem:sliding-cut-cyclic}
\end{lem}

\begin{proof}
First, observe that $\bT_{n,a} \circ \bT\bC^{\bu\bt}_{n,a,b,c} = \bT\bC^{\bu\bt}_{n,0,b,c} \circ \bT_{n,a}$.
This means that the translated cuts $\bT\bC^{\bu\bt}_{n,a,b,c}$ and $\bT\bC^{\bu\bt}_{n,0,b,c}$ are conjugate to each other.
Therefore, the permutation $\bT\bC^{\bu\bt}_{n,a,b,c}$ is cyclic if and only if $\bT\bC^{\bu\bt}_{n,0,b,c}$ is cyclic too,
and we henceforth assume that $a = 0$.

Second, observe that $b \xmapsto{\bT\bC^{\bu\bt}_{n,0,b,c}} n-1 \xmapsto{\bT\bC^{\bu\bt}_{n,0,b,c}} n-2
\xmapsto{\bT\bC^{\bu\bt}_{n,0,b,c}} \ldots \xmapsto{\bT\bC^{\bu\bt}_{n,0,b,c}} b+c-1$.
Hence, the permutation $\bT\bC^{\bu\bt}_{n,0,b,c}$ is cyclic if and only if $\bT\bC^{\bu\bt}_{b+c,0,b,c}$ is cyclic too,
and we henceforth assume that $n = b+c$.

Third, observe that $\bT\bC^{\bu\bt}_{b+c,0,b,c}$ is simply the translation $\bT_{b+c,b+1}$.
Consequently, the permutation $\bT\bC^{\bu\bt}_{b+c,0,b,c}$ is cyclic if and only if $(b+c) \wedge (b+1) = 1$,
i.e. if and only if $(c-1) \wedge (b+1) = 1$.
This completes the proof.
\end{proof}

Remember that $\bs\ba$ are actual coordinates if and only if $\theta$ is a cyclic permutation of $\ZZ_{\ell+1}$.
Hence, Lemmas~\ref{lem:cut-cyclic} and~\ref{lem:sliding-cut-cyclic} prove that
$\bs\ba = (0,a_1,k,a_2,k+m,a_3,0)$ are actual coordinates if and only if we are in the following cases:
\begin{enumerate}[(i)]
\item $a_2 = 0$, $a_3 = 0$ and $k \wedge (m+1) = 1$;
\item $a_2 = 0$, $a_3 = 1$ and $(k+1) \wedge m = 1$;
\item $k+1 \geq a_2 \geq 1$, $a_3 = 0$ and $(k-a_2) \wedge (m+1) = 1$;
\item $2k+1 \geq a_2 \geq k+2$, $a_3 = 0$ and $(a_2-k) \wedge (m-1) = 1$.
\end{enumerate}

In particular, taking into account the virtual coordinates $\bs\ba$ such that $a_1 = 0$
(and whose case was tackled in the first few lines of Section~\ref{subsubsection:proof-2b}),
it follows that, whenever $k \geq 1$ and $m \geq 1$, we obtain the following formula
for the integers
\[C_{k,k+m} = C_{k+m,k} = |\{(a_1,a_2,a_3) : (0,a_1,k,a_2,k+m,a_3,0) \text{ are actual coordinates}\}|:\]
\begin{eqnarray*}
\frac{C_{k,k+m}}{2} & = & \mathbf{1}_{k \wedge (m+1) = 1} + \mathbf{1}_{(k+1) \wedge m = 1} + \\
& & \sum_{a_2=1}^{k+1} \mathbf{1}_{(k-a_2) \wedge (m+1) = 1} + \sum_{a_2=k+2}^{2k+1} \mathbf{1}_{(m-1) \wedge (a_2-k) = 1} \\
& = & \sum_{a=1}^k \mathbf{1}_{a \wedge (m+1) = 1} + \mathbf{1}_{(k+1) \wedge m = 1} + \sum_{a=1}^{k+1} \mathbf{1}_{a \wedge (m-1) = 1}.
\end{eqnarray*}

\subsubsection{Proof of Theorem~\ref{thm:G-and-L3} -- Step 4: Generating Functions}
\label{subsubsection:proof-3}

Focus now on the generating function $\calG_3(z) = \sum_{k \geq 0} g_{3,k}z^k$.
For the sake of clarity and conciseness, we only indicate the main steps of our computations,
which are mainly based on rearranging terms.

We proved in Section~\ref{subsubsection:proof-1} that
$C_{0,0} = 1$, that $C_{0,\ell} = C_{\ell,0} = 2$ for $\ell \geq 1$, and that $C_{k,k} = 2(2k+1)$ for $k \geq 1$.
It follows that
\[\calG_3(z) = \sum_{k,\ell \geq 0} C_{k,\ell} z^{k+\ell} = 1 + \sum_{\ell \geq 1} 4 z^\ell + \sum_{k \geq 1} 2(2k+1) z^{2k} + 2 \sum_{k \geq 1, m \geq 1} C_{k,k+m} z^{2k+m}.\]
Using the above formula for $C_{k,k+m}$ (when $k, m \geq 1$) and , we can rewrite this as
\[\calG_3(z) = 1 + \frac{4 z}{1-z} - \frac{2 z^2 (z^2-3)}{(1-z^2)^2} + 4 (H_1(z) + H_2(z) + H_3(z) - H_4(z)),\]
where
\begin{eqnarray*}
H_1(z) = \sum_{k \geq 1}\sum_{m \geq 1}\sum_{a=1}^k \mathbf{1}_{a \wedge (m+1) = 1}z^{2k+m}, & &
H_2(z) = \sum_{k \geq 1}\sum_{m \geq 1} \mathbf{1}_{(k+1) \wedge m = 1}z^{2k+m}, \\
H_3(z) = \sum_{k \geq 0}\sum_{m \geq 1}\sum_{a=1}^{k+1} \mathbf{1}_{a \wedge (m-1) = 1}z^{2k+m}, & &
\text{ and } H_4(z) = \sum_{m \geq 1} z^m.
\end{eqnarray*}

Then, let us define the function $F(z) := \sum_{\alpha \geq 1}\sum_{\beta \geq 1} \mathbf{1}_{\alpha \wedge \beta = 1}z^{2\alpha+\beta}$.
Using simple substitutions ($t := k-a$, $u := k+1$, $v := m+1$ and $w := m-1$), we get
\begin{eqnarray*}
H_1(z) & = & \sum_{a \geq 1}\sum_{m \geq 1}\sum_{t \geq 0} \mathbf{1}_{a \wedge (m+1) = 1}z^{2(a+t)+m} =
\frac{1}{1-z^2}\sum_{a \geq 1}\sum_{m \geq 1} \mathbf{1}_{a \wedge (m+1) = 1}z^{2a+m} \\
& = & \frac{1}{z(1-z^2)} \left(\sum_{a \geq 1}\sum_{v \geq 1} \mathbf{1}_{a \wedge v = 1}z^{2a+v} - \sum_{a \geq 1} z^{2a+1}\right) =
\frac{F(z)}{z(1-z^2)} - \frac{z^2}{(1-z^2)^2}, \\
H_2(z) & = & \frac{1}{z^2}\left(\sum_{u \geq 1}\sum_{m \geq 1} \mathbf{1}_{u \wedge m = 1}z^{2u+m} - \sum_{m \geq 1} z^{2+m}\right) =
\frac{F(z)}{z^2} - \frac{z}{1-z}, \\
H_3(z) & = & \sum_{a \geq 1}\sum_{m \geq 1}\sum_{t \geq -1} \mathbf{1}_{a \wedge (m-1) = 1}z^{2(a+t)+m} =
\frac{1}{z^2(1-z^2)}\sum_{a \geq 1}\sum_{m \geq 1} \mathbf{1}_{a \wedge (m-1) = 1}z^{2a+m} \\
& = & \frac{1}{z(1-z^2)} \left(\sum_{a \geq 1}\sum_{w \geq 1} \mathbf{1}_{a \wedge w = 1}z^{2a+w} + z^2\right) =
\frac{F(z)}{z(1-z^2)} + \frac{z}{1-z^2}, \text{ and} \\
H_4(z) & = & \frac{z}{1-z}.
\end{eqnarray*}

Moreover, consider the coefficients $f_n$ of the series $F(z)$.
Since $F(z) = \sum_{n \geq 3} f_n z^n$, we have
\begin{eqnarray*}
f_n & = & \left|\left\{a < \frac{n}{2} : a \wedge (n-2a) = 1\right\}\right| = \left|\left\{a < \frac{n}{2} : a \wedge n = 1\right\}\right| \\
& = & \frac{1}{2}\left|\left\{a \leq n : a \neq \frac{n}{2}, a \wedge n = 1\right\}\right|. 
\end{eqnarray*}
Observe that, if $n \geq 3$ is even, then $\frac{n}{2} \wedge n = \frac{n}{2} \neq 1$.
It follows that $f_n = \frac{\varphi(n)}{2}$, i.e. that $F(z) = \frac{1}{2} \sum_{n \geq 3} \varphi(n) z^n$,
where $\varphi$ denotes the Euler totient.
Collecting the above terms, we have
\begin{eqnarray*}
\calG_3(z) & = & 1 + \frac{4 z}{1-z} - \frac{2 z^2 (z^2-3)}{(1-z^2)^2} + 4 \left(\frac{2}{z(1-z^2)} + \frac{1}{z^2}\right) F(z) + \\
& & 4 \left(\frac{z}{1-z^2} - \frac{z^2}{(1-z^2)^2} - \frac{2z}{1-z}\right) \\
& = & 4 \frac{1+2z-z^2}{z^2(1-z^2)} F(z) + \frac{1-3z^2}{1-z^2},
\end{eqnarray*}
and since $\calB_3(z) = z^2 \calG_3(z^2)$, the two first parts of Theorem~\ref{thm:G-and-L3} are proved.

\bigskip

In addition, developing term-wise the series $\calG_3(z) = \sum_{k \geq 0}g_{3,k}z^k$ gives
\[\calG_3(z) = \left(\sum_{k \geq 0} z^{2k}\right) \left(2 (1+2z-z^2)\sum_{k \geq 1} \varphi(k+2) z^k + (1-3z^2)\right),\]
which proves that $g_{3,k} = \sum_{i=0}^{\lfloor k/2 \rfloor} \gamma_{k-2i}$, with
\[\gamma_i = \mathbf{1}_{i = 0} - 3 \cdot \mathbf{1}_{i = 2} + 2 \varphi(i+2) \mathbf{1}_{i \geq 1} + 4 \varphi(i+1) \mathbf{1}_{i \geq 2} -
2 \varphi(i) \mathbf{1}_{i \geq 3}.\]
It follows that
\[g_{3,k} = \mathbf{1}_{k = 0} + 2 \left(\varphi(k+2)-\mathbf{1}_{k \in 2\ZZ} + 2 \sum_{i=1}^{\lfloor k/2 \rfloor} \varphi(k+3-2i)\right) \mathbf{1}_{k \geq 1},\]
which proves the third part of Theorem~\ref{thm:G-and-L3}.

\subsubsection{Proof of Theorem~\ref{thm:G-and-L3} -- Step 5: Holonomy}
\label{subsubsection:holonomy}

Finally, we prove the last part of Theorem~\ref{thm:G-and-L3}, i.e. that the generating functions
\[\calG_3(z) = 2 \frac{1+2z-z^2}{z^2(1-z^2)} \left(\sum_{n \geq 3} \varphi(n) z^n\right) + \frac{-1+3z^2}{1-z^2}, ~ \calB_3(z) = z^2 \calG_3(z^2)\]
are not holonomic.

We do so by using standard tools and results of complex analysis (see~\cite[Annex B.4]{Flajolet:2009:AC:1506267})
and of analytic number theory (see~\cite{hardy-wright-introduction-to-the-theory-of-numbers}).

For the sake of contradiction, let us assume henceforth that $\calG_3(z)$ is holonomic.
Then, so is the generating function $\sum_{n \geq 1} \varphi(n) z^n$,
i.e. the sequence $(\varphi(n))_{n \geq 1}$ is $P$-recursive:
this means that there exists some complex polynomials $A_0(X), \ldots, A_k(X)$ with
such that $A_k \neq 0$ and $\sum_{i=0}^k A_i(n) \varphi(n+i) = 0$ for all integers $n \geq 1$.
In addition, since each term $\varphi(n)$ is a rational number,
we may even assume that $A_0(X), \ldots, A_k(X)$ have integer coefficients.

Let $d := \max\{\deg A_i : 0 \leq i \leq k\}$,
and let us write $A_i(X) = \sum_{j=0}^d a_{i,j} X^j$ for all $i \leq k$.
Then, consider some integer $\ell \in \{0,\ldots,k\}$ such that $\deg A_\ell = d$, i.e. $a_{\ell,d} \neq 0$,
and let us define the integer $a_\infty := \sum_{i=0}^k |a_{i,d}|$.

The Euler identity
\[\prod_{p \text{ prime}} \frac{1}{1-p^{-1}} = \prod_{p \text{ prime}} \left(\sum_{j \geq 0} p^{-j}\right) = \sum_{n \geq 1} n^{-1} = +\infty\]
shows that $\prod_{p \text{ prime}} (1-p^{-1}) = 0$.
Hence, there exists pairwise disjoint sets $P_0,\ldots,P_k$
of primes numbers greater than $k$ and
such that $\prod_{p \in P_i} (1-p^{-1}) \leq \frac{1}{2 a_\infty}$ for all $i \leq k$.
Consequently, the intergers $b_i = \prod_{p \in P_i}$ are pairwise coprime integers
such that $\varphi(b_i) = \prod_{p \in P_i} (p-1) \leq \frac{b_i}{2 a_\infty}$.

The Chinese remainder theorem~\cite[Theorem~59]{hardy-wright-introduction-to-the-theory-of-numbers}
shows that there exists an integer $N \geq 0$ such that
$N+i \equiv 0 \pmod{b_i}$ for all $i \neq \ell$.
Since the prime factors of $b_i$ are greater than $k$,
it follows that $N+\ell$ is coprime with $b_i$, for all $i \neq \ell$.
Hence, the Dirichlet theorem~\cite[Theorem~15]{hardy-wright-introduction-to-the-theory-of-numbers}
states that there exists arbitrarily large integers $n$ (in the set $\{N + z \prod_{i \neq \ell}b_i : z \in \NN\}$) such that
$n + \ell$ is prime.
For such integers $n$, remember that $n+i \equiv 0 \pmod{b_i}$ when $i \neq \ell$.
It follows that $\varphi(n+\ell) = n+\ell-1$ and $\varphi(n+i) \leq \frac{\varphi(b_i)}{b_i} (n+i) \leq \frac{n+i}{2 a_\infty}$.
Since $0 = \sum_{i=0}^k A_i(n) \varphi(n+i)$, we deduce that
\begin{eqnarray*}
(n-1) |A_\ell(n)| & \leq & |A_\ell(n) \varphi(n+\ell)| \leq \sum_{i \neq \ell} |A_i(n) \varphi(i+\ell)| \\
& \leq & \frac{n+k}{2 a_\infty} \sum_{i \neq \ell}|A_i(n)| \leq \frac{n+k}{2 a_\infty} \sum_{i=0}^k|A_i(n)|.
\end{eqnarray*}
When $n \to +\infty$, we have $|A_\ell(n)| \sim |a_{\ell,d}| n^d$ and $\sum_{i=0}^k|A_i(n)| \sim a_\infty n^d$,
from which we deduce that
\[|a_{\ell,d}| n^{d+1} \sim (n-1) |A_\ell(n)| \leq \frac{n+k}{2 a_\infty} \sum_{i=0}^k|A_i(n)| \sim \frac{1}{2} n^{d+1},\]
which is impossible since $|a_{\ell,d}| \geq 1$.
This contradiction shows that the generating function $\calG_3(z)$ could not be holonomic.

Finally, since $\calG_2(z) = z^{-1} \calB_3(z^{1/2})$ and since $z \mapsto z^{1/2}$ is algebraic,
the generating function $\calB_3(z)$ cannot be holonomic either.
This was the last step of the proof of Theorem~\ref{thm:G-and-L3},
which is now completed.

\subsection{Asymptotic Values in $B_3$}
\label{subsubsection:asymptotics}

We first use Theorem~\ref{thm:G-and-L3} to estimate precisely the terms $g_{3,k}$ of the series $\calG_3(z)$.

\begin{pro}
When $n \to +\infty$, we have: \[g_{3,k} \sim 4 \left(1+ \mathbf{1}_{k \in 2\ZZ}\right) \frac{k^2}{\pi^2}.\]
\label{pro:cut-euler-summatory-even-odd}
\end{pro}

\begin{proof}
For the sake of simplicity, let us introduce some notation.
We define $\alpha = \frac{4}{\pi^2}$ and $\phi_k = \sum_{i=0}^{\lfloor (k-1)/2 \rfloor} \varphi(k-2i)$,
as well as real numbers $\varepsilon_k$, $\theta_k$ and $\eta_k$ such that
$\phi_{2k} = (\alpha+\varepsilon_k)k^2$, $\phi_{2k-1} = (2\alpha + \theta_k)k^2$ and
$\eta_k = \varepsilon_k+\theta_k$.
We first want to prove that $\varepsilon_k \to 0$ and that $\theta_k \to 0$ when $k \to +\infty$.

It is a standard result that
\[\sum_{k=1}^n \varphi(k) \sim \frac{3}{\pi^2}n^2\]
when $n \to +\infty$ (see~\cite[Theorem~330]{hardy-wright-introduction-to-the-theory-of-numbers}).
It follows that
\[(3\alpha+\eta_k)k^2 = \phi_{2k} + \phi_{2k-1} = \sum_{i=1}^{2k}\varphi(i) \sim \frac{12}{\pi^2}k^2 = 3 \alpha k^2,\]
which means that $\eta_k \to 0$. Hence, it remains to prove that $\varepsilon_k \to 0$.

Then, let $A$ be some positive constant, and let $K$ be some positive integer such that
$\frac{\alpha}{(2k+1)^2} \leq A$ and $|\eta_k| \leq A$ whenever $k \geq K$.
In addition, for each integer $\ell \geq \log_2(K)$, we define $M_\ell = \max\{|\varepsilon_k| : 2^\ell \leq k \leq 2^{\ell+1}\}$.
If $2^\ell \leq k \leq 2^{\ell+1}$, then
\[\phi_{4k} = \sum_{i=0}^{k} \varphi(4i) + \sum_{i=0}^{k-1} \varphi(4i+2) = 2 \sum_{i=0}^k \varphi(2i) + \sum_{i=0}^{k-1} \varphi(2i+1) =
2\phi_{2k} + \phi_{2k-1},\] i.e.
$\varepsilon_{2k} = \frac{\varepsilon_k + \eta_k}{4}$.
It follows that \[|\varepsilon_{2k}| \leq \frac{M_\ell + A}{4} \leq 2A + \frac{3M_\ell}{4}.\]

Similarly, if $2^\ell \leq k < 2^{\ell+1}$, then
$\phi_{4k+2} = \phi_{4k} + \varphi(4k+2) = 2\phi_{2k}+\phi_{2k+1}$, i.e.
\[\varepsilon_{2k+1} = \frac{\alpha}{(2k+1)^2} + \frac{2k^2}{(2k+1)^2} \varepsilon_k -
\frac{(k+1)^2}{(2k+1)^2} \varepsilon_{k+1} + \frac{(k+1)^2}{(2k+1)^2} \eta_{k+1}.\]
Since $k \geq 2^\ell \geq K$, we know that $\frac{\alpha}{(2k+1)^2} \leq A$.
Moreover, note that \[\frac{2k^2+(k+1)^2}{(2k+1)^2} = \frac{3}{4} - \frac{4k-1}{4(2k+1)^2} \leq \frac{3}{4}.\]
It follows that
\[|\varepsilon_{2k+1}| \leq A + \frac{2k^2+(k+1)^2}{(2k+1)^2} M_\ell + \frac{(k+1)^2}{(2k+1)^2} A \leq 2A + \frac{3M_\ell}{4}.\]

Overall, $|\varepsilon_m| \leq 2A + \frac{3M_\ell}{4}$ whenever $2^{\ell+1} \leq m \leq 2^{\ell+2}$,
which shows that $M_{\ell+1} \leq 2A + \frac{3M_\ell}{4}$.
It follows that $\limsup M_{\ell} \leq 8A$ and,
since $A$ is an arbitrary positive constant, that $M_\ell \to 0$ when $\ell \to +\infty$.
Recall that $|\varepsilon_k| \leq M_{\ell}$ whenever $2^\ell \leq k \leq 2^{\ell+1}$:
this proves that $\varepsilon_k \to 0$, and therefore that $\theta_k = \eta_k - \varepsilon_k \to 0$.
It follows that \[\phi_k \sim \left(1+\mathbf{1}_{k \in 2\ZZ+1}\right) \frac{k^2}{\pi^2}\] when $k \to +\infty$.

With the above notations, and according to Theorem~\ref{thm:G-and-L3}, we have
\[g_{3,k} = \mathbf{1}_{k = 0} + 2 \left(\varphi(k+2)-\mathbf{1}_{k \in 2\ZZ} + 2 \sum_{i=1}^{\lfloor k/2 \rfloor} \varphi(k+3-2i)\right) \mathbf{1}_{k \geq 1} =
4 \phi_{k+1} + \calO(k).\]
Morover, we just showed that
$k^2 \leq \left(1+\mathbf{1}_{k \in 2\ZZ+1}\right) k^2 \sim \pi^2 \phi_k$, and therefore that
$k^2 = \calO(\phi_k)$. It follows that
$g_{3,k} = 4 \phi_{k+1} + \calO(k) = 4 \phi_{k+1} + o(\phi_k)$, i.e. that
$g_{3,k} \sim 4 \phi_{k+1}$,
which proves Proposition~\ref{pro:cut-euler-summatory-even-odd}.
\end{proof}

From Proposition~\ref{pro:cut-euler-summatory-even-odd} follows an
additional result about the ``complexity'' of the sequence $(g_{3,k})_{k \geq 0}$,
whose proof (similar to that of Section~\ref{subsubsection:holonomy}) is omitted here.

\begin{cor}
The Lambert series $S_3(z) := \sum_{n \geq 1} g_{3,n} \frac{z^n}{1-z^n}$ is not holonomic.
\label{cor:lambert-series}
\end{cor}

This result completes Theorem~\ref{thm:G-and-L3} by showing that neither the ``standard'' generating function
nor the Lambert series associated to the sequence $(g_{3,n})_{n \geq 1}$ are holonomic.
In particular, provided that the sequence $(\varphi(n))_{n \geq 1}$ has a very simple Lambert series
$\sum_{n \geq 1}\varphi(n) \frac{z^n}{1-z^n} = \frac{z}{(1-z)^2}$ and that
the function $\calG_3(z) = \sum_{n \geq 0} g_{3,n} z^n$ is the composition of the series $\sum_{n \geq 1}\varphi(n) z^n$ by a
rational fraction, one might have hoped that the Lambert series $S_3(z)$ would be simple too.
Corollary~\ref{cor:lambert-series} proves that this is not the case.

\subsection{Estimates in $B_n$ ($n \geq 4$)}

We did not manage to compute explicitly the generating functions $\calG_n(z)$ nor the integers $g_{n,k}$ for $n \geq 4$.
Hence, we settle for upper and lower bounds.

First, we find an upper bound on the number of \emph{virtual} coordinates
$(s_0,a_1,s_1,\ldots,a_n,s_n)$ such that 
$\sum_{i=0}^n s_i = k$;
of course, this will also provide us with an upper bound on the integers $g_{n,k}$.

\begin{pro}
Let $n \geq 1$ and $k \geq 0$ be integers.
Then, $g_{n,k} \leq 2^n \left(\frac{k+n-1}{n-1}\right)^{n-2} \binom{k+n-2}{n-2}$.
\label{pro:not-too-many-people-v2}
\end{pro}

\begin{proof}
Let us bound above the number of ways in which one can choose a tuple $(s_0,a_1,\ldots,s_n)$ of actual coordinates.
First, there are exactly $\binom{k+n-2}{n-2}$ ways of choosing non-negative integers $s_1,\ldots,s_{n-1}$ whose sum is $k$.

Second, we know that $0 \leq a_i \leq 2 \min \{s_{i-1},s_i\}+1$ for all integers $i \in \{1,2,\ldots,n\}$.
Then, let $u$ be an integer such that $s_u = \max \{s_1,\ldots,s_{n-1}\}$:
we know that $0 \leq a_j \leq 2 s_{j-1}+1$ when $1 \leq j \leq u$ and that
$0 \leq a_j \leq 2 s_j+1$ when $u+1 \leq j \leq n$.
Therefore, let $S_u$ denote the set $\{1,\ldots,u-1,u+1,\ldots,n-1\}$:
the tuple $(a_1,\ldots,a_n)$ must belong to the Cartesian product
$\{0,1\} \times \prod_{j \in S_u} \{0,\ldots,2s_j+1\} \times \{0,1\}$, whose cardinality is
$P = 2^n \prod_{j \in S_u} (s_j+1)$.

By arithmetic-geometric inequality, it follows that
\[P \leq 2^n \left(\frac{\sum_{j \in S_u} (s_j+1)}{n-2}\right)^{n-2}
\leq 2^n \left(\frac{\sum_{j=1}^{n-1} (s_j+1)}{n-1}\right)^{n-2} = 2^n \left(\frac{k+n-1}{n-1}\right)^{n-2},\]
which completes the proof.
\end{proof}

In order to compute a lower bound, we also prove a combinatorial result which is interesting in itself.

\begin{pro}
Let $(s_0,s_1,\ldots,s_n)$ be non-negative integers, with $s_0 = s_n = 0$.
There exists integers $a_1,\ldots,a_n$ such that
$(s_0,a_1,s_1,\ldots,a_n,s_n)$ are actual coordinates.
\end{pro}

\begin{proof}
Let us choose $a_i = s_{i-1}$ if $s_{i-1} \leq s_i$, and $a_i = s_i+1$ if $s_{i-1} > s_i$,
and let $\calD$ be the generalised curve diagram associated to the coordinates $\bs\ba := (s_0,a_1,s_1,\ldots,a_n,s_n)$.
We show below that $\calD$ is a $1$-generalised curve diagram.
We do so by proving, using an induction on $i \in \{0,\ldots,n\}$, the following properties $\calP_i$ and $\calQ_i$:
\begin{eqnarray*}
\calP_i & = & \forall j \in \{1,\ldots,s_i\}, \conc{i}{j} \equiv \conc{i}{2s_i+1-j}; \\
\calQ_i & = & \forall \ell \in \{1,\ldots,2s_{i-1}+1\}, \exists m \in \{1,\ldots,2s_i+1\}
\text{ such that } \conc{i-1}{\ell} \equiv \conc{i}{m}.
\end{eqnarray*}

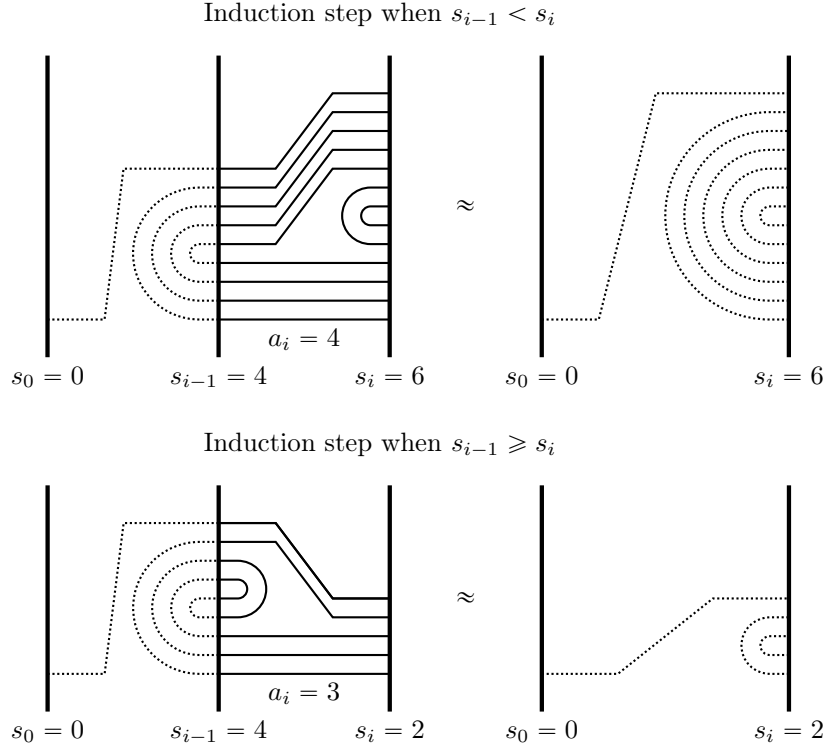
\begin{figure}[!ht]
\begin{center}
\begin{tikzpicture}[scale=0.25]
\draw[draw=black,ultra thick] (5,-2) -- (5,14);
\draw[draw=black,ultra thick] (14,-2) -- (14,14);
\draw[draw=black,ultra thick] (23,-2) -- (23,14);

\draw[draw=black,densely dotted,thick] (5,0) -- (8,0) -- (9,8) -- (14,8);
\draw[draw=black,thick] (14,8) -- (17,8) -- (20,12) -- (23,12);
\draw[draw=black,densely dotted,thick] (14,0) -- (13,0) arc (270:90:3.5) -- (14,7);
\draw[draw=black,densely dotted,thick] (14,1) -- (13,1) arc (270:90:2.5) -- (14,6);
\draw[draw=black,densely dotted,thick] (14,2) -- (13,2) arc (270:90:1.5) -- (14,5);
\draw[draw=black,densely dotted,thick] (14,3) -- (13,3) arc (270:90:0.5) -- (14,4);

\draw[draw=black,thick] (14,0) -- (23,0);
\draw[draw=black,thick] (14,1) -- (23,1);
\draw[draw=black,thick] (14,2) -- (23,2);
\draw[draw=black,thick] (14,3) -- (23,3);
\draw[draw=black,thick] (23,4) -- (22,4) arc (270:90:1.5) -- (23,7);
\draw[draw=black,thick] (23,5) -- (22,5) arc (270:90:0.5) -- (23,6);
\draw[draw=black,thick] (14,4) -- (17,4) -- (20,8) -- (23,8);
\draw[draw=black,thick] (14,5) -- (17,5) -- (20,9) -- (23,9);
\draw[draw=black,thick] (14,6) -- (17,6) -- (20,10) -- (23,10);
\draw[draw=black,thick] (14,7) -- (17,7) -- (20,11) -- (23,11);

\node at (27,6) {$\approx$};

\draw[draw=black,ultra thick] (31,-2) -- (31,14);
\draw[draw=black,ultra thick] (44,-2) -- (44,14);

\draw[draw=black,densely dotted,thick] (31,0) -- (34,0) -- (37,12) -- (44,12);
\draw[draw=black,densely dotted,thick] (44,0) -- (43,0) arc (270:90:5.5) -- (44,11);
\draw[draw=black,densely dotted,thick] (44,1) -- (43,1) arc (270:90:4.5) -- (44,10);
\draw[draw=black,densely dotted,thick] (44,2) -- (43,2) arc (270:90:3.5) -- (44,9);
\draw[draw=black,densely dotted,thick] (44,3) -- (43,3) arc (270:90:2.5) -- (44,8);
\draw[draw=black,densely dotted,thick] (44,4) -- (43,4) arc (270:90:1.5) -- (44,7);
\draw[draw=black,densely dotted,thick] (44,5) -- (43,5) arc (270:90:0.5) -- (44,6);

\node[anchor=north] at (5,-2) {$s_0 = 0$};
\node[anchor=north] at (14,-2) {$s_{i-1} = 4$};
\node[anchor=north] at (18.5,0) {$a_i = 4$};
\node[anchor=north] at (23,-2) {$s_i = 6$};
\node[anchor=north] at (31,-2) {$s_0 = 0$};
\node[anchor=north] at (44,-2) {$s_i = 6$};
\node[anchor=south] at (22.5,15) {Induction step when $s_{i-1} < s_i$};
\end{tikzpicture}

\begin{tikzpicture}[scale=0.25]
\draw[draw=black,ultra thick] (46,-2) -- (46,10);
\draw[draw=black,ultra thick] (55,-2) -- (55,10);
\draw[draw=black,ultra thick] (64,-2) -- (64,10);

\draw[draw=black,densely dotted,thick] (46,0) -- (49,0) -- (50,8) -- (55,8);
\draw[draw=black,thick] (55,8) -- (58,8) -- (61,4) -- (64,4);
\draw[draw=black,densely dotted,thick] (55,0) -- (54,0) arc (270:90:3.5) -- (55,7);
\draw[draw=black,densely dotted,thick] (55,1) -- (54,1) arc (270:90:2.5) -- (55,6);
\draw[draw=black,densely dotted,thick] (55,2) -- (54,2) arc (270:90:1.5) -- (55,5);
\draw[draw=black,densely dotted,thick] (55,3) -- (54,3) arc (270:90:0.5) -- (55,4);

\draw[draw=black,thick] (55,0) -- (64,0);
\draw[draw=black,thick] (55,1) -- (64,1);
\draw[draw=black,thick] (55,2) -- (64,2);
\draw[draw=black,thick] (55,3) -- (56,3) arc (-90:90:1.5) -- (56,6) -- (55,6);
\draw[draw=black,thick] (55,4) -- (56,4) arc (-90:90:0.5) -- (56,5) -- (55,5);
\draw[draw=black,thick] (55,7) -- (58,7) -- (61,3) -- (64,3);
\draw[draw=black,thick] (55,8) -- (58,8) -- (61,4) -- (64,4);

\node at (68,4) {$\approx$};

\draw[draw=black,ultra thick] (72,-2) -- (72,10);
\draw[draw=black,ultra thick] (85,-2) -- (85,10);

\draw[draw=black,densely dotted,thick] (72,0) -- (76,0) -- (81,4) -- (85,4);
\draw[draw=black,densely dotted,thick] (85,0) -- (84,0) arc (270:90:1.5) -- (85,3);
\draw[draw=black,densely dotted,thick] (85,1) -- (84,1) arc (270:90:0.5) -- (85,2);

\node[anchor=north] at (46,-2) {$s_0 = 0$};
\node[anchor=north] at (55,-2) {$s_{i-1} = 4$};
\node[anchor=north] at (59.5,0) {$a_i = 3$};
\node[anchor=north] at (64,-2) {$s_i = 2$};
\node[anchor=north] at (72,-2) {$s_0 = 0$};
\node[anchor=north] at (85,-2) {$s_i = 2$};
\node[anchor=south] at (63.5,11) {Induction step when $s_{i-1} \geq s_i$};
\node at (65.5,14) {};
\end{tikzpicture}
\end{center}
\caption{Constructing actual coordinates}
\label{fig:getting-actual-coordinates}
\end{figure}

First, $\calP_0$ and $\calQ_0$ are vacuously true.
Now, let $i \in \{1,\ldots,n\}$ be some integer such that $\calP_{i-1}$ and $\calQ_{i-1}$ are true,
and let us prove $\calP_i$ and $\calQ_i$.

If $s_{i-1} \leq s_i$, then it follows from $\calP_i$ and Lemma~\ref{lem:left-right-box} that
\begin{itemize}
\item $\conc{i}{j} \sim \conc{i-1}{j} \equiv \conc{i-1}{2s_{i-1}+1-j} \sim \conc{i}{2s_i+1-j}$ whenever $1 \leq j \leq s_{i-1}$;
\item $\conc{i}{s_i+1-j} \sim \conc{i}{s_i+j}$ whenever $1 \leq j \leq s_i-s_{i-1}$;
\item $\conc{i-1}{2s_{i-1}+1} \sim \conc{i}{2s_i+1}$,
\end{itemize}
which proves $\calP_i$ and $\calQ_i$.

If $s_{i-1} > s_i$, then it follows from $\calP_i$ and Lemma~\ref{lem:left-right-box} that
\begin{itemize}
 \item $\conc{i}{j} \sim \conc{i-1}{j} \equiv \conc{i-1}{2s_{i-1}+1-j} \sim \conc{i}{2s_i+1-j}$ whenever $1 \leq j \leq s_i-1$;
 \item $\conc{i-1}{2s_{i-1}+1} \sim \conc{i}{2s_i+1}$,
\end{itemize}
which already proves $\calP_i$ in the case $j \neq s_i$ and $\calQ_i$ in the case $m \notin \{s_i,\ldots,2s_{i-1}-s_i\}$.

Moreover, observe that $\conc{i-1}{j} \equiv \conc{i-1}{2s_{i-1}+1-j} \sim \conc{i-1}{j+2}$ whenever $s_i \leq j \leq 2s_{i-1}-1-s_i$.
An immediate induction on $j$ then shows that $\conc{i-1}{s_i} \equiv \conc{i-1}{j}$ for all $j \in \{s_i,s_i+2,\ldots,2s_{i-1}-s_i\}$ and that
$\conc{i-1}{s_i+1} \equiv \conc{i-1}{j}$ for all $j \in \{s_i+1,s_i+3,\ldots,2s_{i-1}+1-s_i\}$.
Therefore, it follows that $\conc{i}{s_i} \sim \conc{i-1}{s_i} \equiv \conc{i-1}{2s_{i-1}+1-s_i} \equiv \conc{i-1}{s_i+1} \sim \conc{i}{s_i+1}$ and,
incidentally, that $\conc{i-1}{s_i} \equiv (\conc{i-1}{s_i} \text{ and } \conc{i-1}{s_i+1}) \equiv \conc{i-1}{j}$ whenever $s_i \leq j \leq 2s_{i-1}-s_i$.
These two remarks respectively complete the case $j = s_i$ of $\calP_i$ and the case $s_i \leq m \leq 2s_{i-1}-s_i$ of $\calQ_i$,
which proves that both properties $\calP_i$ and $\calQ_i$ must hold.

Overall, we have proved that $\calQ_i$ holds for all $i \in \{0,\ldots,n\}$, which proves, using an immediate induction on $i$,
that $\conc{i}{j} \equiv \conc{n}{1}$ for all $i \in \{0,\ldots,n\}$ and for all $j \in \{1,\ldots,2s_j+1\}$.
This means that $\equiv$ has one unique equivalence class, i.e. that $\calD$ is a $1$-generalised curve diagram, that is, that $\bs\ba$ are actual coordinates.
\end{proof}

\begin{cor}
Consider integers $n \geq 1$ and $k \geq 0$.
Then, $g_{n,k} \geq \binom{k+n-2}{n-2}$.
\label{cor:many-people-v2}
\end{cor}

\subsection{Experimental Data and Conjectures}
\label{subsection:experiments}

Proposition~\ref{pro:not-too-many-people-v2} and Corollary~\ref{cor:many-people-v2} prove that
\[\binom{k+n-2}{n-2} \leq g_{n,k} \leq 2^n \left(\frac{k+n-1}{n-1}\right)^{n-2} \binom{k+n-2}{n-2}.\]
Unfortunately, these lower and upper bounds do not match,
since their ratio is equal to $2^n \left(\frac{k+n-1}{n-1}\right)^{n-2}$,
hence grows arbitrarily when $n$ and $k$ grow.
Therefore, aiming to identify
simple asymptotic estimations of $g_{n,k}$
when $n$ is fixed and $k \to +\infty$,
we look for experimental data.

\begin{figure}[!ht]
\begin{center}
\begin{tikzpicture}
\begin{axis}[axis x line=center,axis y line=center,xlabel style={right},ylabel style={above right},width=0.55\textwidth,
xlabel={$k$},ylabel={$g_{4,k}/k^4$~~---~~{\color{blackgray}$g_{4,k}/(k+4)^4$}~~~~$(\times 100)$}]
\addplot[only marks,color=black,mark size=0.8pt] coordinates{
(17,7.88305) (18,7.38641) (19,7.73118) (20,6.97853) (21,7.32367) (22,6.85181) (23,7.09205) (24,6.60425)
(25,6.92663) (26,6.42468) (27,6.72763) (28,6.34506) (29,6.52765) (30,6.20046) (31,6.5144) (32,6.04214)
(33,6.33226) (34,6.05864) (35,6.20278) (36,5.90613) (37,6.18773) (38,5.80158) (39,6.04466) (40,5.82665)
(41,5.95157) (42,5.71139) (43,5.98382) (44,5.61774) (45,5.83189) (46,5.68053) (47,5.78291) (48,5.55269)
(49,5.80764) (50,5.49556) (51,5.69493) (52,5.55299) (53,5.63112) (54,5.43873) (55,5.68631) (56,5.38194)
(57,5.55797) (58,5.4593) (59,5.52404) (60,5.34093) (61,5.5848) (62,5.30215) (63,5.46611) (64,5.38093)
(65,5.42241) (66,5.27175) (67,5.49995) (68,5.22582) (69,5.37664) (70,5.31654) (71,5.3538) (72,5.20282)
(73,5.42286) (74,5.17215) (75,5.31584) (76,5.26305) (77,5.28647) (78,5.15304) (79,5.36857) (80,5.12008)
(81,5.25436) (82,5.22097) (83,5.23169) (84,5.10218) (85,5.31448) (86,5.08468) (87,5.20393) (88,5.17661)
(89,5.18379) (90,5.06568) (91,5.27369) (92,5.04353) (93,5.15704) (94,5.14451) (95,5.14235) (96,5.03174)
(97,5.23217) (98,5.01028) (99,5.12033) (100,5.11094) (101,5.10504) (102,5.00382) (103,5.19793) (104,4.98051)
(105,5.08376) (106,5.0885) (107,5.07321) (108,4.97213) (109,5.16673) (110,4.95568) (111,5.05774) (112,5.063)
(113,5.04167) (114,4.95055) (115,5.14123) (116,4.93516) (117,5.02718) (118,5.04061) (119,5.01763) (120,4.92742)
(121,5.11536) (122,4.91624) (123,5.0027) (124,5.01977) (125,4.99068) (126,4.91006) (127,5.09421) (128,4.89306)
(129,4.98058) (130,5.00448) (131,4.97367) (132,4.89043) (133,5.07084) (134,4.87848) (135,4.95939) (136,4.98896)
(137,4.95201) (138,4.87373) (139,5.05482) (140,4.86237) (141,4.94324) (142,4.97513) (143,4.93314) (144,4.85973)
(145,5.03608) (146,4.84996) (147,4.92537) (148,4.95846) (149,4.91683) (150,4.84605) (151,5.0242) (152,4.83672)
(153,4.90669) (154,4.94745) (155,4.90182) (156,4.83361) (157,5.0071) (158,4.8236) (159,4.89354) (160,4.93436)
(161,4.88796) (162,4.82334) (163,4.99222) (164,4.81153) (165,4.8793) (166,4.92554) (167,4.87512) (168,4.80975)
(169,4.98156) (170,4.80256) (171,4.86922) (172,4.91549) (173,4.86106) (174,4.80133) (175,4.96966) (176,4.79335)
};
\addplot[only marks,color=blackgray,mark size=0.8pt] coordinates{
(17,4.115) (18,3.987) (19,4.301) (20,3.99) (21,4.293) (22,4.109) (23,4.343) (24,4.123) (25,4.402) (26,4.151)
(27,4.414) (28,4.223) (29,4.403) (30,4.235) (31,4.502) (32,4.222) (33,4.471) (34,4.32) (35,4.464) (36,4.288)
(37,4.53) (38,4.281) (39,4.494) (40,4.363) (41,4.487) (42,4.334) (43,4.569) (44,4.315) (45,4.505) (46,4.412)
(47,4.515) (48,4.357) (49,4.579) (50,4.353) (51,4.531) (52,4.437) (53,4.518) (54,4.381) (55,4.598) (56,4.368)
(57,4.527) (58,4.462) (59,4.53) (60,4.394) (61,4.609) (62,4.389) (63,4.538) (64,4.48) (65,4.527) (66,4.413)
(67,4.616) (68,4.397) (69,4.535) (70,4.495) (71,4.537) (72,4.419) (73,4.616) (74,4.412) (75,4.544) (76,4.508)
(77,4.537) (78,4.431) (79,4.625) (80,4.419) (81,4.543) (82,4.522) (83,4.539) (84,4.434) (85,4.626) (86,4.433)
(87,4.544) (88,4.527) (89,4.54) (90,4.443) (91,4.632) (92,4.436) (93,4.542) (94,4.537) (95,4.541) (96,4.449)
(97,4.632) (98,4.441) (99,4.544) (100,4.541) (101,4.541) (102,4.456) (103,4.634) (104,4.445) (105,4.542)
(106,4.551) (107,4.542) (108,4.456) (109,4.635) (110,4.45) (111, 4.546) (112,4.555) (113,4.54) (114,4.462)
(115,4.638) (116,4.456) (117,4.543) (118,4.559) (119,4.542) (120,4.464) (121,4.638) (122,4.461) (123,4.543)
(124,4.562) (125,4.539) (126,4.469) (127,4.64) (128, 4.46) (129,4.543) (130,4.568) (131,4.543) (132,4.47)
(133, 4.638) (134,4.465) (135,4.542) (136,4.572) (137,4.541) (138,4.472) (139,4.641) (140,4.467) (141,4.544)
(142,4.576) (143,4.54) (144,4.475) (145,4.64) (146,4.471) (147,4.543) (148,4.576) (149,4.54) (150,4.477)
(151, 4.644) (152,4.473) (153,4.54) (154,4.58) (155,4.54) (156,4.479) (157,4.642) (158,4.474) (159,4.541)
(160,4.581) (161,4.54) (162,4.482) (163,4.641) (164,4.475) (165, 4.54) (166,4.585) (167,4.54) (168,4.481)
(169,4.643) (170, 4.478) (171,4.542) (172,4.587) (173,4.538) (174,4.484) (175,4.643) (176,4.48)
};
\end{axis}
\end{tikzpicture}

\medskip

\begin{tikzpicture}
\begin{axis}[axis x line=center,axis y line=center,xlabel style={right},ylabel style={above right},width=0.55\textwidth,
xlabel={$k$},ylabel={$g_{5,k}/k^6$~~---~~{\color{blackgray}$g_{5,k}/(k+5)^6$}~~~~$(\times 1000)$}]
\addplot[only marks,color=black,mark size=0.8pt] coordinates{
(16,4.14276) (17,4.0862) (18,3.77019) (19,3.79798) (20,3.49957) (21,3.544) (22,3.2777) (23,3.36555)
(24,3.08647) (25,3.20386) (26,2.94511) (27,3.06967) (28,2.81866) (29,2.96149) (30,2.70394) (31,2.86665)
(32,2.6152) (33,2.77737) (34,2.53574) (35,2.71046) (36,2.45747) (37,2.64375) (38,2.40093) (39,2.58334)
(40,2.342) (41,2.53478) (42,2.28701) (43,2.4896) (44,2.2451) (45,2.44196) (46,2.20333) (47,2.40946)
(48,2.16125) (49,2.37316) (50,2.13117) (51,2.33762) (52,2.09812) (53,2.31177) (54,2.06657) (55,2.28574)
(56,2.04227) (57,2.25632) (58,2.01725) (59,2.23677) (60,1.98988) (61,2.21336)
};
\addplot[only marks,color=blackgray,mark size=0.8pt] coordinates{
(16,1.086) (17,1.15) (18,1.131) (19,1.207) (20,1.172) (21,1.245) (22,1.203) (23,1.286) (24,1.224) (25,1.315)
(26,1.248) (27,1.34) (28,1.265) (29,1.364) (30,1.276) (31,1.384) (32,1.29) (33,1.398) (34,1.301) (35,1.416)
(36,1.306) (37,1.428) (38,1.317) (39,1.438) (40,1.322) (41,1.45) (42,1.325) (43,1.46) (44,1.332) (45,1.465)
(46,1.336) (47,1.476) (48,1.337) (49,1.482) (50,1.343) (51,1.486) (52,1.345) (53,1.494) (54,1.346) (55,1.5)
(56,1.35) (57,1.502) (58,1.352) (59,1.509) (60,1.351) (61,1.512)
};
\end{axis}
\end{tikzpicture}
\end{center}
\caption{Estimating $g_{n,k}$ --- experimental data for $n = 4$ and $n = 5$}
\label{fig:gnk-decrease}
\end{figure}
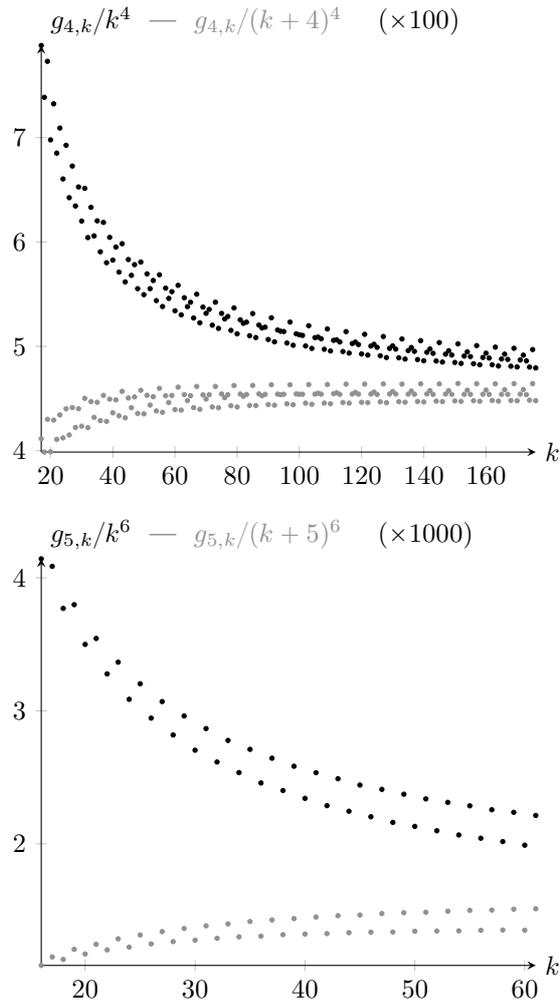

Figure~\ref{fig:gnk-decrease} presents the ratios $g_{n,k}/k^{2(n-2)}$ (in black) and
$g_{n,k}/(k+n)^{2(n-2)}$ (in gray).
We computed $g_{n,k}$ by enumerating all the virtual coordinates,
then checking individually which of them were actual coordinates
(up to refinements such as using the above-mentioned
symmetries to reduce the number of cases to look at).

The two series of points suggest the following conjecture,
which was already proven to be true when $n = 2$ and $n = 3$.

\begin{conj}
Let $n \geq 2$ be some integer. There exists two positive constants $\alpha_n$ and $\beta_n$ such that
$\alpha_n k^{2(n-2)} \leq g_{n,k} \leq \beta_n k^{2(n-2)}$ for all integers $k \geq 1$.
\label{conj:dn2k=k2n-4}
\end{conj}

Figure~\ref{fig:gnk-decrease} also suggests that the ratios $g_{n,k}/k^{2(n-2)}$
might be split into convergent clusters, according to the value of
$k~\mathrm{mod}~6$ (when $n = 4$) or $k~\mathrm{mod}~2$ (when $n = 5$).
Once again, this is coherent with the patterns noticed for $n = 2$ and $n = 3$,
and therefore suggests a stronger conjecture.

\begin{conj}
Let $n \geq 2$ be some integer. There exists some positive integer $\rho_n$ such that,
for every integer $\ell \in \{0,1,\ldots,\rho_n-1\}$, the sequence of ratios
$\frac{g_{n,k\rho_n+\ell}}{k^{2(n-2)}}$ has a positive limit $\lambda_{n,\ell}$ when $k \to +\infty$.
\label{conj:dn2k=clusters}
\end{conj}

Assuming Conjecture~5.11,
a natural further step would be to compute the limits $\lambda_{n,\ell}$ or to
study more precisely the asymptotic behaviour of the ratios $g_{n,k}/k^{2(n-2)}$.
In particular, we hope that computing arbitrarily precise
approximations of the constants $\lambda_{n,\ell}$
for small values of $n$ might help us guess analytic values of $\lambda_{n,\ell}$,
thereby providing insight about the underlying combinatorial or
number-theory-related structure of the integers $g_{n,k}$.

\section*{Acknowledgments}
The author is very thankful to an anonymous referee for his (her) insightful remarks and suggestions,
which helped both simplifying several proofs and improving the overall readability of the article.


\end{document}